\documentclass[12pt,a4paper,reqno]{amsart}
\usepackage{amssymb}
\usepackage{amscd}
\usepackage{enumerate}
\usepackage{graphicx}
\usepackage{siunitx}
\usepackage{tikz-cd}
\usepackage{color}
\usetikzlibrary{arrows}
\numberwithin{equation}{section}

\usepackage{mathabx}

\usepackage{mathtools}
\usepackage[tableposition=top]{caption}
\usepackage{booktabs,dcolumn}

\DeclareFontFamily{OT1}{rsfs}{}
\DeclareFontShape{OT1}{rsfs}{n}{it}{<-> rsfs10}{}
\DeclareMathAlphabet{\mathscr}{OT1}{rsfs}{n}{it}

\theoremstyle{plain}

\newtheorem{theorem}{Theorem}[section]
\newtheorem{proposition}[theorem]{Proposition}
\newtheorem{lemma}[theorem]{Lemma}
\newtheorem{corollary}[theorem]{Corollary}
\newtheorem{conjecture}[theorem]{Conjecture}

\theoremstyle{definition}

\newtheorem{definition}[theorem]{Definition}
\newtheorem{remark}[theorem]{Remark}

\newcommand\E{\mathbb{E}}
\newcommand\R{\mathbb{R}}
\newcommand\Z{\mathbb{Z}}

\newcommand\C{\mathbb{C}}

\newcommand\eps{\varepsilon}

\newcommand\SO{{\operatorname{SO}}}
\newcommand\Rot{{\operatorname{Rot}}}
\newcommand\Dil{{\operatorname{Dil}}}
\newcommand\df{{\operatorname{df}}}

\usepackage{algorithm, algorithmic}

\begin{document}

\title[Blowup for averaged Navier-Stokes]{Finite time blowup for an averaged three-dimensional Navier-Stokes equation}

\author{Terence Tao}
\address{UCLA Department of Mathematics, Los Angeles, CA 90095-1555.}
\email{tao@math.ucla.edu}



\begin{abstract}  The Navier-Stokes equation on the Euclidean space $\R^3$ can be expressed in the form $\partial_t u = \Delta u + B(u,u)$, where $B$ is a certain bilinear operator on divergence-free vector fields $u$ obeying the cancellation property $\langle B(u,u), u\rangle=0$ (which is equivalent to the energy identity for the Navier-Stokes equation).  In this paper, we consider a modification $\partial_t u = \Delta u + \tilde B(u,u)$ of this equation, where $\tilde B$ is an averaged version of the bilinear operator $B$ (where the average involves rotations, dilations and Fourier multipliers of order zero), and which also obeys the cancellation condition $\langle \tilde B(u,u), u \rangle = 0$ (so that it obeys the usual energy identity).  By analysing a system of ODE related to (but more complicated than) a dyadic Navier-Stokes model of Katz and Pavlovic, we construct an example of a smooth solution to such an averaged Navier-Stokes equation  which blows up in finite time.  This demonstrates that any attempt to positively resolve the Navier-Stokes global regularity problem in three dimensions has to use finer structure on the nonlinear portion $B(u,u)$ of the equation than is provided by harmonic analysis estimates and the energy identity.  We also propose a program for adapting these blowup results to the true Navier-Stokes equations.
\end{abstract}

\maketitle


\section{Introduction}

\subsection{Statement of main result}

The purpose of this paper is to formalise the ``supercriticality'' barrier for the (infamous) global regularity problem for the Navier-Stokes equation, using a blowup solution to a certain averaged version of Navier-Stokes equation to demonstrate that any proposed positive solution to the regularity problem which does not use the finer structure of the nonlinearity cannot possibly be successful.  This barrier also suggests a possible route to provide a negative answer to this problem, that is to say it suggests a program for constructing a blowup solution to the true Navier-Stokes equations.

The barrier is not particularly sensitive to the precise formulation\footnote{See \cite{tao-local} for an analysis of the relationship between different formulations of the Navier-Stokes regularity problem in three dimensions.  It is likely that our main results also extend to higher dimensions than three, although we will not pursue this matter here.} of the regularity problem, but to state the results in the cleanest fashion we will take the homogeneous global regularity problem in the Euclidean setting in three spatial dimensions as our formulation:

\begin{conjecture}[Navier-Stokes global regularity]\label{nsconj} \ \cite[(A)]{fefferman} Let $\nu > 0$, and let $u_0: \R^3 \to \R^3$ be a divergence-free vector field in the Schwartz class.  Then there exist a smooth vector field $u: [0,+\infty) \times \R^3 \to \R^3$ (the velocity field) and smooth function $p: \R^3 \to \R$ (the pressure field) obeying the equations
\begin{equation}\label{ns}
\begin{split}
\partial_t u + (u \cdot \nabla) u &= \nu \Delta u - \nabla p \\
\nabla \cdot u &= 0 \\
u(0,\cdot) &= u_0
\end{split}
\end{equation}
as well as the finite energy condition $u \in L^\infty_t L^2_x([0,T] \times \R^3)$ for every $0 < T < \infty$.
\end{conjecture}

By applying the rescaling $\tilde u(t,x) := \nu u( \nu t, x )$, $\tilde p(t,x) := \nu p( \nu t, \nu x)$ we may normalise $\nu=1$ (note that there is no smallness requirement on the initial data $u_0$), and we shall do so henceforth.  

To study this conjecture, we perform some standard computations to eliminate the role of the pressure $p$, and to pass from the category of smooth (classical) solutions to the closely related category of \emph{mild} solutions in a high regularity class.  It will not matter too much what regularity class we take here, as long as it is subcritical, but for sake of concreteness (and to avoid some very minor technicalities) we will take a quite high regularity space, namely the Sobolev space $H^{10}_\df(\R^3)$ of (distributional) vector fields $u: \R^3 \to \R^3$ with $H^{10}$ regularity (thus the weak derivatives $\nabla^j u$ are square-integrable for $j=0,\dots,10$) and which are divergence free in the distributional sense: $\nabla \cdot u = 0$.  By using the $L^2$ inner product\footnote{We will not use the $H^{10}_\df$ inner product in this paper, thus all appearances of the $\langle,\rangle$ notation should be interpreted in the $L^2$ sense.}
$$ \langle u, v \rangle := \int_{\R^3} u \cdot v\ dx$$
on vector fields $u, v: \R^3\to \R^3$, the dual $H^{10}_\df(\R^3)^*$ may be identified with the negative-order Sobolev space $H^{-10}_\df(\R^3)$ of divergence-free distributions $u: \R^3 \to \R^3$ of $H^{-10}$ regularity.  We introduce the \emph{Euler bilinear operator} $B: H^{10}_\df(\R^3) \times H^{10}_\df(\R^3) \to H^{10}_\df(\R^3)^*$ via duality as
$$ \langle B(u,v), w \rangle := 
-\frac{1}{2} \int_{\R^3} \left(\left(\left(u \cdot \nabla\right) v\right) \cdot w\right) + \left(\left(\left(v \cdot \nabla\right) u\right) \cdot w\right)\ dx$$
for $u,v,w \in H^{10}_\df(\R^3)$; it is easy to see from Sobolev embedding that this operator is well defined.  More directly, we can write
$$ B(u,v) = -\frac{1}{2} P \left[ (u \cdot \nabla) v + (v \cdot \nabla) u \right]$$
where $P$ is the \emph{Leray projection} onto divergence-free vector fields, defined on square-integrable $u: \R^3 \to \R^3$ by the formula
$$ Pu_i := u_i - \Delta^{-1} \partial_i \partial_j u_j$$
with the usual summation conventions, where $\Delta^{-1} \partial_i \partial_j$ is defined as the Fourier multiplier with symbol $\frac{\xi_i \xi_j}{|\xi|^2}$.  Note that $B(u,v)$ takes values in $L^2(\R^3)$ (and not just in $H^{10}_\df(\R^3)^*$) when $u,v \in H^{10}_\df(\R^3)$.  We refer to the form $(u,v,w) \mapsto\langle B(u,v),w \rangle$ as the \emph{Euler trilinear form}.  As is well known, we have the important cancellation law
\begin{equation}\label{cancellation}
\langle B(u,u), u \rangle = 0
\end{equation}
for all $u \in H^{10}_\df(\R^3)$, as can be seen by a routine integration by parts exploiting the divergence-free nature of $u$, with all manipulations being easily justified due to the high regularity of $u$.  It will also be convenient to express the Euler trilinear form in terms of the Fourier transform $\hat u(\xi) := \int_{\R^3} u(x) e^{-2\pi i x \cdot \xi}\ dx$ as
\begin{equation}\label{bwing}
\langle B(u,v), w \rangle = -\pi i \int_{\xi_1+\xi_2+\xi_3=0} \Lambda_{\xi_1,\xi_2,\xi_3}( \hat u(\xi_1), \hat v(\xi_2), \hat w(\xi_3) )
\end{equation}
for all $u,v,w \in H^{10}_\df(\R^3)$, 
where we adopt the shorthand
$$\int_{\xi_1+\xi_2+\xi_3=0} F(\xi_1,\xi_2,\xi_3) := \int_{\R^3} \int_{\R^3} F(\xi_1,\xi_2,-\xi_1-\xi_2)\ d\xi_1 d\xi_2$$
and $\Lambda_{\xi_1,\xi_2,\xi_3}: \xi_1^\perp \times \xi_2^\perp \times \xi_3^\perp \to \R$ is the trilinear form
\begin{equation}\label{lambda-def}
 \Lambda_{\xi_1,\xi_2,\xi_3}( X_1, X_2, X_3 ) := (X_1 \cdot \xi_2)(X_2 \cdot X_3) + (X_2 \cdot \xi_1) (X_1 \cdot X_3),
\end{equation}
defined for vectors $X_i$ in the orthogonal complement $\xi_i^\perp := \{ X_i \in \R^3: X_i \cdot \xi_i = 0 \}$ of $\xi_i$ for $i=1,2,3$; note the divergence-free condition ensures that $\hat u(\xi_1) \in \xi_1^\perp$ for (almost) all $\xi_1 \in \R^3$, and similarly for $v$ and $w$.  This also provides an alternate way to establish \eqref{cancellation}.  

Given a Schwartz divergence-free vector field $u_0: \R^3 \to \R^3$ and a time interval $I \subset [0,+\infty)$ containing $0$, we define a \emph{mild $H^{10}$ solution to the Navier-Stokes equations} (or \emph{mild solution} for short) with initial data $u_0$ to be a continuous map $u: I \to H^{10}_\df(\R^3)$ obeying the integral equation
\begin{equation}\label{integral}
 u(t) = e^{t\Delta} u_0 + \int_0^t e^{(t-t')\Delta} B(u(t'),u(t'))\ dt'
 \end{equation}
for all $t \in I$, where $e^{t\Delta}$ are the usual heat propagators (defined on $L^2(\R^3)$, for instance); formally,\eqref{integral} implies the projected Navier-Stokes equation
\begin{equation}\label{ns-smooth}
\begin{split}
 \partial_t u &= \Delta u + B(u,u)\\
 u(0,\cdot) &= u_0
 \end{split}
 \end{equation}
in a distributional sense at least (actually, at the $H^{10}_\df$ level of regularity it is not difficult to justify \eqref{ns-smooth} in the classical sense for mild solutions).  

The distinction between smooth finite energy solutions and $H^{10}_\df$ mild solutions is essentially non-existent (at least\footnote{For data which is only in $H^{10}_\df$, there is a technical distinction between the two solution concepts, due to a lack of unlimited time regularity at the initial time $t=0$ that is ultimately caused by the non-local effects of the divergence-free condition $\nabla \cdot u = 0$, requiring one to replace the notion of a smooth solution with that of an \emph{almost smooth solution}; see \cite{tao-local} for details.  However, in this paper we will only concern ourselves with Schwartz initial data, so that this issue does not arise.} for Schwartz initial data), and the reader may wish to conflate the two notions on a first reading.  More rigorously, we can reformulate Conjecture \ref{nsconj} as the following logically equivalent conjecture:

\begin{conjecture}[Navier-Stokes global regularity, again]\label{reg-again}  Let $u_0: \R^3 \to \R^3$ be a divergence-free vector field in the Schwartz class.  Then there exists a mild solution $u: [0,+\infty) \to H^{10}_\df(\R^3)$ to the Navier-Stokes equations with initial data $u_0$.
\end{conjecture}

\begin{lemma} Conjecture \ref{nsconj} and Conjecture \ref{reg-again} are equivalent.
\end{lemma}

\begin{proof}  We use the results from \cite{tao-local}, although this equivalence is essentially classical and was previously well known to experts.

Let us first show that Conjecture \ref{nsconj} implies Conjecture \ref{reg-again}.  Let $u_0: \R^3 \to\R^3$ be a Schwartz divergence-free vector field, then by Conjecture \ref{nsconj} we may find a smooth vector field $u: [0,+\infty) \times \R^3 \to \R^3$ and smooth function $p: \R^3 \to \R$ obeying the equations \eqref{ns} and the finite energy condition.  By \cite[Corollary 11.1]{tao-local}, $u$ is an $H^1$ solution, that is to say $u \in L^\infty_t H^1_x([0,T] \times \R^3)$ for all finite $T$.  By \cite[Corollary 4.3]{tao-local}, we then have the integral equation \eqref{integral}, and by \cite[Theorem 5.4(ii)]{tao-local}, $u \in L^\infty_t H^k_x([0,T] \times \R^3)$ for every $k$, which easily implies (from \eqref{integral}) that $u$ is a continuous map from $[0,+\infty)$ to $H^{10}_\df(\R^3)$.  This gives Conjecture \ref{reg-again}.

Conversely, if Conjecture \ref{reg-again} holds, and $u_0: \R^3 \to \R^3$ is a Schwartz class solution, we may find a mild solution $u: [0,+\infty) \to H^{10}_\df(\R^3)$ with this initial data.  By \cite[Theorem 5.4(ii)]{tao-local}, $u\in L^\infty_t H^k_x([0,T] \times \R^3)$ for every $k$.  If we define the normalised pressure
$$ p := -\Delta^{-1} \partial_i \partial_j(u_i u_j)$$
then by \cite[Theorem 5.4(iv)]{tao-local}, $u$ and $p$ are smooth on $[0,+\infty) \times \R^3$, and for each $j, k \geq 0$, the functions $\partial_t^j u, \partial_t^j p$ lie in $L^\infty_t H^k_x([0,T] \times \R^3)$ for all finite $T$.  By differentiating \eqref{integral}, we have
\begin{align*}
\partial_t u &= \Delta u + B(u,u) \\
&= \Delta u - (u \cdot \nabla) u - \nabla p,
\end{align*}
and Conjecture \ref{nsconj} follows.
\end{proof}

If we take the inner product of \eqref{ns-smooth} with $u$ and integrate in time using \eqref{cancellation}, we arrive at\footnote{One has to justify the integration by parts of course, but this is routine under the hypothesis of a mild solution; we omit the (standard) details. } the fundamental \emph{energy identity}
\begin{equation}\label{energy-ident}
\frac{1}{2} \int_{\R^3} |u(T,x)|^2\ dx + \int_0^T \int_{\R^3} |\nabla u(t,x)|^2\ dx dt = \frac{1}{2} \int_{\R^3} |u_0(x)|^2\ dx
\end{equation}
for any mild solution to the Navier-Stokes equation.

If one was unaware of the supercritical nature of the Navier-Stokes equation, one might attempt to obtain a positive solution to Conjecture \ref{nsconj} or Conjecture \ref{reg-again} by combining \eqref{energy-ident} (or equivalently, \eqref{cancellation}) with various harmonic analysis estimates for the inhomogeneous heat equation
\begin{align*}
\partial_t u &= \Delta u + F \\
u(0,\cdot) &= u_0
\end{align*}
(or, in integral form, $u(t) = e^{t\Delta} u_0 + \int_0^t e^{(t-t')\Delta} F(t')\ dt'$), together with harmonic analysis estimates for the Euler bilinear operator $B$, a simple example of which is the estimate
\begin{equation}\label{bouv}
 \| B(u,v) \|_{L^2(\R^3)} \leq C \left( \|\nabla u \|_{L^4(\R^3)} \| v \|_{L^4(\R^3)} + \|\nabla v \|_{L^4(\R^3)} \| u \|_{L^4(\R^3)} \right)
\end{equation}
for some absolute constant $C$.  Such an approach succeeds for instance if the initial data $u_0$ is sufficiently small\footnote{One can of course also consider other perturbative regimes, in which the solution $u$ is expected to be close to some other special solution than the zero solution.  There is a vast literature in these directions, see e.g. \cite{chemin} and the references therein.} in a suitable critical norm (see \cite{koch-tataru} for an essentially optimal result in this direction), or if the dissipative operator $-\Delta$ is replaced by a hyperdissipative operator $(-\Delta)^\alpha$ for some $\alpha \geq 5/4$ (see \cite{katz-hyper}) or with very slightly less hyperdissipative operators (see \cite{tao-hyper}).  Unfortunately, standard scaling heuristics (see e.g. \cite[\S 2.4]{tao-struct}) have long indicated to the experts that the energy estimate \eqref{energy-ident} (or \eqref{cancellation}), together with the harmonic analysis estimates available for the heat equation and for the Euler bilinear operator $B$, are not sufficient by themselves to affirmatively answer Conjecture \ref{nsconj}.  However, these scaling heuristics are not formalised as a rigorous barrier to solvability, and the above mentioned strategy to solve the Navier-Stokes global regularity problem continues to be attempted on occasion.

The most conclusive way to rule out such a strategy would of course be to demonstrate\footnote{It is a classical fact that mild solutions to a given initial data are unique, see e.g. \cite[Theorem 5.4(iii)]{tao-local}.} a mild solution to the Navier-Stokes equation that develops a singularity in finite time, in the sense that the $H^{10}_\df$ norm of $u(t)$ goes to infinity as $t$ approaches a finite time $T_*$.  Needless to say, we are unable to produce such a solution.  However, we will in this paper obtain a finite time blowup (mild) solution to an \emph{averaged} equation
\begin{equation}\label{ns-modified}
\begin{split}
\partial_t u &= \Delta u + \tilde B(u,u)\\
u(0,\cdot) &= u_0,
\end{split}
\end{equation}
where $\tilde B: H^{10}_\df(\R^3) \times H^{10}_\df(\R^3) \to H^{10}_\df(\R^3)^*$ will be a (carefully selected) averaged version of $B$ that has equal or lesser ``strength'' from a harmonic analysis point of view (indeed, $\tilde B$ obeys slightly \emph{more} estimates than $B$ does), and which still obeys the fundamental cancellation property \eqref{cancellation}.  Thus, any successful method to affirmatively answer Conjecture \ref{nsconj} (or Conjecture \ref{reg-again}) must either use finer structure of the Navier-Stokes equation beyond the general form \eqref{ns-smooth}, or else must rely crucially on some estimate or other property of the Euler bilinear operator $B$ that is not shared by the averaged operator $\tilde B$.

We pause to mention some previous blowup results in this direction.  If one drops the cancellation requirement \eqref{cancellation}, so that one no longer has the energy identity \eqref{energy-ident}, then blowup solutions for various Navier-Stokes type equations have been constructed in the literature.  For instance, in \cite{mont} finite time blowup for a ``cheap Navier-Stokes equation'' $\partial_t u = \Delta u + \sqrt{-\Delta}(u^2)$ (with $u$ now a scalar field) was constructed in the one-dimensional setting, with the results extended to higher dimensions in \cite{gallagher}.  As remarked in that latter paper, it is essential to the methods of proof that no energy identity is available.  In a slightly different direction, finite time blowup was established in \cite{sinai} for a complexified version of the Navier-Stokes equations, in which the energy identity was again unavailable (or more precisely, it is available but non-coercive).  These models are not exactly of the type \eqref{ns-modified} considered in this paper, but are certainly very similar in spirit.  

Further models of Navier-Stokes type, which obey an energy identity, were introduced by Plech\'a\c{c} and \c{S}ver\'ak \cite{plechac}, \cite{plechac-2}, by Katz and Pavlovic \cite{katz-dyadic}, and by Hou and Lei \cite{hou}; of these three, the model in \cite{katz-dyadic} is the most relevant for our work and will be discussed in detail in Section \ref{overview-sec} below.  These models differ from each other in several respects, but interestingly, in all three cases there is substantial evidence of blowup in five and higher dimensions, but not in three or four dimensions; indeed, for all three of the models mentioned above there are global regularity results in three dimensions, even in the presence of blowup results for the corresponding inviscid model.  Numerical evidence for blowup for the Navier-Stokes equations is currently rather scant (except in the infinite energy setting, see \cite{grundy}, \cite{naga}); the blowup evidence is much stronger in the case of the Euler equations (see \cite{hlwz} for a recent result in this direction, and \cite{hou-survey} for a survey), but it is as yet unclear\footnote{However, in \cite{hsw}, finite time blowup for a three-dimensional ``partially viscous'' Navier-Stokes type model, in which some but not all of the fields are subject to a viscosity term, was established.} whether these blowup results have direct implications for Navier-Stokes in the three-dimensional setting, due to the relatively significant strength of the dissipation.  

Finally, we mention work \cite{kiselev}, \cite{alibaud}, \cite{dong} establishing finite time blowup for supercritical fractal Burgers equations; such equations are not exactly of Navier-Stokes type, being scalar one-dimensional equations rather than incompressible vector-valued three-dimensional ones, but from a scaling perspective the results are of the same type, namely a demonstration of blowup whenever the norms controlled by the conservation and monotonicity laws are all supercritical.  

We now describe more precisely the type of averaged operator $\tilde B: H^{10}_\df(\R^3) \times H^{10}_\df(\R^3) \to H^{10}_\df(\R^3)^*$ we will consider.  We consider three types of symmetries on $H^{10}_\df(\R^3)$ that we will average over.  Firstly, we have rotation symmetry: if $R \in \SO(3)$ is a rotation matrix on $\R^3$ and $u \in H^{10}_\df(\R^3)$, then the rotated vector field
$$ \Rot_R(u)( x ) := R u( R^{-1} x )$$
is also in $H^{10}_\df(\R^3)$; note that the Fourier transform also rotates by the same law,
$$ \widehat{\Rot_R(u)}(\xi) = R \hat u(R^{-1}\xi).$$
Clearly, these rotation operators are uniformly bounded on $H^{10}_\df(\R^3)$, and also on every Sobolev space $W^{s,p}(\R^3)$ with $s \in \R$ and $1 < p < \infty$.

Next, define a (complex) \emph{Fourier multiplier of order $0$} to be an operator $m(D)$ defined on (the complexification $H^{10}_\df(\R^3) \otimes \C$ of) $H^{10}_\df(\R^3)$ by the formula
$$ \widehat{m(D) u}(\xi) := m(\xi) \hat u(\xi)$$
where $m: \R^3 \to \C$ is a function that is smooth away from the origin, with the seminorms
\begin{equation}\label{seminorm}
\| m \|_k := \sup_{\xi \neq 0} |\xi|^k |\nabla^k m(\xi)|
\end{equation}
being finite for every natural number $k$.  We say that $m(D)$ is \emph{real} if the symbol $m$ obeys the symmetry $m(-\xi) = \overline{m(\xi)}$ for all $\xi \in \R^3 \backslash \{0\}$, then $m(D)$ maps $H^{10}_\df(\R^3)$ to itself.  From the H\"ormander-Mikhlin multiplier theorem (see e.g. \cite{stein-singular}), complex Fourier multipliers of order $0$ are also bounded on (the complexifications of) every Sobolev space $W^{s,p}(\R^3)$ for all $s \in \R$ and $1 < p < \infty$, with an operator norm that depends linearly on finitely many of the $\|m\|_k$.  We let $\mathcal{M}_0$ denote the space of all real Fourier multipliers of order $0$, so that $\mathcal{M}_0 \otimes \C$ is the space of complex Fourier multipliers (note that every complex Fourier multiplier $m(D)$ of order $0$ can be uniquely decomposed as $m(D) = m_1(D) + i m_2(D)$ with $m_1(D),m_2(D)$ real Fourier multipliers of order $0$).  Fourier multipliers of order $0$ do not necessarily commute with the rotation operators $\Rot_R$, but the group of rotation operators normalises the algebra $\mathcal{M}_0$, and hence also the complexification $\mathcal{M}_{0} \otimes \C$.

Finally, we will average\footnote{In an earlier version of this manuscript, no averaging over dilations was assumed, but it was pointed out to us by the referee that the non-degeneracy condition \eqref{c-nondeg} failed if one did not introduce dilation averaging.} over the dilation operators
\begin{equation}\label{dila}
 \Dil_\lambda(u)(x) := \lambda^{3/2} u(\lambda x)
\end{equation}
for $\lambda > 0$.  These operators do not quite preserve the $H^{10}_\df(\R^3)$ norm, but if $\lambda$ is restricted to a compact subset of $(0,+\infty)$ then these operators (and their inverses) will be uniformly bounded on $H^{10}_\df(\R^3)$.

We now define an \emph{averaged Euler bilinear operator} to be an operator $\tilde B: H^{10}_\df(\R^3) \times H^{10}_\df(\R^3) \to H^{10}_\df(\R^3)^*$, defined via duality by the formula
\begin{equation}\label{bogo}
 \langle \tilde B(u,v), w \rangle := \E \left\langle B\left( m_1(D) \Rot_{R_1} \Dil_{\lambda_1} u, m_2(D) \Rot_{R_2} \Dil_{\lambda_2} v\right), m_3(D) \Rot_{R_3} \Dil_{\lambda_3} w\right\rangle
\end{equation}
for all $u,v,w \in H^{10}_\df(\R^3)$, where $m_1(D),m_2(D),m_3(D)$ are random real Fourier multipliers of order $0$, $R_1,R_2,R_3$ are random rotations, and $\lambda_1,\lambda_2,\lambda_3$ are random dilations, obeying the moment bounds
$$ \E \| m_1 \|_{k_1} \|m_2 \|_{k_2} \|m_3\|_{k_3} < \infty$$
and
$$ C^{-1} \leq \lambda_1,\lambda_2,\lambda_3 \leq C$$
almost surely for any natural numbers $k_1,k_2,k_3$ and some finite $C$.  To phrase this definition without probabilistic notation, we have
\begin{equation}\label{bogo-alt}
\langle \tilde B(u,v), w \rangle = \int_\Omega \left\langle B\left( m_{1,\omega}(D) \Rot_{R_{1,\omega}} \Dil_{\lambda_1} u, m_{2,\omega}(D) \Rot_{R_{2,\omega}} \Dil_{\lambda_2} v\right), m_{3,\omega}(D) \Rot_{R_{3,\omega}} \Dil_{\lambda_3} w\right\rangle\ d\mu(\omega)
\end{equation}
for some probability space $(\Omega,\mu)$ and some measurable maps $R_{i,\cdot}: \Omega \to \SO(3)$, $\lambda_{i,\cdot}: \Omega \to (0,+\infty)$ and $m_{i,\cdot}(D): \Omega \to \mathcal{M}_{0}$, where $\mathcal{M}_{0}$ is given the Borel $\sigma$-algebra coming from the seminorms $\| \|_k$, and one has
$$ \int_\Omega \| m_{1,\omega} \|_{k_1} \| m_{2,\omega} \|_{k_2} \| m_{3,\omega} \|_{k_3}\ d\mu(\omega) < \infty$$
and
$$ C^{-1} \leq\lambda_1(\omega), \lambda_2(\omega), \lambda_3(\omega) < C$$
for all natural numbers $k_1,k_2,k_3$.  One can also express $\tilde B(u,v)$ without duality by the formula
$$ \tilde B(u,v) = \int_\Omega \Dil_{\lambda_{3,\omega}^{-1}} \Rot_{R_{3,\omega}^{-1}} \overline{m_{3,\omega}}(D) B\left ( m_{1,\omega}(D) \Rot_{R_{1,\omega}} \Dil_{\lambda_{1,\omega}} u, m_{2,\omega}(D) \Rot_{R_{2,\omega}} \Dil_{\lambda_{2,\omega}} v\right)\ d\mu(\omega)$$
where the integral is interpreted in the weak sense (i.e. the Gelfand-Pettis integral).  However, we will not use this formulation of $\tilde B$ here.

\begin{remark} By the rotation symmetry $\langle B( \Rot_R u, \Rot_R v), \Rot_R w \rangle = \langle B(u,v), w \rangle$, we may eliminate one of the three rotation operators $\Rot_{R_{i,\omega}}$ in \eqref{bogo-alt} if desired, and similarly for the dilation operator.  By some Fourier analysis (related to the fractional Leibniz rule) it should also be possible to eliminate one of the Fourier multipliers $m_{i,\omega}(D)$.  However, we will not attempt to do so here.
\end{remark}

From duality, the triangle inequality (or more precisely, Minkowski's inequality for integrals), and the H\"ormander-Mikhlin multiplier theorem, we see that every estimate on the Euler bilinear operator $B$ in Sobolev spaces $W^{s,p}(\R^3)$ with $1 < p < \infty$ implies a corresponding estimate for averaged Euler bilinear operators $\tilde B$ (but possibly with a larger constant).  For instance, from \eqref{bouv} we have
\begin{equation}\label{bouv-avg}
 \| \tilde B(u,v) \|_{L^2(\R^3)} \leq C_{\tilde B} ( \|\nabla u \|_{L^4(\R^3)} \| v \|_{L^4(\R^3)} + \|\nabla v \|_{L^4(\R^3)} \| u \|_{L^4(\R^3)} ).
\end{equation}
for $u,v \in H^{10}_\df(\R^3)$, where the constant $C_{\tilde B}$ depends\footnote{Note that by applying the transformation $(u,\tilde B) \to (\lambda u, \lambda^{-1} \tilde B)$ to \eqref{ns-modified}, we have the freedom to multiply $\tilde B$ by an arbitrary constant, and so the constants $C_{\tilde B}$ appearing in any given estimate such as \eqref{bouv-avg} can be normalised to any absolute constant (e.g. $1$) if desired.} only on $\tilde B$.  A similar argument shows that the expectation in \eqref{bogo} (or the integral in \eqref{bogo-alt}) is absolutely convergent for any $u,v,w \in H^{10}_\df(\R^3)$.

Similar considerations hold for most other basic bilinear estimates\footnote{There is a possible exception to this principle if the estimate involves endpoint spaces such as $L^1$ and $L^\infty$ for which the H\"ormander-Mikhlin multiplier theorem is not available, or non-convex spaces such as $L^{1,\infty}$ for which the triangle inequality is not available.  However, as the Leray projection $P$ is also badly behaved on these spaces, such endpoint spaces rarely appear in these sorts of analyses of the Navier-Stokes equation.} on $B$ in popular function spaces such as H\"older spaces, Besov spaces, or Morrey spaces.  Because of this, the local theory (and related theory, such as the concentration-compactness theory) for \eqref{ns-modified} is essentially identical to that of \eqref{ns-smooth} (up to changes in the explicit constants), although we will not attempt to formalise this assertion here.  In particular, we may introduce the notion of a \emph{mild solution} to the averaged Navier-Stokes equation \eqref{ns-smooth} with initial data $u_0 \in H^{10}_\df(\R^3)$ on a time interval $I\subset [0,+\infty)$ containing $0$, defined to be a continuous map $u: I \to H^{10}_\df(\R^3)$ obeying the integral equation
\begin{equation}\label{integral-avg}
 u(t) = e^{t\Delta} u_0 + \int_0^t e^{(t-t')\Delta} \tilde B(u(t'),u(t'))\ dt'
 \end{equation}
for all $0 \leq t \leq T$.  It is then a routine matter to extend the $H^{10}$ local existence and uniqueness theory (see e.g. \cite[\S 5]{tao-local}) for mild solutions of the Navier-Stokes equations, to mild solutions of the averaged Navier-Stokes equations, basically because of the previous observation that all the estimates on $B$ used in that local theory continue to hold for $\tilde B$.

Because we have not imposed any symmetry or anti-symmetry hypotheses on the averaging measure $\mu$, rotations $R_j$, and Fourier multipliers $m_j(D)$, the analogue
\begin{equation}\label{cancellation-2}
\langle \tilde B( u, u ), u \rangle = 0
\end{equation}
of the cancellation condition \eqref{cancellation} is not automatically satisfied.  If however we have \eqref{cancellation-2} for all $u \in H^{10}_\df(\R^3)$, then mild solutions to \eqref{ns-modified} enjoy the same energy identity \eqref{energy-ident} as mild solutions to the true Navier-Stokes equation. 

We are now ready to state the main result of the paper.

\begin{theorem}[Finite time blowup for an averaged Navier-Stokes equation]\label{main}  There exists a symmetric averaged Euler bilinear operator $\tilde B: H^{10}_\df(\R^3) \times H^{10}_\df(\R^3) \to H^{10}_\df(\R^3)^*$ obeying the cancellation property \eqref{cancellation-2} for all $u \in H^{10}_\df(\R^3)$, and a Schwartz divergence-free vector field $u_0$, such that there is no global-in-time mild solution $u: [0,+\infty) \to H^{10}_\df(\R^3)$ to the averaged Navier-Stokes equation \eqref{ns-modified} with initial data $u_0$.
\end{theorem}

In fact, the arguments used to prove the above theorem can be pushed a little further to construct a smooth mild solution $u: [0,T_*) \to H^{10}_\df(\R^3)$ for some $0 < T_* < \infty$ that blows up (at the spatial origin) as $t$ approaches $T_*$ (and with subcritical norms such as $\|u(t)\|_{H^{10}_\df(\R^3)}$ diverging to infinity as $t \to T_*$).

\begin{remark} One can also rewrite the averaged Navier-Stokes equation \eqref{ns-modified} in a form more closely resembling \eqref{ns}, namely
\begin{align*}
\partial_t u + T(u,u) &= \Delta u - \nabla p \\
\nabla \cdot u &= 0 \\
u(0,\cdot) &= u_0
\end{align*}
where $T$ is an averaged version of the convection operator $(u \cdot \nabla)u$, defined by $T = \frac{1}{2} (T_{12} + T_{21})$ where
$$ T_{ij}(u,u) := 
\int_\Omega \Rot_{R_{3,\omega}^{-1}} \overline{m_{3,\omega}}(D) \left ( (m_{i,\omega}(D) \Rot_{R_{i,\omega}} u \cdot \nabla) m_{j,\omega}(D) \Rot_{R_{j,\omega}} u\right)\ d\mu(\omega)$$
for $ij=12,21$.  We can also ensure that the inviscid form of the averaged Navier-Stokes equation conserves helicity, as well as total momentum, angular momentum, and vorticity; see Remark \ref{helicity} below.
\end{remark}

Our construction of this averaged bilinear operator $\tilde B: H^{10}_\df(\R^3) \times H^{10}_\df(\R^3) \to H^{10}_\df(\R^3)^*$ and blowup solution $u$ will admittedly be rather artificial, as the averaged operator $\tilde B$ will only retain a carefully chosen (and carefully weighted) subset of the nonlinear interactions present in the original operator $B$, with the weights designed to facilitate a specific blowup mechanism while suppressing other nonlinear interactions that could potentially disrupt this mechanism.  There is however a possibility that the proof strategy in Theorem \ref{main} could be adapted to the true Navier-Stokes equations; see Section \ref{program} below.   Even without this possibility, however, we view this result as a significant (but not completely inpenetrable) \emph{barrier} to a certain class of strategies for excluding such blowup based on treating the bilinear Euler operator $B$ abstractly, as it shows that any strategy that fails to distinguish between the Euler bilinear operator $B$ and its averaged counterparts $\tilde B$ (assuming that the averages obey the cancellation \eqref{cancellation-2}) is doomed to failure.   We emphasise however that this barrier does not rule out arguments that crucially exploit specific properties of the Navier-Stokes equation that are not shared by the averaged versions.  For instance, the arguments in \cite{ess} (see also the subsequent paper \cite{kenig-koch} for an alternate treatment), which establish global regularity for Navier-Stokes subject to a hypothesis of bounded critical norm, rely on a unique continuation property for backwards heat equations which in turn relies on being able to control the nonlinearity pointwise in terms of the solution and its first derivatives.  This is a particular feature of the Navier-Stokes equation \eqref{ns} (in vorticity formulation) which is difficult to discern from the projected formulation \eqref{ns-smooth}, and does not hold in general in \eqref{ns-modified}; in particular, it is not obvious to the author whether the main results in \cite{ess} extend\footnote{This would not be in contradiction to Theorem \ref{main}, as the blowup solution constructed in the proof of that theorem is of ``Type II'' in the sense that critical norms of the solution $u(t)$ diverge in the limit $t \to T_*$.  In contrast, the results in \cite{ess} rules out ``Type I'' blowup, in which a certain critical norm stays bounded.} to averaged Navier-Stokes equations.   As such, arguments based on such unique continuation properties are (currently, at least) examples of approaches to the regularity problem that are not manifestly subject to this barrier (unless progress is made on the program outlined in Section \ref{program} below).  Another example of a positive recent result on the Navier-Stokes problem that uses the finer structure of the nonlinearity (and is thus not obviously subject to this barrier) is the work in \cite{chemin} constructing large data smooth solutions to the Navier-Stokes equations in which the initial data varies slowly in one direction, and which relies on certain delicate algebraic properties of the symbol of $B$.

\subsection{Overview of proof}\label{overview-sec}

The philosophy of proof of Theorem \ref{main} is to treat the dissipative term $\Delta u$ of \eqref{ns-modified} as a perturbative error (which is possible thanks to the supercritical nature of the energy, due to the fact that we are in more than two spatial dimensions), and to construct a stable blowup solution to the ``averaged Euler equation'' $\partial_t u = \tilde B(u,u)$ that blows up so rapidly that the effect of adding a dissipation\footnote{Indeed, our arguments permit one to add any supercritical hyperdissipation $(-\Delta)^\alpha$, $\alpha < 5/4$, to the equation \eqref{ns-modified} while still obtaining blowup for certain choices of initial data, although for sake of exposition we will only discuss the classical $\alpha=1$ case here.} term is negligible.  This blowup solution will have a significant portion of its energy concentrating on smaller and smaller balls around the spatial origin $x=0$; more precisely, there will be an increasing sequence of times $t_n$ converging exponentially fast to a finite limit $T_*$, such that a large fraction of the energy (at least $(1+\epsilon_0)^{-\eps n}$ for some small $\eps, \epsilon_0>0$) is concentrated in the ball $B(0,(1+\epsilon_0)^{-n})$ centred at the origin.  We will be able to make the difference $t_{n+1}-t_n$ of the order $(1+\epsilon_0)^{(-\frac{5}{2}+O(\eps))n}$ for some small $\eps>0$; this is about as short as one can hope from scaling heuristics (see e.g. \cite{tao-hyper} for a discussion), and indicates a blowup which is almost as rapid and efficient as possible, given the form of the nonlinearity.  In particular, for large $n$, the time difference $t_{n+1}-t_n$ will be significantly shorter than the dissipation time $(1+\epsilon_0)^{-2n}$ at that spatial scale, which helps explain why the effect of the dissipative term $\Delta u$ will be negligible.  

To construct the stable blowup solution, we were motivated by the work on regularity and blowup of the system of ODE
\begin{equation}\label{xn}
\partial_t X_n = - \lambda^{2n\alpha} X_n + \lambda^{n-1} X_{n-1}^2 - \lambda^n X_n X_{n+1}
\end{equation}
for a system $(X_n)_{n\in \Z}$ of scalar unknown functions $X_n: [0,T_*) \to \R$, where $\lambda>1$ and $\alpha>0$ are parameters.  This system was introduced by Katz-Pavlovic \cite{katz-dyadic} (with $\lambda=2$ and $\alpha=2/5$) as a dyadic model\footnote{Strictly speaking, the equation studied in \cite{katz-dyadic} is slightly different, in that there is a nonlinear interaction between each wavelet in the model and all of the children of that wavelet, whereas the model here corresponds to the case where each wavelet interacts with only one of its children at most.  The equation in \cite{katz-dyadic} turns out to be a bit more dispersive than the model \eqref{xn}, and in particular enjoys global regularity (by an unpublished argument of Nazarov), and is thus not directly suitable as a model for proving Theorem \ref{main}.}  for the Navier-Stokes equations \eqref{ns}, and are related to hierarchical shell models for these equations (see also \cite{des} for an earlier derivation of these equations from Fourier-analytic considerations).  Roughly speaking, a solution $(X_n)_{n \in \Z}$ to this system (with $\alpha=2/5$) corresponds (at a heuristic level) to a solution $u$ to an equation similar to \eqref{ns-smooth} or \eqref{ns-modified} with $u$ of the shape
\begin{equation}\label{utx}
u(t,x) \approx \sum_n X_n(t) \lambda^{3n/5} \psi( \lambda^{2n/5} x )
\end{equation}
for some Schwartz function $\psi$ with Fourier transform vanishing near the origin.  We remark that the analogue of the energy identity \eqref{energy-ident} in this setting is the identity
\begin{equation}\label{energy-ident-dyadic}
\frac{1}{2} \sum_n X_n(T)^2 + \int_0^T \sum_n \lambda^{2n\alpha} X_n(t)^2 dt = \frac{1}{2} \sum_n X_n(0)^2,
\end{equation}
valid whenever $X_n$ exhibits sufficient decay as $n \to \pm \infty$ (we do not formalise this statement here).

We will defer for now the technical issue (which we regard as being of secondary importance) of transferring blowup results from dyadic Navier-Stokes models to averaged Navier-Stokes models, and focus on the question of whether blowup solutions may be constructed for ODE systems such as \eqref{xn}.

Blowup solutions for the equation \eqref{xn} are known to exist for sufficiently small $\alpha$; specifically, for $\alpha < 1/4$ this was (essentially) established in \cite{katz-dyadic}, while for $\alpha < 1/3$ this was established
in \cite{ches}, with global regularity established in the critical and subcritical regimes $\alpha \geq 1/2$.  If a blowup solution could be constructed\footnote{The results in \cite{katz-dyadic} can be however adapted to establish a version of Theorem \ref{main} in six and higher dimensions, while the results in \cite{ches} give a version in five and higher dimensions (and just barely miss the four-dimensional case); this can be done by adapting the arguments in this paper (and using the above-cited blowup results as a substitute for the lengthier ODE analysis in this paper), and we leave the details to the interested reader.  Interestingly, the results in \cite{plechac}, \cite{plechac-2} on a somewhat different Navier-Stokes type model also indicate blowup in five and higher dimensions, while giving global regularity instead in lower dimensions; similarly for a third Navier-Stokes model introduced in \cite{hou}.} with the value $\alpha=2/5$, then this would be a dyadic analogue of Theorem \ref{main}.  Unfortunately for our purposes, for the values $\lambda=2^{1/\alpha}, \alpha=2/5$, global regularity was established in \cite{bmr} (for non-negative initial data $X_n(0)$), by carefully identifying a region of phase space that is invariant under forward evolution of \eqref{xn}, and which in particular prevents the energy $X_n$ from concentrating too strongly at a single value of $n$.  However, the argument in \cite{bmr} is sensitive to the specific numerical value of $\lambda$ (and also relies heavily on the assumption of initial non-negativity), and does not rule out the possibility of blowup at $\alpha=2/5$ for some variant of the system \eqref{xn}. 

From multiplying \eqref{xn} by $X_n$, we arrive at the energy transfer equations
\begin{equation}\label{e-transfer}
 \partial_t \left(\frac{1}{2} X_n^2\right) = - \lambda^{2n\alpha} X_n^2 + \lambda^{n-1} X_n X_{n-1}^2 - \lambda^n X_{n+1} X_n^2
\end{equation}
for $n \in \Z$,
which are a local version of \eqref{energy-ident-dyadic}, and reveal in particular (in the non-negative case $X_n \geq 0$) that there is a flow of energy at rate $\lambda^n X_{n+1} X_n^2$ from the $n^{\operatorname{th}}$ mode $X_n$ to the $(n+1)^{\operatorname{st}}$ mode $X_{n+1}$.  In principle, whenever one is in the supercritical regime $\alpha < 1/2$, one should be able to start with a delta function initial data $X_n(0) = 1_{n=n_0}$ for some sufficiently large $n_0$, and then this transfer of energy should allow for a ``low-to-high frequency cascade'' solution in which the energy moves rapidly from the $n_0^{\operatorname{th}}$ mode to the $(n_0+1)^{\operatorname{st}}$ mode, with the cascade fast enough to ``outrun'' the dissipative effect of the term $-\lambda^{2n\alpha} X_n^2$ in the energy transfer equation \eqref{e-transfer}, which is lower order when $\alpha<1/2$.  However, as observed in \cite{bmr}, this cascade scenario does not actually occur as strongly as the above heuristic reasoning suggests, because the energy in $X_{n+1}$ is partially transferred to $X_{n+2}$ before the transfer of energy from $X_n$ to $X_{n+1}$ is fully complete, leading instead to a solution in which the bulk of the energy remains in low values of $n$ and is eventually dissipated away by the $-\lambda^{2n\alpha} X_n^2$ term before forming a singularity.  Thus we see that there is an interference effect between the energy transfer between $X_n$ and $X_{n+1}$, and the energy transfer between $X_{n+1}$ and $X_{n+2}$, that disrupts the naive blowup scenario.

One can fix this problem by suitably modifying the model equation \eqref{xn}.  One rather drastic (and not particularly satisfactory) way to do this is to forcibly (i.e., \emph{exogenously}) shut off most of the nonlinear interactions, so that only one pair $X_n,X_{n+1}$ of adjacent modes experiences a nonlinear (but energy-conserving) interaction at any given time.  Specifically, one can consider a truncated-nonlinearity ODE
\begin{equation}\label{trunc-non}
\partial_t X_n = - \lambda^{2n\alpha} X_n + 1_{n-1 = n(t)} \lambda^{n-1} X_{n-1}^2 - 1_{n = n(t)} \lambda^n X_n X_{n+1}
\end{equation}
where $n: [0,T_*) \to \Z$ is a piecewise constant function that one specifies in advance, and which describes which pair of modes $X_{n(t)}, X_{n(t)+1}$ is ``allowed'' to interact at a given time $t$.  It is not difficult to construct a blowup solution for this truncated ODE; we do so in Section \ref{beyond}.  Such a result corresponds to a weak version of Theorem \ref{main} in which the averaged nonlinearity $\tilde B$ is now allowed to be time dependent, $\tilde B = \tilde B(t)$, with the dependence of $\tilde B(t)$ on $t$ being piecewise constant (and experiencing an unbounded number of discontinuities as $t$ approaches $T_*$). In particular, the nonlinearity $\tilde B(t)$ is experiencing an exogenous oscillatory singularity in time as $t$ approaches $T_*$, making the spatial singularity of the solution $u$ become significantly less surprising\footnote{It is worth noting, however, that a surprisingly large portion of the local theory for Navier-Stokes would survive with a time-dependent nonlinearity, even if it were discontinuous in time, so even this weakened version of Theorem \ref{main} provides a somewhat non-trivial barrier that can still exclude certain solution strategies to the Navier-Stokes regularity problem.}.  

Our strategy, then, is to design a system of ODE similar to \eqref{xn} that can \emph{endogenously} simulate the exogenous truncations $1_{n-1=n(t)}$, $1_{n=n(t)}$ of \eqref{trunc-non}.  As shown in \cite{bmr}, this cannot be done for the scalar equation \eqref{xn}, at least when $\lambda$ is equal to $2$.  However, by replacing \eqref{xn} with a vector-valued generalisation, in which one has four scalar functions $X_{1,n}(t),X_{2,n}(t),X_{3,n}(t),X_{4,n}(t)$ associated to each scale $n$, rather than a single scalar function $X_n(t)$, it turns out to be possible to use quadratic interactions of the same strength as the terms $\lambda^{n-1} X_{n-1}^2, \lambda^n X_n X_{n+1}$ appearing in \eqref{xn} to induce such a simulation, while still respecting the energy identity.  The precise system of ODE used is somewhat complicated (see Section \ref{blowup-sec}), but it can be described as a sequence of ``quadratic circuits'' connected in series, with each circuit built out of a small number of ``quadratic logic gates'', each corresponding to a certain type of basic quadratic nonlinear interaction.  Specifically, we will combine together some ``pump'' gates that transfer energy from one mode to another (and which are the only gate present in \eqref{xn}) with ``amplifier'' gates (that use one mode to ignite exponential growth in another mode) and ``rotor'' gates (that use one mode to rotate the energy between two other modes).  By combining together these gates with carefully chosen coupling constants (a sort of ``quadratic engineering'' task somewhat analogous to the more linear circuit design tasks in electrical engineering), we can set up a transfer of energy from scale $n$ to scale $n+1$ which can be made arbitrarily abrupt, in that the duration of the time interval separating the regime in which most of the energy is at scale $n$, and most of the energy is at scale $n+1$, can be made as small as desired.  Furthermore, this transfer is delayed somewhat from the time at which the scale $n$ first experiences a large influx of energy.  The combination of the delay in energy transfer and the abruptness of that transfer means that the process of transferring energy from scale $n$ to scale $n+1$ is not itself interrupted (up to negligible errors) by the process of transferring energy from scale $n+1$ to $n+2$, and this permits us (after a lengthy bootstrap argument) to construct a blowup solution to this equation, which resembles the blowup solution for the truncated ODE \eqref{trunc-non}.

We now briefly discuss how to pass from the dyadic model problem of establishing blowup for a variant of \eqref{xn} to a problem of the form \eqref{ns-modified}, though as noted before we view the dyadic analysis as containing the core results of the paper, with the conversion to the non-dyadic setting being primarily for aesthetic reasons (and to eliminate any lingering suspicion that the blowup here is arising from some purely dyadic phenomenon that is somehow not replicable in the non-dyadic setup).  By using an ansatz of the form \eqref{utx} and rewriting everything in Fourier space, one can map the dyadic model problem to a problem similar to \eqref{ns-modified}, but with the Laplacian replaced by a ``dyadic Laplacian'' (similar to the one appearing in \cite{katz-dyadic}, \cite{fp}), and with a bilinear operator $\tilde B$ which has a Fourier representation
$$ \langle \tilde B(u,v), w \rangle = \int\int\int \tilde m(\xi_1,\xi_2,\xi_3)( \hat u(\xi_1), \hat v(\xi_2), \hat w(\xi_3) )\ d\xi_1 d\xi_2 d\xi_3 $$
for a certain tensor-valued symbol $\tilde m(\xi_1,\xi_2,\xi_3)$ that is supported on the region of frequency space where $\xi_1,\xi_2,\xi_3$ are comparable in magnitude, and having magnitude $\sim |\xi_1|$ in that region (together with the usual estimates on derivatives of the symbol).  Meanwhile, thanks to \eqref{bwing}, $B$ has a similar representation
$$ \langle B(u,v), w \rangle = \int\int\int m(\xi_1,\xi_2,\xi_3)( \hat u(\xi_1), \hat v(\xi_2), \hat w(\xi_3) )\ d\xi_1 d\xi_2 d\xi_3 $$
with $m$ being a singular (tensor-valued) distribution on the hyperplane $\xi_1+\xi_2+\xi_3=0$.  After averaging over some rotations, one can\footnote{For minor technical and notational reasons, the formal version of this argument performed in Section \ref{euler-avg} does not quite perform these steps in the order indicated here, however all the ingredients mentioned here are still used at some point in the rigorous argument.} ``smear out'' the distribution $m$ to be absolutely continuous with respect to $d\xi_1 d\xi_2 d\xi_3$, and then by suitably modulating by Fourier multipliers of order $0$ (and in particular, differentiation operators of imaginary order) one can localise the symbol to the region of frequency space where $\xi_1,\xi_2,\xi_3$ are comparable in magnitude.  By performing some suitable Fourier-type decompositions of the latter symbol, we are then able to express $\tilde m$ as an average of various transformations of $m$, giving rise to a description of $\tilde B$ as an averaged Navier-Stokes operator.  Ultimately, the problem boils down to the task of establishing a certain non-degeneracy property of the tensor symbol $\Lambda$ defined in \eqref{lambda-def}, which one establishes by a short geometric calculation.  The averaging over dilations in Theorem \ref{main} is needed in order to ensure this non-degeneracy property, but it is likely that this averaging can be dropped by a more careful analysis.

This almost finishes the proof of Theorem \ref{main}, except that the dyadic model equation involves the dyadic Laplacian instead of the Euclidean Laplacian.  However, it turns out that the analysis of the dyadic system of ODE can be adapted to the case of non-dyadic dissipation, by using local energy inequalities as a substitute for the exact ODE that appear in the dyadic model.  While this complicates the analysis slightly, the effect is ultimately negligible due to the perturbative nature of the dissipation.

\subsection{A program for establishing blowup for the true Navier-Stokes equations?}\label{program}

To summarise the strategy of proof of Theorem \ref{main}, a solution to a carefully chosen averaged version
$$ \partial_t u = \tilde B(u,u)$$
of the Euler equations is constructed which behaves like a ``von Neumann machine'' (that is, a self-replicating machine) in the following sense: at a given time $t_n$, it evolves as a sort of ``quadratic computer'', made out of ``quadratic logic gates'', which is ``programmed'' so that after a reasonable period of time $t_{n+1}-t_n$, it abruptly ``replicates'' into a rescaled version of itself (being $1+\epsilon_0$ times smaller, and about $(1+\epsilon_0)^{5/2}$ times faster), while also erasing almost completely the previous iteration of this machine.  This replication process is stable with respect to perturbations, and in particular can survive the presence of a supercritical dissipation if the initial scale of the machine is sufficiently small.

This suggests an ambitious (but not obviously impossible) program (in both senses of the word) to achieve the same effect for the true Navier-Stokes equations, thus obtaining a negative answer to Conjecture \ref{nsconj}.  Define an \emph{ideal (incompressible, inviscid) fluid} to be a divergence-free vector field $u$ that evolves according to the true Euler equations
$$ \partial_t u = B(u,u).$$
Somewhat analogously to how a quantum computer can be constructed from the laws of quantum mechanics (see e.g. \cite{benioff}), or a Turing machine can be constructed from cellular automata such as Conway's ``Game of Life'' (see e.g. \cite{adam}), one could hope to design logic gates entirely out of ideal fluid (perhaps by using suitably shaped vortex sheets to simulate the various types of physical materials one would use in a mechanical computer).  If these gates were sufficiently ``Turing complete'', and also ``noise-tolerant'', one could then hope to combine enough of these gates together to ``program'' a von Neumann machine consisting of ideal fluid that, when it runs, behaves qualitatively like the blowup solution used to establish Theorem \ref{main}.  Note that such replicators, as well as the related concept of a \emph{universal constructor}, have been built within cellular automata such as the ``Game of Life''; see e.g. \cite{adam-life}.

Once enough logic gates of ideal fluid are constructed, it seems that the main difficulties in executing the above program are of a ``software engineering'' nature, and would be in principle achievable, even if the details could be extremely complicated in practice.  The main mathematical difficulty in executing this ``fluid computing'' program would thus be to arrive at (and rigorously certify) a design for logical gates of inviscid fluid that has some good noise tolerance properties.  In this regard, ideas from quantum computing (which faces a unitarity constraint somewhat analogous to the energy conservation constraint for ideal fluids, albeit with the key difference of having a linear evolution rather than a nonlinear one) may prove to be useful.

A significant (but perhaps not insuperable) obstacle to this program is that in addition to the conservation of energy, the Euler equations obey a number of additional conservation laws, such as conservation of helicity, with vortex lines also being transported by the flow; see e.g. \cite{bert}.  This places additional limitations on the type of fluid gates one could hope to construct; however, as these conservation laws are indefinite in sign, it may still be possible to design computational gates that respect all of these laws.

It is worth pointing out, however, that even if this program is successful, it would only demonstrate blowup for a very specific type of initial data (and tiny perturbations thereof), and is not necessarily in contradiction with the belief that one has global regularity for \emph{most} choices of initial data (for some carefully chosen definition of ``most'', e.g. with overwhelming (but not almost sure) probability with respect to various probability distributions of initial data).  However, we do not have any new ideas to contribute on how to address this latter question, other than to state the obvious fact that deterministic methods alone are unlikely to be sufficient to resolve the problem, and that stochastic methods (e.g. those based on invariant measures) are probably needed.

\subsection{Acknowledgments}  

I thank Nets Katz for helpful discussions, Zhen Lei and Gregory Seregin for help with the references, and the anonymous referee for a careful reading and pointing out an error in a previous version of this manuscript.  The author is supported by a Simons Investigator grant, the James and Carol Collins Chair, the Mathematical Analysis \& Application Research Fund Endowment, and by NSF grant DMS-1266164.  

\section{Notation}

We use $X=O(Y)$ or $X \lesssim Y$ to denote the estimate $|X| \leq CY$, for some quantity $C$ (which we call the \emph{implied constant}).  If we need the implied constant to depend on a parameter (e.g. $k$), we will either indicate this convention explicitly in the text, or use subscripts, e.g. $X = O_k(Y)$ or $X \lesssim_k Y$.

If $\xi$ is an element of $\R^3$, we use $|\xi|$ to denote its Euclidean magnitude.  For $\xi_0 \in \R^3$ and $r>0$, we use $B(\xi_0,r) := \{ \xi \in \R^3: |\xi-\xi_0| < r \}$ to denote the open ball of radius $r$ centred at $\xi_0$. Given a subset $B$ of $\R^3$ and a real number $\lambda$, we use $\lambda \cdot B := \{ \lambda \xi: \xi \in B \}$ to denote the dilate of $B$ by $\lambda$.

If $P$ is a mathematical statement, we use $1_P$ to denote the quantity $1$ when $P$ is true and $0$ when $P$ is false.

Given two real vector spaces $V,W$, we define the tensor product $V \otimes W$ to be the real vector space spanned by formal tensor products $v \otimes w$ with $v \in V$ and $w \in W$, subject to the requirement that the map $(v,w) \mapsto v \otimes w$ is bilinear.  Thus for instance $V \otimes \C$ is the complexification of $V$, that is to say the space of formal linear combinations $v_1+iv_2$ with $v_1,v_2 \in V$.

\section{Averaging the Euler bilinear operator}\label{euler-avg}

In this section we show that certain bilinear operators, which are spatially localised variants of the ``cascade operators'' introduced in \cite{katz-dyadic}, can be viewed as averaged Euler bilinear operators.  

We now formalise the class of local cascade operators we will be working with.
For technical reasons, we will use the integer powers $(1+\epsilon_0)^n$ of $1+\epsilon_0$ for some sufficiently small $\epsilon_0>0$ as our dyadic range of scales, rather than the more traditional powers of two, $2^n$.  Roughly speaking, the reason for this is to ensure that any triangle of side lengths that are of comparable size, in the sense that they all between $(1+\epsilon_0)^{n-O(1)}$ and $(1+\epsilon_0)^{n+O(1)}$ for some $n$, are almost equilateral; this lets us avoid some degeneracies in the tensor symbol implicit in \eqref{bwing} that would otherwise complicate the task of expressing certain bilinear operators as averages of the Euler bilinear operator $B$ (specifically, the smallness of $\epsilon_0$ is needed to establish the non-degeneracy condition \eqref{c-nondeg} below).

\begin{definition}[Local cascade operators]\label{cascdef}  Let $\epsilon_0>0$.  A \emph{basic local cascade operator} (with dyadic scale parameter $\epsilon_0>0$) is a bilinear operator $C: H^{10}_\df(\R^3) \times H^{10}_\df(\R^3) \to H^{10}_\df(\R^3)^*$ defined via duality by the formula
\begin{equation}\label{cuw}
 \langle C(u,v), w \rangle = \sum_{n \in \Z} (1+\epsilon_0)^{5n/2} \langle u, \psi_{1,n} \rangle \langle v, \psi_{2,n} \rangle \langle w, \psi_{3,n} \rangle
\end{equation}
for all $u,v,w \in H^{10}_\df(\R^3)$, where for $i=1,2,3$ and $n \in \Z$, $\psi_{i,n}: \R^3 \to \R^3$ is the $L^2$-rescaled function
$$ \psi_{i,n}(x) := (1+\epsilon_0)^{3n/2} \psi_i\left( (1+\epsilon_0)^n x \right)$$
and $\psi_i: \R^3 \to \R^3$ is a Schwartz function whose Fourier transform is supported on the annulus $\{ \xi: 1-2\epsilon_0 \leq |\xi| \leq 1+2\epsilon_0 \}$.  A \emph{local cascade operator} is defined to be a finite linear combination of basic local cascade operators.
\end{definition}

Note from the Plancherel theorem that one has
$$ \sum_n (1 + (1+\epsilon_0)^{2n})^{10} |\langle u, \psi_{1,n} \rangle|^2 < \infty$$
whenever $u \in H^{10}_\df(\R^3)$ and $\psi_{1,n}$ is as in Definition \ref{cascdef}.  Similarly for $\psi_{2,n}$ and $\psi_{3,n}$.  From this and the H\"older inequality it is an easy matter to ensure that the sum in \eqref{cuw} is absolutely convergent for any $u,v,w \in H^{10}_\df(\R^3)$, so the definition of a cascade operator is well-defined, and that such operators are bounded from $H^{10}_\df(\R^3) \times H^{10}_\df(\R^3)$ to $H^{10}_\df(\R^3)^*$; indeed, the same argument shows that such operators map $H^{10}_\df(\R^3) \times H^{10}_\df(\R^3)$ to $L^2(\R^3)$.  (One could in fact extend such operators to significantly rougher spaces than $H^{10}_\df(\R^3)$, but we will not need to do so here.) 

We did not impose that the $\psi_i$ were divergence free, but one could easily do so via Leray projections if desired, in which case the operators $C(u,v)$ defined via duality in \eqref{cuw} can be expressed more directly as
$$
 C(u,v) = \sum_{n \in \Z} (1+\epsilon_0)^{5n/2} \langle u, \psi_{1,n} \rangle \langle v, \psi_{2,n} \rangle \psi_{3,n}.
$$
We remark that the exponent $5/2$ appearing in \eqref{cuw} ensures that local cascade operators enjoy a dyadic version of the scale invariance that the Euler bilinear form enjoys.  Indeed, recalling the dilation operators \eqref{dila}, one can compute that for any $u,v,w \in H^{10}_\df(\R^3)$, one has
$$ \langle B( \Dil_\lambda u, \Dil_\lambda v ), \Dil_\lambda w \rangle = \lambda^{5/2} \langle B(u,v), w \rangle,$$
and similarly for any local cascade operator $C$ one has
$$ \langle C( \Dil_\lambda u, \Dil_\lambda v ), \Dil_\lambda w \rangle = \lambda^{5/2} \langle C(u,v), w \rangle,$$
under the additional restriction that $\lambda$ is an integer power of $1+\epsilon_0$.

Theorem \ref{main} is then an immediate consequence of the following two results.

\begin{theorem}[Local cascade operators are averaged Euler operators]\label{avg}  Let $\epsilon_0>0$ be a sufficiently small absolute constant.  Then every local cascade operator (with dyadic scale parameter $\epsilon_0$) is an averaged Euler bilinear operator.
\end{theorem}

\begin{theorem}[Blowup for a local cascade equation]\label{blowup}  Let $0 < \epsilon_0 < 1$.  Then there exists a symmetric local cascade operator $C: H^{10}_\df(\R^3) \times H^{10}_\df(\R^3) \to H^{10}_\df(\R^3)^*$ (with dyadic scale parameter $\epsilon_0$) obeying the cancellation property 
\begin{equation}\label{cancelled}
 \langle C(u,u), u \rangle = 0
\end{equation}
for all $u \in H^{10}_\df(\R^3)$, and Schwartz divergence-free vector field $u_0$, such that there does not exist any global mild solution $u: [0,+\infty) \to H^{10}_\df(\R^3)$ to the initial value problem 
\begin{equation}\label{system}
\begin{split}
\partial_t u &= \Delta u + C(u,u) \\
u(0,\cdot) &= u_0,
\end{split}
\end{equation}
that is to say there does not exist any continuous $u: [0,+\infty) \to H^{10}_\df(\R^3)$ with
$$ u(t) = e^{t\Delta} u_0 + \int_0^t e^{(t-t')\Delta} C\left( u(t'), u(t') \right)\ dt'$$
for all $t \in [0,+\infty)$.
\end{theorem}

Theorem \ref{blowup} is the main technical result of this paper, and its proof will occupy the subsequent sections of this paper.  In this section we establish Theorem \ref{avg}.  This will be done by a somewhat lengthy series of averaging arguments and Fourier decompositions, together with some elementary three-dimensional geometry, with the result ultimately following from a certain non-degeneracy property of the trilinear form $\Lambda$ defined in \eqref{lambda-def}; the arguments are unrelated to those in the rest of the paper, and readers may wish to initially skip this section and move on to the rest of the argument.

Henceforth $\epsilon_0>0$ will be assumed to be sufficiently small (e.g. $\epsilon_0 = 10^{-10}$ will suffice).  In this section, the implied constants in the $O()$ notation are not permitted to depend on $\epsilon_0$. 

\subsection{First step: complexification}

It will be convenient to complexify the problem in order to freely use Fourier-analytic tools at later stages of the argument.  To this end, we introduce the following notation.

\begin{definition}[Complex averaging]\label{avg-def}  Let $C, C': H^{10}_\df(\R^3) \otimes \C \times H^{10}_\df(\R^3) \otimes \C \to H^{10}_\df(\R^3)^* \otimes \C$ be bounded (complex-)bilinear operators.  We say that $C$ is a \emph{complex average} of $C'$ if there exists a finite measure space $(\Omega,\mu)$ and measurable functions $m_{i,\cdot}(D): \Omega \to {\mathcal M}_0 \otimes \C$, $R_{i,\cdot}: \Omega \to \SO(3)$, $\lambda_{i,\cdot}: \Omega \to (0,+\infty)$ for $i=1,2,3$ such that
\begin{equation}\label{cuvw}
\begin{split}
&\langle C(u,v), w \rangle = \int_\Omega \\
&\left\langle C'\left( m_{1,\omega}(D) \Rot_{R_{1,\omega}} \Dil_{\lambda_{1,\omega}} u, m_{2,\omega}(D) \Rot_{R_{2,\omega}} \Dil_{\lambda_{2,\omega}} v\right), m_{3,\omega}(D) \Rot_{R_{3,\omega}} \Dil_{\lambda_{3,\omega}} w\right\rangle\ d\mu(\omega),
\end{split}\end{equation}
and that one has the integrability conditions
\begin{equation}\label{integrab}
 \int_\Omega \| m_{1,\omega}(D) \|_{k_1} \| m_{2,\omega}(D) \|_{k_2} \| m_{3,\omega}(D) \|_{k_3}\ d\mu(\omega) < \infty
 \end{equation}
and
$$ C_0^{-1} \leq \lambda_1(\omega), \lambda_2(\omega), \lambda_3(\omega) \leq C_0$$
for any natural numbers $k_1,k_2,k_3$ (recall that the seminorms $\| \|_k$ on ${\mathcal M}_0 \otimes \C$ were defined in \eqref{seminorm}) and some finite $C_0$.  Here, we complexify the inner product $\langle,\rangle$ by defining 
$$ \langle u, v \rangle := \int_{\R^3} u(x) \cdot v(x)\ dx$$
for complex vector fields $u,v \in H^{10}_\df(\R^3) \otimes \C$; note that we do \emph{not} place a complex conjugate on the $v$ factor, so the inner product is complex bilinear rather than sesquilinear.
\end{definition}

Suppose we can show that every local cascade operator $C$ is a complex average of the Euler bilinear operator $B$ in the sense of the above definition.  The multipliers $m_{j,\omega}(D)$ for $j=1,2,3$ appearing in the expansion \eqref{cuvw} are not required to be real, but we can decompose them as $m_{j,\omega,1}(D) + i m_{j,\omega,2}(D)$ where $m_{j,\omega,1}(D), m_{j,\omega,2}(D)$ are real (and with the seminorms of $m_{j,\omega,1}(D), m_{j,\omega,2}(D)$ bounded by a multiple of the corresponding seminorm of $m_{j,\omega}(D)$).  Thus we can decompose the right-hand side of \eqref{cuvw} as the sum of $2^3=8$ pieces, each of which is of the same form as the original right-hand side up to a power of $i$, and with all the $m_{j,\omega}(D)$ appearing in each piece being a real Fourier multiplier.  As the left-hand side of \eqref{cuvw} is real (as are the inner products on the right-hand side), we may eliminate all the terms on the right-hand side involving odd powers of $i$ by taking real parts.  The power of $i$ in each of the four remaining terms is now just a sign $\pm 1$ and can be absorbed into the $m_{1,\omega}(D)$ factor; by concatenating together four copies of $(\Omega,\mu)$ we may now obtain an expansion of the form \eqref{cuvw} in which all the $m_{j,\omega}(D)$ are real.  Finally, by multiplying $m_{1,\omega}(D)$ by a normalising constant we may take $(\Omega,\mu)$ to be a probability space rather than a finite measure space.  Combining all these manipulations, we conclude Theorem \ref{avg}.  Thus, it will suffice to show that every local cascade operator is a complex average of the Euler bilinear operator $B$.

\subsection{Second step: frequency localisation}

By again using $m_{1,\omega}(D)$ to absorb scalar factors, we see that if $C$ is a complex average of $C'$, then any complex scalar multiple of $C$ is a complex average of $C'$; also, by concatenating finite measure spaces together we see from Definition \ref{avg-def} that if $C_1, C_2$ are both complex averages of $C'$, then $C_1+C_2$ is an complex average of $C'$.  Thus the space of averages of the Euler bilinear operator is closed under finite linear combinations, and so it will suffice to show that every \emph{basic} local cascade operator is a complex average of the Euler bilinear operator.  

By decomposing the $\psi_j$, $j=1,2,3$ in \eqref{cuvw} into finitely many (complex-valued) pieces, we may replace the basic local cascade operator with the complexified basic local cascade operator $C$ defined by 
\begin{equation}\label{cuw-soft}
 \langle C(u,v), w \rangle = \sum_{n \in \Z} (1+\epsilon_0)^{5n/2} \langle u, \overline{\psi_{1,n}} \rangle \langle v, \overline{\psi_{2,n}} \rangle \langle w, \overline{\psi_{3,n}} \rangle,
\end{equation}
where each $\psi_j: \R^3 \to \C^3$ is now a Schwartz complex vector field with Fourier transform $\hat \psi_j$ supported on the ball $B(\xi^0_j, \epsilon_0^3)$ for some non-zero $\xi^0_j \in \R^3$ with magnitude comparable to $1$.  Henceforth we fix $C$ to be such a complexified basic local cascade operator.  Note that due to the presence of rotations and dilations in the definition of a complex average, we have the freedom to rotate each and dilate each of the $\xi^0_j$ as we please.  We shall select the normalisation
\begin{equation}\label{xi-split}
\begin{split}
\xi^0_1 &= (0,1,0) \\
\xi^0_2 &= (-1,-1,0) \\
\xi^0_3 &= (1,0,0)
\end{split}
\end{equation}
so that in particular
\begin{equation}\label{xisum}
\xi^0_1 + \xi^0_2 + \xi^0_3 = 0;
\end{equation}
see Figure \ref{fig:freq}.  The exact normalisation in \eqref{xi-split} is somewhat arbitrary, but the vanishing \eqref{xisum} is convenient for technical reasons; also, it is necessary to ensure that $\xi^0_1,\xi^0_2,\xi^0_3$ have distinct magnitudes in order to avoid a certain degeneracy later in the argument (namely, the failure of \eqref{c-nondeg} below).

\begin{figure} [t]
\centering
\includegraphics{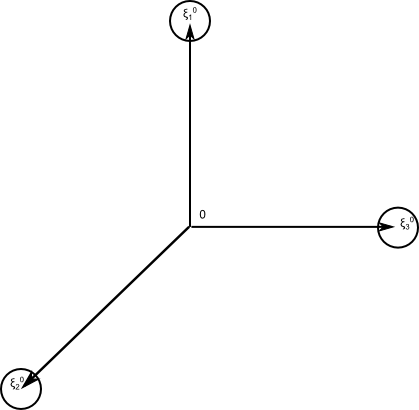}
\caption[Frequencies]{The frequencies $\xi^0_1,\xi^0_2,\xi^0_3$.    The frequency variables $\xi_j$ (and later on, the normalised frequencies $\tilde \xi_j$) will be localised to within $O(\epsilon_0^3)$ of $\xi_j^0$; this localisation is represented schematically in this figure by the circles around the reference frequencies $\xi_j^0$.}
\label{fig:freq}
\end{figure}

Once we perform this normalisation, we will have no further need of averaging over dilations, and will rely purely on Fourier and rotation averaging to obtain the required representation of the cascade operator $C$.

Note that ${\mathcal M}_0 \otimes \C$ is closed under composition, and from \eqref{seminorm} and the Leibniz rule we have the inequalities
$$ \| m(D) m'(D) \|_k \leq C_k \sum_{k_1=0}^k \sum_{k_2=0}^k \| m(D) \|_{k_1} \|m'(D)\|_{k_2}$$
for all natural numbers $k$ and all $m(D), m'(D) \in {\mathcal M}_0 \otimes \C$ (where $C_k$ depends only on $k$).  From this, Fubini's theorem, and H\"older's inequality, together with the observation that rotation and dilation operators normalise ${\mathcal M}_0 \otimes \C$, we have the following transitivity property: if $C_1$ is a complex average of $C_2$, and $C_2$ is a complex average of $C_3$, then $C_1$ is a complex average of $C_3$.  Our proof strategy will exploit this transitivity by passing from the Euler bilinear operator $B$ to the local cascade operator $C$ in stages, performing a sequence of averaging operations on $B$ to gradually make it resemble the local cascade operator.

\subsection{Third step: forcing frequency comparability}

We now use some differential operators of imaginary order to localise the frequencies $\xi_1,\xi_2,\xi_3$ to be comparable to each other in magnitude.  Let $\varphi: \R \to \R$ be a smooth function supported on $[-2,2]$ that equals one on $[-1,1]$.  We then define the function $\eta: (0,+\infty)^3 \to \R$ by
$$ \eta(N_1,N_2,N_3) := \prod_{j=2}^3 \varphi\left( \frac{1}{10\epsilon_0^2} \left( \frac{N_j}{N_1} - \frac{|\xi^0_j|}{|\xi^0_1|} \right) \right);$$
thus $\eta(|\xi_1|, |\xi_2|, |\xi_3|)$ is only non-vanishing when $\xi_1,\xi_2,\xi_3$ have comparable magnitude.

Note that $\eta(N_1,N_2,N_3) = \eta(1,e^{\log(N_2/N_1)},e^{\log(N_3/N_1)})$, and that $(x,y) \mapsto \eta(1,e^x,e^y)$ is a smooth compactly supported function.  By Fourier\footnote{One could also use Mellin inversion here if desired.} inversion, we thus have a representation of the form
\begin{align*}
 \eta(N_1,N_2,N_3) &= \eta( 1, e^{\log(N_2/N_1)}, e^{\log(N_3/N_1)}) \\
&= \int_\R \int_\R e^{it_2 \log(N_2/N_1)} e^{it_3 \log(N_3/N_1)} \phi(t_2,t_3)\ dt_2 dt_3\\
&= \int_\R \int_\R N_1^{-it_2-it_3} N_2^{it_2} N_3^{it_3} \phi(t_2,t_3)\ dt_2 dt_3
\end{align*}
for any $N_1,N_2,N_3 > 0$, where $\phi: \R^2 \to \C$ is a rapidly decreasing function, thus
$$ \int_\R \int_\R |\phi(t_2,t_3)| (1+|t_2|+|t_3|)^k\ dt_2 dt_3 <\infty$$
for all $k \geq 0$.  If we then define the bilinear operator $B_\eta: H^{10}_\df(\R^3) \otimes \C \times H^{10}_\df(\R^3) \otimes \C \to H^{10}_\df(\R^3)^* \otimes \C$ via duality by the formula
$$ \langle B_\eta(u,v), w \rangle :=
\int_\R \int_\R \left\langle B\left( D^{-it_2-it_3} u, D^{it_2} v \right), D^{it_3} w \right\rangle \phi(t_2,t_3)\ dt_2 dt_3$$
where $D^{it}$ is the Fourier multiplier
$$ \widehat{D^{it} u}(\xi) := |\xi|^{it} \hat u(\xi)$$
then $B_\eta$ is a complex average of $B$ (note that $\| D^{it} \|_k$ grows polynomially in $t$ for each $k$).  From \eqref{bwing} and Fubini's theorem (working first with Schwartz $u,v,w$ to justify all the exchange of integrals, and then taking limits) we see that\footnote{Note that we do not define $\eta(|\xi_1|,|\xi_2|,|\xi_3|)$ when one of $\xi_1,\xi_2,\xi_3$ vanishes, but this is only occurs on a set of measure zero and so there is no difficulty defining the integral.}
$$
\langle B_\eta(u,v), w \rangle = -\pi i \int_{\xi_1+\xi_2+\xi_3=0} \eta(|\xi_1|,|\xi_2|,|\xi_3|) \Lambda_{\xi_1,\xi_2,\xi_3}( \hat u(\xi_1), \hat v(\xi_2), \hat w(\xi_3) ).
$$
It thus suffices to show that $C$ is a complex average of $B_\eta$.

Next, we localise the frequency $\xi_1$ to the correct sequence of balls.  Let $\rho: \R^3 \to \C$ be the function
$$ \rho(\xi_1) := \sum_{n \in \Z} \varphi\left( \frac{1}{\epsilon_0^2} \left( (1+\epsilon_0)^{-n} \xi_1 - \xi_1^0 \right) \right)$$
with $\varphi$ defined as before; thus $\rho$ is supported on the union of the balls $(1+\epsilon_0)^n \cdot B( \xi_1^0, 2\epsilon_0^2)$ for $n \in \Z$.  Let $\rho(D)$ be the associated Fourier multiplier; this is easily checked to be a Fourier multiplier of order $0$.  By Definition \ref{avg-def}, the bilinear operator $B_{\eta,\rho}$ defined by
$$ B_{\eta,\rho}(u,v) := B_\eta( \rho(D) u, v )$$
is clearly a complex average of $B_\eta$, and so it suffices to show that $C$ is a complex average of $B_{\eta,\rho}$.

\subsection{Fourth step: localising to a single frequency scale}

Now we localise to a single scale.  Observe that we can decompose $B_{\eta,\rho} = -\pi i \sum_{n \in \Z} (1+\epsilon_0)^{5n/2} B_{\eta,\rho,n}$, where for each $n \in \Z$ we may define the operator $B_{\eta,\rho,n}: H^{10}_\df(\R^3) \otimes \C \times H^{10}_\df(\R^3) \otimes \C \to H^{10}_\df(\R^3)^* \otimes \C$ by the formula
\begin{align*}
\langle B_{\eta,\rho,n}(u,v), w \rangle &= (1+\epsilon_0)^{-5n/2} \int_{\xi_1+\xi_2+\xi_3=0} \varphi\left(\frac{1}{\epsilon_0^2} \left((1+\epsilon_0)^{-n} \xi_1 - \xi_1^0\right)\right) \\
&\quad \eta(|\xi_1|,|\xi_2|,|\xi_3|) \Lambda_{\xi_1,\xi_2,\xi_3}( \hat u(\xi_1), \hat v(\xi_2), \hat w(\xi_3))
\end{align*}
for $u,v,w \in H^{10}_\df(\R^3)$.
In a similar vein, we may use \eqref{cuw-soft} to decompose $C = \sum_{n \in \Z} (1+\epsilon_0)^{5n/2} C_n$, where
\begin{equation}\label{cnuv}
 \langle C_n(u,v), w \rangle = \langle u, \overline{\psi_{1,n}} \rangle \langle v, \overline{\psi_{2,n}} \rangle \langle w, \overline{\psi_{3,n}} \rangle.
\end{equation}
Observe (by using the change of variables $\tilde \xi := \xi / (1+\epsilon_0)^n$) that we have the scaling laws
$$
 \langle B_{\eta,\rho,n}(u,v), w \rangle  =  \left\langle B_{\eta,\rho,0}\left(\Dil_{(1+\epsilon_0)^{-n}} u, \Dil_{(1+\epsilon_0)^{-n}} v\right), \Dil_{(1+\epsilon_0)^{-n}} w \right\rangle 
$$
and similarly
$$
 \langle C_n(u,v), w \rangle  =  \left\langle C_0\left(\Dil_{(1+\epsilon_0)^{-n}} u, \Dil_{(1+\epsilon_0)^{-n}} v\right), \Dil_{(1+\epsilon_0)^{-n}} w \right\rangle 
$$
for any $n \in \Z$ and $u,v,w \in H^{10}_\df(\R^3) \otimes \C$.

Suppose for now that we can show that $C_0$ is a complex average of $B_{\eta,\rho,0}$ (without the use of dilation operators), thus
\begin{equation}\label{couv}
\langle C_0(u,v), w \rangle = \int_\Omega \left\langle B_{\eta,\rho,0}\left( m_{1,\omega}(D) \Rot_{R_{1,\omega}} u, m_{2,\omega}(D) \Rot_{R_{2,\omega}} v\right), m_{3,\omega}(D) \Rot_{R_3} w\right\rangle\ d\mu(\omega)
\end{equation}
for some $m_{j,\omega}$, $R_j$ ($j=1,2,3$), and $(\Omega,\mu)$ as in Definition \ref{avg-def}. From the definition of $C_0$ (and the support hypotheses on $\psi_1,\psi_2,\psi_3$), we see that we may smoothly localise each $m_{j,\omega}$ to the ball $B( \xi_j^0, O(\epsilon_0^3) )$ without loss of generality (and without destroying the fact that the $m_{i,\omega}(D)$ are Fourier multipliers of order $0$ that obey \eqref{integrab}).  If we then define
$$ m_{i,\omega,n}(\xi) := m_{i,\omega}( (1+\epsilon_0)^{-n} \xi )$$
and $\tilde m_{i,\omega} := \sum_{n \in \Z} m_{i,\omega,n}$, then the $m_{i,\omega}(D)$ are also Fourier multipliers of order $0$ obeying \eqref{integrab}, and the quantity
$$ \int_\Omega \left\langle B_{\eta,\rho}\left( m_{1,\omega,n_1}(D) \Rot_{R_{1,\omega}} u, m_{2,\omega,n_2}(D) \Rot_{R_{2,\omega}} v\right), m_{3,\omega,n_3}(D) \Rot_{R_3} w\right\rangle\ d\mu(\omega)$$
is equal to $-\pi i \langle C_{n_1}(u,v),w\rangle$ when $n_1=n_2=n_3$, and vanishing otherwise if $\epsilon_0$ is small enough (thanks to the support properties of $m_{i,\omega,n}$, $\eta$ and $\rho$).  Summing, we see that
$$ \langle C(u,v), w \rangle = \frac{1}{-\pi i} \int_\Omega \left\langle B_{\eta,\rho}\left( \tilde m_{1,\omega}(D) \Rot_{R_{1,\omega}} u, \tilde m_{2,\omega}(D) \Rot_{R_{2,\omega}} v\right), \tilde m_{3,\omega}(D) \Rot_{R_3} w\right\rangle\ d\mu(\omega)$$
(as before, one can work first with Schwartz $u,v,w$, and then take limits), thus demonstrating that $C$ is a complex average of $B_{\eta,\rho}$ as desired (absorbing the $\frac{1}{-\pi i}$ factor into $m_{1,\omega}$).  Thus, to finish the proof of Theorem \ref{avg}, it suffices to show that $C_0$ is a complex average of $B_{\eta,\rho,0}$.

\subsection{Fifth step: extracting the symbol}

We have reduced matters to the task of obtaining a representation \eqref{couv} for $\langle C_0(u,v), w \rangle$.  By \eqref{cnuv} and Plancherel's theorem, we may expand $\langle C_0(u,v),w\rangle$ as
$$
\int_{\R^3} \int_{\R^3} \int_{\R^3} \left(u(\xi_1) \cdot \overline{\hat \psi_1(\xi_1)}\right) \left(v(\xi_2) \cdot \overline{\hat \psi_2(\xi_2)}\right) \left(w(\xi_3) \cdot \overline{\hat \psi_3(\xi_3)}\right)\ d\xi_1 d\xi_2 d\xi_3$$
which we rewrite as
\begin{equation}\label{expansion}
\int_{\R^3} \int_{\R^3} \int_{\R^3} \left(u(\xi_1) \otimes v(\xi_2) \otimes w(\xi_3)\right) \cdot \left(\overline{\hat \psi_1(\xi_1)} \otimes \overline{\hat \psi_2(\xi_2)} \otimes \overline{\hat \psi_3(\xi_3)}\right)\ d\xi_1 d\xi_2 d\xi_3
\end{equation}
where $\cdot$ here denotes the standard complex-bilinear inner product on the $3^3=27$-dimensional complex vector space $\C^3 \otimes \C^3 \otimes \C^3$.
Meanwhile, the right-hand side of \eqref{couv} can be expanded as
\begin{align*}
&\int_\Omega \int_{\xi_1+\xi_2+\xi_3=0} m_{1,\omega}(\xi_1) m_{2,\omega}(\xi_2) m_{3,\omega}(\xi_3) \varphi\left(\frac{1}{\epsilon_0^2} \left(\xi_1 - \xi_1^0\right)\right) \eta(|\xi_1|,|\xi_2|,|\xi_3|) \times \\
&\quad 
\Lambda_{\xi_1,\xi_2,\xi_3}\left( R_{1,\omega} \hat u(R_{1,\omega}^{-1} \xi_1), R_{2,\omega} \hat v(R_{2,\omega}^{-1} \xi_2), R_{3,\omega} \hat w(R_{3,\omega}^{-1} \xi_3)\right)\ d\mu(\omega).
\end{align*}
Rewriting the integral $\int_{\xi_1+\xi_2+\xi_3=0}$ (by a slight abuse\footnote{If one wanted to be more formally rigorous here, one could replace the Dirac delta function $\delta(\xi)$ here with an approximation to the identity $\frac{1}{\eps^3} \phi(\frac{\xi}{\eps})$ for some smooth compactly supported function $\phi: \R^3 \to \R$ of total mass one, and then add a limit symbol $\lim_{\eps \to 0}$ outside of the integration.}  of notation) as $\int_{\R^3} \int_{\R^3} \int_{\R^3} \delta(\xi_1+\xi_2+\xi_3)\ d\xi_1 d\xi_2 d\xi_3$, where $\delta$ is the Dirac delta function on $\R^3$, and then applying the change of variables $\xi_j \mapsto R_{j,\omega}^{-1} \xi_j$, we may rewrite the above expression as
\begin{align*}
&\int_{\R^3} \int_{\R^3} \int_{\R^3} \int_\Omega \delta( R_{1,\omega} \xi_1 + R_{2,\omega} \xi_2 + R_{3,\omega} \xi_3 )
m_{1,\omega}(R_{1,\omega} \xi_1) m_{2,\omega}(R_{2,\omega} \xi_2) m_{3,\omega}(R_{3,\omega} \xi_3)\\
&\quad \varphi\left(\frac{1}{\epsilon_0^2} \left(R_{1,\omega} \xi_1 - \xi_1^0\right)\right) \eta(|\xi_1|, |\xi_2|,|\xi_3|) \times \\
&\quad 
\Lambda_{R_{1,\omega} \xi_1, R_{2,\omega} \xi_2, R_{3,\omega} \xi_3}( R_{1,\omega} \hat u(\xi_1), R_{2,\omega} \hat v(\xi_2), R_{3,\omega} \hat w(\xi_3)) \\
&\quad d\mu(\omega) d\xi_1 d\xi_2 d\xi_3.
\end{align*}
Comparing this with the expansion \eqref{expansion} of the left-hand side of \eqref{couv}, we claim that our task is now reduced to that of constructing a finite measure space $(\Omega,\mu)$ and measurable functions $R_{i,\cdot}: \Omega \to \SO(3)$, $m_{i,\cdot}(D): \Omega \to {\mathcal M}_0 \otimes \C$, and $F: \Omega \to \C^3\otimes \C^3 \otimes \C^3$ obeying \eqref{integrab} with $F$ bounded, such that we have the identity
\begin{equation}\label{bigmess}
\begin{split}
X_1 \otimes X_2 \otimes X_3 &= \int_\Omega \delta( R_{1,\omega} \xi_1 + R_{2,\omega} \xi_2 + R_{3,\omega} \xi_3 ) F(\omega)
m_{1,\omega}(R_{1,\omega} \xi_1) m_{2,\omega}(R_{2,\omega} \xi_2) m_{3,\omega}(R_{3,\omega} \xi_3)\\
&\quad \varphi\left(\frac{1}{\epsilon_0^2} \left(R_{1,\omega} \xi_1 - \xi_1^0\right)\right) \eta(|\xi_1|,|\xi_2|,|\xi_3|) \times \\
&\quad 
\Lambda_{R_{1,\omega} \xi_1, R_{2,\omega} \xi_2, R_{3,\omega} \xi_3}( R_{1,\omega} X_1, R_{2,\omega} X_2, R_{3,\omega} X_3)\ d\mu(\omega) 
\end{split}
\end{equation}
for all $\xi_j \in B(\xi_j^0,\epsilon_0^3)$ and $X_j \in \xi_j^\perp$, $j=1,2,3$.  Indeed, if one applies \eqref{bigmess} with $(X_1,X_2,X_3) = (\hat u(\xi_1), \hat v(\xi_2),\hat w(\xi_3))$, contracts the resulting tensor against 
$\overline{\hat \psi_1(\xi_1)} \otimes \overline{\hat \psi_2(\xi_2)} \otimes \overline{\hat \psi_3(\xi_3)}$ and then integrates in $\xi_1,\xi_2,\xi_3$ (absorbing the $\overline{\hat \psi_i}$ and $F$ factors into the $m_{j,\omega}$ terms, after first breaking $F$ into $27$ components), we obtain the desired decomposition \eqref{couv} (after replacing $\Omega$ with the disjoint union of $27$ copies of $\Omega$ to accommodate the contributions from the various components of $F$).  As before, one may wish to first work with Schwartz $u,v,w$ to justify the interchanges of integrals, and then take limits at the end of the argument.

\subsection{Sixth step: simplifying the weights}

It remains to obtain the decomposition \eqref{bigmess}.  We will restrict attention to those rotations $R_{j,\omega}$ which almost fix $\xi_j^0$ in the sense that
\begin{equation}\label{roxi}
 |R_{j,\omega} \xi_j^0 - \xi_j^0| < \epsilon_0^2/2
\end{equation}
for $j=1,2,3$.  With this restriction, the weight $\varphi(\frac{1}{\epsilon_0^2} (R_{1,\omega} \xi_1 - \xi_1^0)) \eta(|\xi_1|,|\xi_2|,|\xi_3|)$ is equal to one (for $\epsilon_0$ small enough), and so \eqref{bigmess} simplifies to
\begin{equation}\label{bigmess-2}
\begin{split}
X_1 \otimes X_2 \otimes X_3 &= \int_\Omega \delta( R_{1,\omega} \xi_1 + R_{2,\omega} \xi_2 + R_{3,\omega} \xi_3 ) F(\omega)
m_{1,\omega}(R_{1,\omega} \xi_1) m_{2,\omega}(R_{2,\omega} \xi_2) m_{3,\omega}(R_{3,\omega} \xi_3) \times \\
&\quad 
\Lambda_{R_{1,\omega} \xi_1, R_{2,\omega} \xi_2, R_{3,\omega} \xi_3}( R_{1,\omega} X_1, R_{2,\omega} X_2, R_{3,\omega} X_3)\ d\mu(\omega). 
\end{split}
\end{equation}
Let 
$$ \Sigma \subset \SO(3) \times \SO(3) \times \SO(3) \times \R^3 \times \R^3 \times \R^3$$ 
denote the set of sextuples $(R_1,R_2,R_3,\xi_1,\xi_2,\xi_3)$ where $R_j \in \SO(3)$ with $|R_j \xi_j^0 -\xi_j^0| < \epsilon_0^2/4$ for $j=1,2,3$, and $\xi_i \in B( \xi_i^0, 2\epsilon_0^3 )$ for $i=1,2,3$ with
$$ R_1 \xi_1 + R_2 \xi_2 + R_3 \xi_3 = 0.$$
For $\epsilon_0$ small enough, we see from the implicit function theorem that this is a smooth manifold (of dimension $15$), and that for any choice of $\xi_j \in B( \xi_j^0, 2\epsilon_0^3 )$ for $j=1,2,3$, the slice
$$ \Sigma_{\xi_1,\xi_2,\xi_3} := \{ (R_1,R_2,R_3): (R_1,R_2,R_3,\xi_1,\xi_2,\xi_3) \in \Sigma \}$$
is a smooth manifold (of dimension $6$).

Suppose that we can find a smooth function
$$ F': \Sigma \to \C^3 \otimes \C^3 \otimes \C^3$$
such that we have the identity
\begin{equation}\label{bigmess-3}
\begin{split}
X_1 \otimes X_2 \otimes X_3 &= \int_{\Sigma_{\xi_1,\xi_2,\xi_3}} F'( R_1,R_2,R_3, \xi_1,\xi_2,\xi_3) \times \\
&\quad 
\Lambda_{R_1 \xi_1, R_2 \xi_2, R_3 \xi_3}( R_1 X_1, R_2 X_2, R_3 X_3 )\ d\sigma(R_1,R_2,R_3)
\end{split}
\end{equation}
whenever $\xi_j \in B(\xi_j^0,\epsilon_0^3)$ and $X_j \in \xi_j^\perp$, $j=1,2,3$, where $d\sigma(R_1,R_2,R_3)$ is surface measure on $\Sigma_{\xi_1,\xi_2,\xi_3}$.  By a change of variables, this can be rewritten as
\begin{align*}
X_1 \otimes X_2 \otimes X_3 &= \int_{U} \delta( R_1 \xi_1 + R_2 \xi_2 + R_3 \xi_3 )
\tilde F'( R_1,R_2,R_3, \xi_1,\xi_2,\xi_3) \times \\
&\quad 
\Lambda_{R_1 \xi_1, R_2 \xi_2, R_3 \xi_3}( R_1 X_1, R_2 X_2, R_3 X_3 )\ dR_1 dR_2 dR_3
\end{align*}
where $\tilde F': \Sigma \to \C^3 \otimes \C^3 \otimes \C^3$ is another smooth function ($F'$ multiplied by some Jacobian factors) and $dR_1,dR_2,dR_3$ denote Haar measure on $\SO(3)$, with
$$ U := \{ (R_1,R_2,R_3) \in \SO(3)^3: |R_j \xi_j^0 -\xi_j^0| < \epsilon_0^2/4 \hbox{ for } j=1,2,3 \}.$$
We may smoothly extend $\tilde F'$ to become a smooth compactly supported function on the larger domain
$$ U \times B(\xi_1^0,2\epsilon_0^3) \times B(\xi_2^0,2\epsilon_0^3) \times B(\xi_3^0,2\epsilon_0^3).$$
By a Fourier expansion and another smooth truncation, we may thus write
$$ \tilde F'( R_1,R_2,R_3, \xi_1,\xi_2,\xi_3) = \int_{\R^3 \times \R^3 \times \R^3} f(R_1,R_2,R_3,x_1,x_2,x_3) \prod_{j=1}^3 (e^{2\pi i x_j \cdot \xi_j} m_j(\xi_j))\ dx_1 dx_2 dx_3$$
whenever $(R_1,R_2,R_3) \in U$ and $\xi_j \in B(\xi_j^0,\epsilon^3)$, where $m_j$ is a smooth function supported on $B(\xi_j^0, 3\epsilon_0^3)$, and $f: U \times \R^3 \times \R^3 \times \R^3 \to \C^3 \times \C^3 \times \C^3$ is rapidly decreasing in $x_1,x_2,x_3$, uniformly in $R_1,R_2,R_3$.  Inserting this expansion into \eqref{bigmess-3}, we obtain the desired expansion \eqref{bigmess-2} (taking $\Omega$ to be $U \times \R^3 \times \R^3 \times \R^3$, with $\mu$ being Haar measure weighted by $|f|$, choosing the $m_{j,\omega}$ to be an appropriately rotated version of $m_j$, twisted by a plane wave, and with $F := f/|f|$).

\subsection{Seventh step: restricting to rotations around fixed axes}

It remains to find a smooth function $F'$ for which one has the required representation \eqref{bigmess-3}.  Observe from \eqref{xisum} and the implicit function theorem (for $\epsilon_0$ small enough) that if $\xi_j \in B(\xi_j^0,\epsilon_0^3)$ for $j=1,2,3$, one can find rotations $R_{j,\xi_1,\xi_2,\xi_3} \in \SO(3)$ for $j=1,2,3$ with
$$
 R_{j,\xi_1,\xi_2,\xi_3} = I +  O(\epsilon_0^3) $$
(where $I$ is the identity matrix) and the tuple $(\tilde \xi_1,\tilde \xi_2, \tilde \xi_3)$
defined by
\begin{equation}\label{txdef}
 \tilde \xi_j := R_{j,\xi_1,\xi_2,\xi_3} \xi_j
\end{equation}
lives in the space
\begin{equation}\label{gamma-def}
\Gamma := \{ (\eta_1,\eta_2,\eta_3) \in B(\xi_1^0, C\epsilon_0^3) \times B(\xi_2^0, C\epsilon_0^3) \times B(\xi_3^0, C\epsilon_0^3): \eta_1 + \eta_2 + \eta_3 = 0 \}
\end{equation}
for some absolute constant $C$ independent of $\epsilon_0$.  Furthermore, from the implicit function theorem we may make $R_{j,\xi_1,\xi_2,\xi_3}$ and hence $\tilde \xi_1, \tilde \xi_2, \tilde \xi_3$ depend smoothly on $\xi_1,\xi_2,\xi_3$ in the indicated domain if $\epsilon_0$ is small enough.  If we let $R_\xi^\theta \in \SO(3)$ denote the rotation by $\theta$ around the axis $\xi$ using the right-hand rule\footnote{More precisely, if $u$ is the unit vector $u = \xi/|\xi|$, we define $R_\xi^\theta X := (X \cdot u) u + \cos(\theta) (X - (X \cdot u) u) + \sin(\theta) u \times X$.} for any $\xi \in \R^3 \backslash \{0\}$ and $\theta \in \R/2\pi \Z$, we then see that the six-dimensional manifold
\begin{equation}\label{slor}
\{ ( S R_{\tilde \xi_1}^{\theta_1} R_{1,\xi_1,\xi_2,\xi_3}, S R_{\tilde \xi_2}^{\theta_2} R_{2,\xi_1,\xi_2,\xi_3}, S R_{\tilde \xi_3}^{\theta_3} R_{3,\xi_1,\xi_2,\xi_3}): S \in \SO(3); \theta_1,\theta_2,\theta_3 \in \R/2\pi \Z; \| S - I \| \leq \epsilon^2/8 \}
\end{equation}
(where $\| \|$ denotes the operator norm) is an open submanifold of $\Sigma_{\xi_1,\xi_2,\xi_3}$.  Also, if we use the ansatz
$$ (R_1,R_2,R_3) = ( S R_{\tilde \xi_1}^{\theta_1} R_{1,\xi_1,\xi_2,\xi_3}, S R_{\tilde \xi_2}^{\theta_2} R_{2,\xi_1,\xi_2,\xi_3}, S R_{\tilde \xi_3}^{\theta_3} R_{3,\xi_1,\xi_2,\xi_3})$$
then from \eqref{lambda-def} we see that
$$ \Lambda_{R_1 \xi_1, R_2 \xi_2, R_3\xi_3}(R_1 X_1, R_2 X_2, R_3 X_3) = \Lambda_{\tilde \xi_1, \tilde \xi_2, \tilde \xi_3}\left( 
R_{\tilde \xi_1}^{\theta_1} R_{1,\xi_1,\xi_2,\xi_3} X_1, 
R_{\tilde \xi_2}^{\theta_2} R_{2,\xi_1,\xi_2,\xi_3} X_2,
R_{\tilde \xi_3}^{\theta_2} R_{2,\xi_1,\xi_2,\xi_3} X_3 \right)$$
for $X_j \in \xi_j^\perp$, $j=1,2,3$.
Thus, if we can find a smooth function 
$$F'': \R/2\pi\Z \times \R/2\pi\Z \times \R/2\pi\Z \times \Gamma \to \C^3 \otimes \C^3 \otimes \C^3$$
with the property that
\begin{equation}\label{bigmess-4}
\begin{split}
Y_1 \otimes Y_2 \otimes Y_3 &= \int_{ \R/2\pi\Z \times \R/2\pi\Z \times \R/2\pi\Z} F''( \theta_1,\theta_2,\theta_3, \eta_1, \eta_2, \eta_3 ) \times \\
&\quad 
\Lambda_{\eta_1,\eta_2,\eta_3}( R_{\eta_1}^{\theta_1} Y_1, R_{\eta_2}^{\theta_2} Y_2, R_{\eta_3}^{\theta_3} Y_3 )\ d\theta_1 d\theta_2 d\theta_3
\end{split}
\end{equation}
for all $(\eta_1,\eta_2,\eta_3) \in \Gamma$ and $Y_j \in \eta_j^\perp$ for $j=1,2,3$, then by substituting $Y_j = R_{j,\xi_1,\xi_2,\xi_3} X_j$ and $\eta_j = \tilde \xi_j$, we have
\begin{align*}
R_{1,\xi_1,\xi_2,\xi_3} X_1 \otimes R_{2,\xi_1,\xi_2,\xi_3} X_2 \otimes R_{3,\xi_1,\xi_2,\xi_3} X_3 &= \int_{\R/2\pi\Z \times \R/2\pi\Z \times \R/2\pi\Z} F''( \theta_1,\theta_2,\theta_3, \tilde \xi_1, \tilde \xi_2, \tilde \xi_3 ) \times \\
&\quad 
\Lambda_{R_1 \xi_1, R_2 \xi_2, R_3 \xi_3}( R_1 X_1, R_2 X_2, R_3 X_3 )\ d\theta_1 d\theta_2 d\theta_3
\end{align*}
for any $\xi_j \in B(\xi_j^0,\epsilon^3)$ and $X_j \in \xi_j^\perp$.  Averaging this over all $S \in \SO(3)$ with $\| S - I \| \leq \epsilon^2/8$, and inverting the tensored rotation operator $R_{1,\xi_1,\xi_2,\xi_3} \otimes R_{2,\xi_1,\xi_2,\xi_3} \otimes R_{3,\xi_1,\xi_2,\xi_3}$, we obtain a representation of the desired form \eqref{bigmess-3}.  Thus it suffices to find a smooth function $F''$ with the representation \eqref{bigmess-4}.

\subsection{Eighth step: parameterising in terms of rotation angles}

Note that if $(\eta_1,\eta_2,\eta_3) \in \Gamma$, then the vectors $\eta_1,\eta_2,\eta_3$ are coplanar, and so we may find a unit vector $n = n(\eta_1,\eta_2,\eta_3)$ orthogonal to all of the $\eta_i$; by the implicit function theorem we may ensure that $n$ depends smoothly on $\eta_1,\eta_2,\eta_3$.  From \eqref{xi-split} we may normalise $n$ to be close to $(0,0,1)$ (as opposed to close to $(0,0,-1)$). To prove \eqref{bigmess-4}, it suffices by homogeneity to consider the case when $Y_1,Y_2,Y_3$ are unit vectors; as $Y_j \in \eta_j^\perp$, this means that we may write $Y_j = R_{\eta_j}^{\alpha_j} n$ for some $\alpha_j \in \R/2\pi\Z$ for all $j=1,2,3$.  We may thus rewrite \eqref{bigmess-4} as the claim that
\begin{equation}\label{bigmess-5}
\begin{split}
R_{\eta_1}^{\alpha_1} n \otimes R_{\eta_2}^{\alpha_2} n \otimes R_{\eta_3}^{\alpha_3} n &= \int_{ \R/2\pi\Z \times \R/2\pi\Z \times \R/2\pi\Z} F''( \theta_1,\theta_2,\theta_3, \eta_1, \eta_2, \eta_3 ) \times \\
&\quad \Theta_{\eta_1,\eta_2,\eta_3}( \theta_1 + \alpha_1, \theta_2 + \alpha_2, \theta_3 + \alpha_3 )\ d\theta_1 d\theta_2 d\theta_3
\end{split}
\end{equation}
for all $(\eta_1,\eta_2,\eta_3) \in \Gamma$ and $\alpha_1,\alpha_2,\alpha_3 \in \R/2\pi\Z$, where $\Theta_{\eta_1,\eta_2,\eta_3}: (\R/2\pi\Z)^3 \to \R$ is the function
\begin{equation}\label{theta-def} \Theta_{\eta_1,\eta_2,\eta_3}(\gamma_1,\gamma_2,\gamma_3) :=
\Lambda_{\eta_1,\eta_2,\eta_3}( R_{\eta_1}^{\gamma_1} n, R_{\eta_2}^{\gamma_2} n, R_{\eta_3}^{\gamma_3} n ).
\end{equation}
Note that for fixed $\eta_1,\eta_2,\eta_3$ and each $j=1,2,3$, each of the three coefficients of $R_{\eta_j}^{\alpha_j} n \in \R^3$ is a complex linear combination of $e^{-i \alpha_j}$ and $e^{ i \alpha_j}$, with coefficients depending smoothly on $\eta_1,\eta_2,\eta_3$.  Thus to show \eqref{bigmess-5}, it suffices to obtain a representation
\begin{equation}\label{bigmess-6}
\begin{split}
e^{i (\sigma_1 \alpha_1 + \sigma_2 \alpha_2 + \sigma_3 \alpha_3)} &= \int_{ \R/2\pi\Z \times \R/2\pi\Z \times \R/2\pi\Z} F_{\sigma_1,\sigma_2,\sigma_3}( \theta_1,\theta_2,\theta_3, \eta_1, \eta_2, \eta_3 ) \times \\
&\quad \Theta_{\eta_1,\eta_2,\eta_3}( \theta_1 + \alpha_1, \theta_2 + \alpha_2, \theta_3 + \alpha_3 )\ d\theta_1 d\theta_2 d\theta_3
\end{split}
\end{equation}
for all eight choices of sign patterns $(\sigma_1,\sigma_2,\sigma_3) \in \{-1,+1\}^3$, and some smooth functions
$$ F_{\sigma_1,\sigma_2,\sigma_3}: \R/2\pi\Z \times \R/2\pi\Z \times \R/2\pi\Z \times \Gamma \to \C.$$

\subsection{Ninth step: Fourier inversion and checking a non-degeneracy condition}

By \eqref{theta-def}, \eqref{lambda-def} and decomposing $R_{\eta_j}^{\gamma_j} n$ into a complex linear combination of $e^{-i \gamma_j}$ and $e^{i \gamma_j}$, we see that for fixed $\eta_1,\eta_2,\eta_3$, we may expand 
\begin{equation}\label{cdoc}
 \Theta_{\eta_1,\eta_2,\eta_3}( \gamma_1,\gamma_2,\gamma_3 ) = \sum_{(\sigma_1,\sigma_2,\sigma_3) \in \{-1,+1\}^3} c_{\sigma_1,\sigma_2,\sigma_3}(\eta_1,\eta_2,\eta_3) e^{i (\sigma_1 \gamma_1 + \sigma_2 \gamma_2 + \sigma_3 \gamma_3)} 
 \end{equation}
for some smooth coefficients $c_{\sigma_1,\sigma_2,\sigma_3}: \Gamma \to \C$.  From the Fourier inversion formula on $(\R/2\pi\Z)^3$, we thus obtain \eqref{bigmess-6} as long as we have the non-degeneracy condition
\begin{equation}\label{c-nondeg}
c_{\sigma_1,\sigma_2,\sigma_3}(\eta_1,\eta_2,\eta_3) \neq 0
\end{equation}
for all $(\eta_1,\eta_2,\eta_3) \in \Gamma$ and all choices of signs $(\sigma_1,\sigma_2,\sigma_3) \in \{-1,+1\}^3$.

For this, we finally need to use the precise form of $\Lambda$.  From \eqref{theta-def}, \eqref{lambda-def} we can write $\Theta_{\eta_1,\eta_2,\eta_3}( \gamma_1,\gamma_2,\gamma_3 )$ as
$$
(R_{\eta_1}^{\gamma_1} n \cdot \eta_2) (R_{\eta_2}^{\gamma_2} n \cdot R_{\eta_3}^{\gamma_3} n) + (R_{\eta_2}^{\gamma_2} n \cdot \eta_1) (R_{\eta_1}^{\gamma_1} n \cdot R_{\eta_3}^{\gamma_3} n)$$
which we expand further as
\begin{align*}
& \sin(\gamma_1) \left(\left(u_1 \times n\right) \cdot \eta_2\right) \left(\cos(\gamma_2) \cos(\gamma_3) + \left(u_2 \cdot u_3\right) \sin(\gamma_2) \sin(\gamma_3)\right)\\
& \quad +
\sin(\gamma_2) \left(\left(u_2 \times n\right) \cdot \eta_1\right) \left(\cos(\gamma_1) \cos(\gamma_3) + \left(u_1 \cdot u_3\right) \sin(\gamma_1) \sin(\gamma_3)\right)
\end{align*}
where $u_i := \eta_i/|\eta_i|$.  Expanding 
\begin{equation}\label{sin}
\sin(\gamma) = \frac{1}{2i} (e^{i\gamma} - e^{-i\gamma}) = \frac{1}{2i} \sum_{\sigma = \pm 1} \sigma e^{i \sigma \gamma}
\end{equation}
and 
$$\cos(\gamma) = \frac{1}{2} (e^{i\gamma} + e^{-i\gamma}) = \frac{1}{2} \sum_{\sigma = \pm 1} e^{i \sigma \gamma}$$ 
we have
$$ \cos(\gamma_2) \cos(\gamma_3) = \frac{1}{4} \sum_{\sigma_2,\sigma_3 = \pm 1} e^{i (\sigma_2 \gamma_2 + \sigma_3 \gamma_3)}$$
and
$$ \sin(\gamma_2) \sin(\gamma_3) = - \frac{1}{4} \sum_{\sigma_2,\sigma_3 = \pm 1} \sigma_2 \sigma_3 e^{i (\sigma_2 \gamma_2 + \sigma_3 \gamma_3)}$$
(the minus sign arising here from the $i$ in the denominator in \eqref{sin}).  Similarly with $\sigma_2, \gamma_2$ replaced by $\sigma_1, \gamma_1$ respectively.  Inserting these expansions and comparing with \eqref{cdoc}, we conclude that
$$ c_{\sigma_1,\sigma_2,\sigma_3}(\eta_1,\eta_2,\eta_3) =\frac{1}{8i}
\left( \left(\left(u_1 \times n\right) \cdot \eta_2\right) \sigma_1 \left(1 - \left(u_2 \cdot u_3\right) \sigma_2 \sigma_3 \right) + 
\left(\left(u_2 \times n\right) \cdot \eta_1\right) \sigma_2 \left(1 - \left(u_1 \cdot u_3\right) \sigma_1 \sigma_3 \right) \right).$$
But by \eqref{gamma-def}, $\eta_j = \xi_j^0 + O(\epsilon_0^3)$, which from \eqref{xi-split} implies that
\begin{align*}
(u_1 \cdot n) \cdot \eta_2 &= -1 + O(\epsilon_0) \\
(u_2 \cdot n) \cdot \eta_1 &= \frac{1}{\sqrt{2}} + O(\epsilon_0) \\
u_2 \cdot u_3 &= -\frac{1}{\sqrt{2}} + O(\epsilon_0) \\
u_1 \cdot u_3 &= O(\epsilon_0)
\end{align*}
and thus
$$ c_{\sigma_1,\sigma_2,\sigma_3}(\eta_1,\eta_2,\eta_3) = -\frac{1}{8i} ( - \sigma_1 + \frac{1}{\sqrt{2}} \sigma_2 + \frac{1}{\sqrt{2}} \sigma_1 \sigma_2 \sigma_3 ) + O(\epsilon_0).$$
As $- \sigma_1 + \frac{1}{\sqrt{2}} \sigma_2 + \frac{1}{\sqrt{2}} \sigma_1 \sigma_2 \sigma_3$ is bounded away from zero for $\sigma_1,\sigma_2,\sigma_3 \in \{-1,+1\}$, the non-degeneracy claim \eqref{c-nondeg} follows for $\epsilon_0$ small enough.  This concludes the proof of Theorem \ref{avg}.

\begin{remark}  The averaging over dilation operators was only needed to place the base frequencies $\xi^0_1,\xi^0_2,\xi^0_3$ in a location where the non-degeneracy condition \eqref{c-nondeg} held.  This condition in fact holds for generic $\xi^0_1,\xi^0_2,\xi^0_3$, and so even without the use of averaging over dilations it should be the case that \emph{most} local cascade operators are expressible as averaged Euler operators.  As there is some freedom to select the local cascade operators in Theorem \ref{blowup}, this should still be enough to establish a slightly stronger version of Theorem \ref{main} in which one does not use any averaging over dilations.  We will however not pursue this matter here.
\end{remark}
	 
\section{Reduction to an infinite-dimensional ODE}

We now begin the proof of Theorem \ref{blowup}.  We fix $0 < \epsilon_0 < 1$; henceforth we allow all implied constants in the $O()$ notation to depend on $\epsilon_0$.  We suppose that Theorem \ref{blowup} failed, so that one can always construct\footnote{This hypothesis of global existence is technically convenient so that we may assume some \emph{a priori} regularity on our solution, namely $H^{10}$.  Alternatively, one could develop an $H^{10}$ local well-posedness theory for \eqref{system}, and unconditionally construct a mild $H^{10}$ solution that blows up in a finite time by a minor modification of the arguments in this paper; we leave the details of this variant of the argument to the interested reader.}
 global mild solutions to any initial value problem of the form \eqref{system} with $C$ a local cascade operator and $u_0$ a Schwartz divergence-free vector field.

To apply this hypothesis, we need to construct a local cascade operator $C$ and an initial velocity field $u_0$.  We need a dimension parameter $m$, which will be a positive integer (eventually we will set $m=4$).  Let $B_1,\dots,B_m$ be balls in the annulus $\{ \xi \in \R^3: 1 < |\xi| \leq 1+\epsilon_0/2 \}$, chosen so that the $2m$ balls $B_1,\dots,B_m,-B_1,\dots,-B_m$ are all disjoint. For each $i=1,\dots,m$, let $\psi_i \in H^{10}_\df(\R^3)$ be Schwartz with Fourier transform real-valued and supported\footnote{We need to have $\hat \psi_i$ supported on $B_i \cup -B_i$ rather than just $B_i$, otherwise we could not require $\psi_i$ to be real.} on $B_i \cup -B_i$, normalised so that $\|\psi_i\|_{L^2(\R^3)} = 1$.  


As in Definition \ref{cascdef}, we define the rescaled functions
$$ \psi_{i,n}(x) := (1+\epsilon_0)^{3n/2} \psi_i\left( (1+\epsilon_0)^n x \right)$$
for $i=1,\dots,m$ and $n \in \Z$, and then define the local cascade operator $C$ by the formula
\begin{equation}\label{cudef}
\begin{split}
C(u,v) &:= \sum_{n \in \Z} \sum_{(i_1,i_2,i_3,\mu_1,\mu_2,\mu_3) \in \{1,\dots,m\}^3 \times S}\\
&\quad
\alpha_{i_1,i_2,i_3,\mu_1,\mu_2,\mu_3} (1+\epsilon_0)^{5n/2} \langle u, \psi_{i_1,n+\mu_1} \rangle \langle v, \psi_{i_2,n+\mu_2} \rangle \psi_{i_3,n+\mu_3} 
\end{split}
\end{equation}
for $u,v \in H^{10}_\df(\R^3)$,
where $S \subset \Z^3$ is the four-element set 
$$ S := \{ (0,0,0), (1,0,0), (0,1,0), (0,0,1) \},$$
the $\alpha_{i_1,i_2,i_3,\mu_1,\mu_2,\mu_3} \in \R$ are structure constants to be chosen later, and which obey the symmetry condition
\begin{equation}\label{symmetry}
 \alpha_{i_1,i_2,i_3,\mu_1,\mu_2,\mu_3} = \alpha_{i_2,i_1,i_3,\mu_2,\mu_1,\mu_3}
\end{equation}
for $(i_1,i_2,i_3,\mu_1,\mu_2,\mu_3) \in S$.  From Definition \ref{cascdef} we see that $C$ is indeed a local cascade operator (it is a sum of $|S| = 7m^3$ basic local cascade operators), and \eqref{symmetry} ensures that $C$ is symmetric.  Clearly
\begin{align*}
 \langle C(u,u),u \rangle &= \sum_{n \in \Z} \sum_{(i_1,i_2,i_3,\mu_1,\mu_2,\mu_3) \in \{1,\dots,m\}^3 \times S}\\
&\quad
\alpha_{i_1,i_2,i_3,\mu_1,\mu_2,\mu_3} (1+\epsilon_0)^{5n/2} \langle u, \psi_{i_1,n+\mu_1} \rangle \langle u, \psi_{i_2,n+\mu_2} \rangle \langle u, \psi_{i_3,n+\mu_3} \rangle
\end{align*}
for $u \in H^{10}_\df(\R^3)$.  From this, we see that the cancellation condition \eqref{cancelled} will follow from the cancellation conditions 
\begin{equation}\label{cyclic}
\sum_{\{a,b,c\} = \{1,2,3\}} \alpha_{i_a,i_b,i_c,\mu_a,\mu_b,\mu_c} = 0
\end{equation}
for all $i_1,i_2,i_3 \in \{1,\dots,m\}$ and $(\mu_1,\mu_2,\mu_3) \in S$.

We will select initial data $u_0$ of the form\footnote{Our analysis is in fact somewhat stable, and will also apply if $u_0$ is a sufficiently small perturbation of $\psi_{1,n_0}$ in the $H^{10}_\df$ norm, thus creating blowup  for a non-empty open set of initial data in smooth topologies, although this open set is rather small and is also quite far from the origin (due to the large nature of $n_0$).  We leave the details of this modification to the interested reader.}
\begin{equation}\label{upsi}
 u_0 := \psi_{1,n_0}
\end{equation}
for some sufficiently large\footnote{Alternatively (and equivalently), one could hold $n_0$ fixed (e.g. $n_0=0$), and rescale the viscosity $\nu$ to be small, thus one is now studying the equation $\partial_t u = \nu \Delta u + C(u,u)$ with some small $\nu>0$.  One can then repeat all the arguments below, basically with $\nu$ playing the role of the quantity $(1+\epsilon_0)^{-n_0/2}$ that will make a prominent appearance in later sections.  We leave the details of this variant of the argument to the interested reader.}  integer $n_0$ to be chosen later.  This is clearly a Schwarz divergence-free vector field.  By hypothesis, we thus have a global mild solution $u: [0,+\infty) \to H^{10}_\df(\R^3)$ to the system \eqref{system}.  We record some basic properties of this solution here:

\begin{lemma}[Equations of motion]\label{eqmot}  Let $C$ be a cascade operator of the form \eqref{cudef}, with coefficients $\alpha_{i_1,i_2,i_3,\mu_1,\mu_2,\mu_3}$ obeying the symmetry \eqref{symmetry} and cancellation property \eqref{cyclic}. Let $u: [0,+\infty) \to H^{10}_\df(\R^3)$ be a global mild solution to the equation \eqref{system} with initial data given by \eqref{upsi} for some $n_0$.  For each $n \in \Z$, $t \geq 0$ and $i=1,\dots,m$, let $u_{i,n}(t)$ be the Fourier projection of $u(t)$ to the region $(1+\epsilon_0)^n \cdot (B_i \cup -B_i)$, thus
$$ \widehat{u_{i,n}(t)}(\xi) = \widehat{u(t)}(\xi) 1_{\xi \in (1+\epsilon_0)^n \cdot (B_i \cup -B_i)}$$
and then define the coefficients
$$ X_{i,n}(t) := \langle u(t), \psi_{i,n} \rangle = \langle u_{i,n}(t),\psi_{i,n} \rangle$$
and the local energies
\begin{equation}\label{ein-def}
 E_{i,n}(t) := \frac{1}{2} \|u_{i,n}(t)\|_{L^2(\R^3)}^2.
\end{equation}
\begin{itemize}
\item[(i)]  (A priori regularity) We have
\begin{equation}\label{slog}
 \sup_{0 \leq t\leq T} \sup_{n \in\Z} \sup_{i=1,\dots,m} (1 + (1+\epsilon_0)^{10n}) |X_{i,n}(t)| < \infty
\end{equation}
and 
\begin{equation}\label{slog-2}
 \sup_{0 \leq t\leq T} \sup_{n \in\Z} \sup_{i=1,\dots,m} (1 + (1+\epsilon_0)^{10n}) E_{i,n}(t)^{1/2} < \infty
\end{equation}
for all $0 < T < \infty$.
\item[(ii)] (Initial conditions) For any $n \in \Z$ and $i=1,\dots,m$, we have
\begin{equation}\label{ein-init}
 E_{i,n}(0) = \frac{1}{2} X_{i,n}(0)^2
\end{equation}
and
\begin{equation}\label{xin-init}
 X_{i,n}(0) = 1_{(i,n) = (1,n_0)}.
\end{equation}
\item[(iii)]  (Equations of motion) For any $n \in\Z$ and $i=1,\dots,m$, we have the equation of motion
\begin{equation}\label{xin}
 \partial_t X_{i,n} = \sum_{i_1,i_2 \in \{1,\dots,m\}} \sum_{(\mu_1,\mu_2,\mu_3) \in S}
\alpha_{i_1,i_2,i,\mu_1,\mu_2,\mu_3} (1+\epsilon_0)^{5(n-\mu_3)/2} X_{i_1,n-\mu_3+\mu_1} X_{i_2,n-\mu_3+\mu_2} + O\left( (1+\epsilon_0)^{2n} E_{i,n}^{1/2} \right).
\end{equation}
and the energy inequality
\begin{equation}\label{ein}
 \partial_t E_{i,n} \leq \sum_{i_1,i_2 \in \{1,\dots,m\}} \sum_{(\mu_1,\mu_2,\mu_3) \in S} \alpha_{i_1,i_2,i,\mu_1,\mu_2,\mu_3} (1+\epsilon_0)^{5(n-\mu_3)/2} X_{i_1,n-\mu_3+\mu_1} X_{i_2,n-\mu_3+\mu_2} X_{i,n}
\end{equation}
for all $t \geq 0$.
\item[(iv)] (Energy defect) For any $n \in \Z$ and $i=1,\dots,m$, we have
\begin{equation}\label{exin}
\frac{1}{2} X_{i,n}^2(t)  \leq E_{i,n}(t) \leq \frac{1}{2} X_{i,n}^2(t) + O\left( (1+\epsilon_0)^{2n} \int_0^t E_{i,n}(t')\ dt' \right)
\end{equation}
for all $t \geq 0$.
\item[(v)] (No very low frequencies)  One has
\begin{equation}\label{nolo}
X_{i,n}(t) = E_{i,n}(t) = 0
\end{equation}
for all $n < n_0$, $i=1,\dots,m$, and $t \geq 0$.
\end{itemize}
\end{lemma}

\begin{proof}  As $u$ is a mild solution to \eqref{system}, we have
\begin{equation}\label{uko}
 u(t) = e^{t\Delta} u_0 + \int_0^t e^{(t-t')\Delta} C( u(t'), u(t') )\ dt'
\end{equation}
for all $t \geq 0$.  Taking Fourier transforms, we see in particular that $\hat u(t)$ is supported on the union $\bigcup_{n \in \Z} \bigcup_{i=1}^m \bigcup_{\mu = \pm 1} \mu (1+\epsilon_0)^n \cdot B_i$ of dilations of the balls $\pm B_1,\ldots,\pm B_m$.  As these dilated balls are disjoint, we thus have a decomposition
$$ u(t) = \sum_{n \in \Z} \sum_{i=1}^m u_{i,n}(t)$$
(which is unconditionally convergent in $H^{10}$).  If we define the scalar functions $X_{i,n}: [0,+\infty) \to \C$ by the formula
$$ X_{i,n}(t) := \langle u(t), \psi_{i,n} \rangle = \langle u_{i,n}(t),\psi_{i,n} \rangle,$$
then from the \emph{a priori} regularity $u \in C^0_t H^{10}_x$ we obtain \eqref{slog} from the Plancherel identity.  Taking inner products of \eqref{uko} with $\psi_{i,n}$, we have
\begin{align*}
 u_{i,n}(t) &= e^{t\Delta} X_{i,n}(0) \psi_{i,n} + 
 \sum_{i_1,i_2 \in \{1,\dots,m\}} \sum_{(\mu_1,\mu_2,\mu_3) \in S} \\
&\quad\quad
\alpha_{i_1,i_2,i,\mu_1,\mu_2,\mu_3} (1+\epsilon_0)^{5(n-\mu_3)/2} \int_0^t X_{i_1,n-\mu_3+\mu_1}(t') X_{i_2,n-\mu_3+\mu_2}(t') e^{(t-t')\Delta} \psi_{i,n}\ dt'
\end{align*}
or in differentiated form (using \eqref{slog} to justify the calculations)
\begin{equation}\label{upp}
\partial_t u_{i,n} = \Delta u_{i,n} +  \sum_{i_1,i_2 \in \{1,\dots,m\}} \sum_{(\mu_1,\mu_2,\mu_3) \in S}
\alpha_{i_1,i_2,i,\mu_1,\mu_2,\mu_3} (1+\epsilon_0)^{5(n-\mu_3)/2} X_{i_1,n-\mu_3+\mu_1} X_{i_2,n-\mu_3+\mu_2} \psi_{i,n}.
\end{equation}
In particular this shows that $u_{i,n}$ is continuously differentiable in time (in the $L^2_x$ topology, say), which implies that the $X_{i,n}$ are continously differentiable.

It is unfortunate that the $\psi_{i,n}$ are not eigenfunctions of the Laplacian $\Delta$, otherwise $u_{i,n}$ would be always be a scalar multiple of $\psi_{i,n}$ (that is, $u_{i,n} = X_{i,n} \psi_{i,n}$), and the equation \eqref{system} would collapse to a system of ODE in the $X_{i,n}$ variables.  However, it is still possible to get good control on the dynamics even without the eigenfunction property.  To do this, we use the local energies $E_{i,n}$ from \eqref{ein-def}.  From Cauchy-Schwarz we have
\begin{equation}\label{loco}
\frac{1}{2} X_{i,n}(t)^2 \leq E_{i,n}(t),
\end{equation}
and from Plancherel and the $C^0_t H^{10}_x$ bound on $u$ we have \eqref{slog-2}
for all $0 < T < \infty$.

By taking inner products of \eqref{upp} with $u_{i,n}$, and noting that
$$ \langle \Delta u_{i,n}, u_{i,n} \rangle \leq 0$$
we obtain the \emph{local energy inequality} \eqref{ein}.  Indeed, one could use Fourier analysis to place an additional dissipation term of $8 \pi^2 (1+\epsilon_0)^{2n} E_{i,n}$ on the right-hand side of \eqref{ein}, but we will not need to use this term here (it is too small to be of much use, since we are in the regime where dissipation can be treated as a negligible perturbation).

If instead, if we take inner products of \eqref{upp} with $\psi_{i,n}$, and note that
\begin{align*}
\langle \Delta u_{i,n}, \psi_{i,n} \rangle &= \langle u_{i,n}, \Delta \psi_{i,n} \rangle \\
&= O\left( E_{i,n}^{1/2} \|\Delta \psi_{i,n} \|_{L^2(\R^3)}^2 \right) \\
&= O\left( (1+\epsilon_0)^{2n} E_{i,n}^{1/2} \right)
\end{align*}
we conclude \eqref{xin}.

From \eqref{xin}, \eqref{ein} we see that
$$
 \partial_t \left(E_{i,n} - \frac{1}{2} X_{i,n}^2\right) \leq O\left( (1+\epsilon_0)^{2n} E_{i,n} \right)
$$
while from \eqref{ein-init} we see that $E_{i,n} - \frac{1}{2} X_{i,n}^2$ vanishes at time zero.  The claim \eqref{exin} then follows from \eqref{loco} and the fundamental theorem of calculus.

Finally, we prove \eqref{nolo}.  For $i_3=1,\dots,m$ and $n < n_0$, we see from \eqref{ein}, \eqref{ein-init}, \eqref{xin-init} and the fundamental theorem of calculus that
$$
E_{i_3,n}(t) \leq \sum_{i_1,i_2 \in \{1,\dots,m\}} \sum_{(\mu_1,\mu_2,\mu_3) \in S} \alpha_{i_1,i_2,i_3,\mu_1,\mu_2,\mu_3} (1+\epsilon_0)^{5(n-\mu_3)/2} \int_0^t X_{i_1,n-\mu_3+\mu_1} X_{i_2,n-\mu_3+\mu_2} X_{i_3,n}(t')\ dt'$$
for any $t \geq 0$.  Summing this for $i=1,\dots,m$ and $n < n_0$, and using \eqref{slog}, \eqref{slog-2} to ensure all summations and integrals are absolutely convergent, we conclude that
\begin{align*}
\sum_{n<n_0} \sum_{i=1}^m E_{i,n}(t) &\leq \sum_{i_1,i_2,i_3 \in \{1,\dots,m\}} \sum_{(\mu_1,\mu_2,\mu_3) \in S} \sum_{n < n_0}\\
&\quad \alpha_{i_1,i_2,i_3,\mu_1,\mu_2,\mu_3}  (1+\epsilon_0)^{5(n-\mu_3)/2} \int_0^t X_{i_1,n-\mu_3+\mu_1} X_{i_2,n-\mu_3+\mu_2} X_{i_3,n}(t')\ dt'.
\end{align*}
By \eqref{cyclic}, all the terms here can be grouped into terms that sum to zero, except for those terms with $n=n_0-1$, $(\mu_1,\mu_2,\mu_3) \in \{ (1,0,0), (0,1,0) \}$; thus
\begin{align*}
\sum_{n<n_0} \sum_{i=1}^m E_{i,n}(t) &\leq \sum_{i_1,i_2,i_3 \in \{1,\dots,m\}} \sum_{(\mu_1,\mu_2,\mu_3) \in \{ (1,0,0), (0,1,0) \}} 
\alpha_{i_1,i_2,i_3,\mu_1,\mu_2,\mu_3}  (1+\epsilon_0)^{5(n_0-1-\mu_3)/2} \\
&\quad \int_0^t X_{i_1,n_0-1-\mu_3+\mu_1} X_{i_2,n_0-1-\mu_3+\mu_2} X_{i_3,n_0-1}(t')\ dt'.
\end{align*}
By the constraint on $(\mu_1,\mu_2,\mu_3)$, two of the terms $X_{i_1,n_0-1-\mu_3+\mu_1}$, $X_{i_2,n_0-1-\mu_3+\mu_2}$, $X_{i_3,n_0-1}$ may be bounded by $\sum_{n<n_0} \sum_{i=1}^m E_{i,n}$, and the remaining term may be controlled by \eqref{slog}, leading to the bound
$$
\sum_{n<n_0} \sum_{i=1}^m E_{i,n}(t) \leq C_{T,n_0} \int_0^{t'}\sum_{n<n_0} \sum_{i=1}^m E_{i,n}(t')\ dt'
$$
for all $0 \leq t \leq T$ and some finite quantity $C_{T,n_0}$ depending on $T, n_0$ (and on the quantity in \eqref{slog}).  By Gronwall's inequality, we conclude that $\sum_{n<n_0} \sum_{i=1}^m E_{i,n}(t) =0$ for all $t \geq 0$, giving \eqref{nolo}.
\end{proof}

The above lemma shows that \eqref{system} almost collapses into an ODE system for the $X_{i,n}$.  As a first approximation, the reader may wish to ignore the role of the energies $E_{i,n}$ (or identify them with $\frac{1}{2} X_{i,n}^2$), and pretend that \eqref{xin} is replaced by either the inviscid equation
\begin{equation}\label{inviscid-model}
 \partial_t X_{i,n} = \sum_{i_1,i_2 \in \{1,\dots,m\}} \sum_{(\mu_1,\mu_2,\mu_3) \in S}
\alpha_{i_1,i_2,i,\mu_1,\mu_2,\mu_3} (1+\epsilon_0)^{5(n-\mu_3)/2} X_{i_1,n-\mu_3+\mu_1} X_{i_2,n-\mu_3+\mu_2}
\end{equation}
or the viscous equation
\begin{equation}\label{viscous-model}
 \partial_t X_{i,n} = - (1+\epsilon_0)^{2n} X_{i,n} + \sum_{i_1,i_2 \in \{1,\dots,m\}} \sum_{(\mu_1,\mu_2,\mu_3) \in S}
\alpha_{i_1,i_2,i,\mu_1,\mu_2,\mu_3} (1+\epsilon_0)^{5(n-\mu_3)/2} X_{i_1,n-\mu_3+\mu_1} X_{i_2,n-\mu_3+\mu_2}
\end{equation}
in the analysis that follows.  Note that the viscous equation generalises the dyadic Katz-Pavlovic equation \eqref{xn} (with $\lambda=(1+\epsilon_0)^{5/2}$ and $\alpha=2/5$), which corresponds to a simple case in which $m=1$.

Theorem \ref{blowup} now follows from the following ODE result:

\begin{theorem}[ODE blowup]\label{ood}  Let $0 < \epsilon_0 < 1$.  Then there exist a natural number $m \geq 0$, structure constants $\alpha_{i_1,i_2,i_3,\mu_1,\mu_2,\mu_3} \in \R$ for $i_1,i_2,i_3 \in \{1,\dots,m\}$ and $(\mu_1,\mu_2,\mu_3) \in S$ obeying the symmetry condition \eqref{symmetry} and the cancellation condition \eqref{cyclic}, with the property that for sufficiently large $n_0$ (sufficiently large depending on implied constants in \eqref{ein}, \eqref{xin}), there does not exist continuously differentiable functions $X_{i,n}: [0,+\infty) \to \R$ and $E_{i,n}: [0,+\infty) \to [0,+\infty)$ obeying the conclusions \eqref{slog}-\eqref{nolo} of Lemma \ref{eqmot}.
\end{theorem}

We will prove Theorem \ref{ood} in Section \ref{blowup-sec}, but we first warm up with some finite dimensional ODE toy problems in the next section.

\begin{remark}[Helicity conservation]\label{helicity} As is well known (see e.g. \cite{bert}), the inviscid Euler equations $\partial_t u = B(u,u)$ conserve helicity $\int_{\R^3} u \cdot (\operatorname{curl} u)$.  This is equivalent to the additional cancellation law
\begin{equation}\label{cancel-hel}
\langle B(u,u), \operatorname{curl} u \rangle = 0
\end{equation}
for all $u \in H^{10}_\df(\R^3)$.  One can ask whether we can similarly enforce the cancellation law
\begin{equation}\label{cancel-hel-2}
\langle \tilde B(u,u), \operatorname{curl} u \rangle = 0
\end{equation}
for the averaged operators $\tilde B$.  In general, the operator $C$ defined in \eqref{cudef} will not obey \eqref{cancel-hel-2}.  However, we may still ensure \eqref{cancel-hel-2} (while preserving the other desired properties of $\tilde B$) as follows.  Firstly, observe that we may choose the functions $\psi_1,\psi_2,\psi_3$ in the construction of $C$ to be odd, thus $\psi_i(-x) = -\psi_i(x)$ for all $i=1,2,3$ and $x \in \R^3$.  Next, from \eqref{cudef}, \eqref{cyclic}, and \eqref{symmetry}, we see that the operator $C(u,v)$ in \eqref{cudef} is a finite linear combination of operators of the form 
$$\frac{1}{2} \sum_{n \in \Z} (1+\epsilon_0)^{5n/2} (A_n(u,v) + A_n(v,u)),$$
where
$$ A_n(u,v) := \langle u, \psi'_n \rangle \langle v, \psi''_n \rangle \psi'''_n - \langle u, \psi'_n \rangle \langle v, \psi'''_n \rangle \psi''_n $$
and $\psi', \psi'', \psi''' \in H^{10}_\df(\R^3)$ are odd functions with Fourier transform supported on an annulus.  These operators obey the energy cancellation law $\langle A_n(u,u), u \rangle = 0$, but do not necessarily obey the helicity cancellation law $\langle A_n(u,u), \operatorname{curl} u \rangle = 0$.  However, if we introduce the modified operator 
\begin{align*}
\tilde A_n(u,v) &:= \langle u, \psi'_n \rangle \langle v, \psi''_n \rangle \psi'''_n - \langle u, \psi'_n \rangle \langle v, \psi'''_n \rangle \psi''_n \\
&\quad - \langle u, \psi'_n \rangle \langle v, \operatorname{curl} \psi'''_n \rangle \operatorname{curl}^{-1} \psi''_n + \langle u, \operatorname{curl}^{-1} \psi''_n \rangle \langle v, \operatorname{curl} \psi'''_n \rangle \psi'_n \\
&\quad - \langle u, \operatorname{curl}^{-1} \psi''_n \rangle \langle v, \operatorname{curl} \psi'_n \rangle \psi'''_n + \langle u, \psi'''_n \rangle \langle v, \operatorname{curl} \psi'_n \rangle \operatorname{curl}^{-1} \psi''_n
\end{align*}
where $\operatorname{curl}^{-1} := \Delta^{-1} \operatorname{curl}$ inverts the curl operator on divergence free functions with Fourier support on an annulus, one can check that $\tilde A_n$ obeys both the energy cancellation $\langle \tilde A_n(u,u), u \rangle$ and the helicity cancellation $\langle \tilde A_n(u,u), \operatorname{curl} u \rangle = 0$.  Thus if we define $\tilde C$ by replacing all occurrences of $A_n$ with their counterparts $\tilde A_n$, then $\tilde C$ also obeys energy and helicity cancellation.  Furthermore, observe that $\tilde A_n(u,u) = A_n(u,u)$ is an odd function whenever $u$ is an odd function (basically because the curl or inverse curl of an odd function is even, and thus orthogonal to all odd functions), and so $\tilde C(u,u) = C(u,u)$ when $u$ is an odd function.  It is then easy to see that any mild solution to $\partial_t u = \tilde C(u,u)$ with odd initial data is then odd for all time, and thus also solves $\partial_t u = C(u,u)$.  From this, we see that Theorem \ref{blowup} for $C$ implies Theorem \ref{blowup} for $\tilde C$.  As a consequence, we can enforce helicity conservation in Theorem \ref{main} if desired.  Of course, it was unlikely in any event that global helicity conservation would have been useful for the global regularity problem, given that the helicity of an odd vector field is automatically zero, and that odd vector fields are preserved by the Euler and Navier-Stokes flows, and are not expected to be any more difficult\footnote{Indeed, any non-odd initial data for such flows may be made odd by first translating by a large displacement and then anti-symmetrising, which will asymptotically have no impact on the dynamics after renormalising.} to handle than general vector fields.

Finally, we remark that the Euler equations $\partial_t u = B(u,u)$ formally conserve total momentum $\int_{\R^3} u\ dx$, total angular momentum $\int_{\R^3} x \times u\ dx$, and total vorticity $\int_{\R^3} \nabla \times u\ dx$.  These quantities are also formally conserved by the equation $\partial_t u = C(u,u)$ for any local cascade operator $C$, basically because the wavelets $\psi_1,\psi_2,\psi_3$ used in building these cascade operators have Fourier transform vanishing near the origin; we omit the details.
\end{remark}

\section{Quadratic circuits}

Our objective is to solve an infinite-dimensional system of ODE, roughly of the form \eqref{inviscid-model}.  In order to build up some intuition for doing so, we will first study a finite-dimensional ``toy'' model, namely ODEs of the form
\begin{equation}\label{ode}
\partial_t X = G(X,X)
\end{equation}
where $X: [t_1,t_2] \to \R^m$ is a vector-valued trajectory for some finite $m$, and $G: \R^m \times \R^m \to \R^m$ is a bilinear operator obeying the cancellation condition
\begin{equation}\label{g-cancel}
 G(X,X) \cdot X  = 0
\end{equation}
for all $X \in \R^m$ (so in particular, the flow \eqref{ode} preserves the norm of $X$, and so the ODE is globally well posed).  It will be important for us that there is no size restriction on the coefficients on the bilinear operator $G$, although the coefficients must of course be real.  The terminology ``circuit'' is meant to invoke an analogy with electrical engineering (and also with computational complexity theory).  Clearly, \eqref{ode} is a toy model for the system \eqref{inviscid-model}, and can also be viewed as a toy model for the Euler equations $\partial_t u = B(u,u)$. We will build a quadratic circuit to accomplish a specific task (namely, to abruptly transfer energy from one mode to another, after a delay) out of ``quadratic logic gates'', by which we mean quadratic circuits \eqref{ode} of a very small size (with $m=2$ or $m=3$) and a simple structure to $G$, which each accomplish a single simple task of transforming a certain type of input into a certain type of output.  

We first discuss in turn the three quadratic logic gates we will be using, which we call the ``pump'', the ``amplifier'', and the ``rotor'', and then show how these gates can be combined to build a circuit with the desired properties.  It looks likely that the set of quadratic gates is sufficiently ``Turing complete'' in that they can perform extremely general computational tasks\footnote{Of course, this is bearing in mind that, being globally well-posed ODE, circuits of the form \eqref{ode} are necessarily limited to perform continuous (i.e. analog) operations rather than perfectly digital operations.  Also, as the equation \eqref{ode} is time reversible, only reversible computing tasks may be performed by quadratic circuits, at least in the absence of dissipation.}, but we will not pursue\footnote{See \cite{pour} for a treatment of continuous computation in PDE, and \cite{turing} for continuous computation in ODE.} this matter further here.

Strictly speaking, the discussion here is not actually needed for the proof of our main results, but we believe that the model problems studied here will assist the reader in understanding what may otherwise be a highly unmotivated construction and set of arguments in the next section.

\subsection{The pump gate}

We first describe the \emph{pump gate}.  This is a binary gate (so $m=2$), with unknown $X(t) = (x(t),y(t)) \in \R^2$ obeying the quadratic ODE
\begin{equation}\label{pump}
\begin{split}
\partial_t x &= - \alpha xy \\
\partial_t y &= \alpha x^2
\end{split}
\end{equation}
where $\alpha >0$ is a fixed coupling constant (representing the strength of the pump).  We will be applying this pump in the regime where $x$ is initially positive and $y \geq 0$; by Gronwall's inequality (or by integrating factors), we see that $x$ remains positive for all subsequent time, while $y$ is increasing.  As the total energy $x^2+y^2$ is conserved, we thus see that energy is being pumped from $x$ to $y$.  For instance, we have the explicit solution
\begin{equation}\label{pomp}
x(t) = A \operatorname{sech}(\alpha A t); \quad y(t) = A \operatorname{tanh}( \alpha A t )
\end{equation}
for any amplitude $A>0$, which at time $t=0$ is at the initial state $(x(0),y(0)) = (A,0)$.  For times $0 \leq t \leq \frac{1}{\alpha A}$, the $y$ component increases more or less linearly at rate comparable to $\alpha A^2$, with a corresponding drain of energy from $x$; after this time, $x$ decays exponentially fast (at rate $\alpha A$), with the energy in $x$ being transferred more or less completely to $y$ after time $t \geq \frac{C}{\alpha A}$ for a large constant $C$.  Thus, the pump can be used to execute a delayed, but \emph{gradual}, transition of energy from one mode (the $x$ mode) to another (the $y$ mode).  We will schematically depict the pump by a thick arrow: see Figure \ref{fig:pump}.

\begin{figure} [t]
\centering
\includegraphics{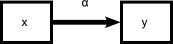}
\caption[Pump gate]{The pump gate from the $x$ mode to the $y$ mode with coupling constant $\alpha$.}
\label{fig:pump}
\end{figure}

If one ignores the dissipation term, the dyadic model equation \eqref{xn} can be viewed as a sequence of pumps chained together, with the coupling constant $\lambda^n$ of the pump from one mode $X_n$ to the next $X_{n+1}$ increasing exponentially with $n$.

One useful feature of the pump which we will exploit is that it can ``integrate'' an alternating input $x$ into a monotone output $y$, somewhat analogously to how a rectifier in electrical engineering converts AC current to DC current.  Indeed, if one couples the $x$ input of the pump to an external forcing term, thus
\begin{align*}
\partial_t x &= - \alpha xy + F\\
\partial_t y &= \alpha x^2
\end{align*}
with $F$ highly oscillatory, then $x$ may oscillate in sign also (if the $F$ term dominates the energy drain term $-\alpha xy$), but the $y$ output continues to increase at a more or less steady rate.  If for instance $F(t) = A\omega \cos(\omega t)$ with some quantities $A, \omega$ which are large compared to the coupling constant $\alpha$, and we set initial conditions $x(0)=y(0)=0$ for simplicity, then we expect $x$ to behave like $A \sin(\omega t)$, and $y$ to increase at rate about $\frac{1}{2} \alpha A^2$ on average.

If instead we couple the pump to an oscillatory forcing term on the output, thus
\begin{align*}
\partial_t x &= - \alpha xy \\
\partial_t y &= \alpha x^2 + G
\end{align*}
then it is possible that $y$ can turn negative, which causes the pump to reverse in energy flow to become an amplifier (see below).  This behaviour will be undesirable for us, so we will take some care to design our circuit so that the output of a pump does not experience significant negative forcing at key epochs in the dynamics, unless this forcing is counterbalanced by an almost equivalent amount of positive forcing.

\subsection{Application: finite time blowup for an exogenously truncated dyadic model}\label{beyond}

As a quick application of the pump gate, we establish blowup for the truncated version \eqref{trunc-non} of the dyadic model system \eqref{xn}, whenever one has supercritical dissipation:

\begin{proposition}[Blowup for a truncated dyadic model]\label{bkp}  Let $\lambda > 1$ and $0 < \alpha < 1/2$, and let $0 < \delta < 1-2\alpha$.  Then there exists a natural number $n_0$, a sequence of times
$$ 0 = t_{n_0} < t_{n_0+1} < t_{n_0+2} < \dots$$
increasing to a finite limit $T_*$, and continuous, piecewise smooth functions $X_n: [0,T_*) \to \R$ for $n \geq n_0$ such that $X_{n_0+k}(t)=0$ whenever $k \geq 1$ and $0 \leq t \leq t_{n_0+k-1}$, and such that
\begin{equation}\label{trunc-non-again}
\partial_t X_n = - \lambda^{2n\alpha} X_n + 1_{(t_{n-1},t_n)}(t) \lambda^{n-1} X_{n-1}^2 - 1_{(t_n,t_{n+1})}(t) \lambda^n X_n X_{n+1}
\end{equation}
for all $t \in [0,T_*)$ other than the times $t_{n_0}, t_{n_1},\ldots$, and all $n \geq n_0$, with the convention that $t_{n_0-1}=0$ and $X_{n_0-1}=0$.  Furthermore, we have
$$ X_{n_0+k}(t_{n_0+k}) = \lambda^{- \delta k}$$
for every $k \geq 0$.  In particular, for any $\delta'>\delta$, we have the blowup
$$\limsup_{t \to T_*} \sup_n \lambda^{\delta'n} |X_n(t)| = +\infty.$$
\end{proposition}

This proposition is not needed for the blowup results in the rest of the paper, but is easier to prove than those results, and already illustrates the basic features of the blowup solutions being constructed.  Note that the blowup here is available for all values of the dissipation parameter up to the critical value of $1/2$, in contrast to the results in \cite{katz-dyadic} and \cite{ches} for the untruncated equation \eqref{xn} which cover the ranges $\alpha < 1/4$ and $\alpha < 1/3$ respectively, as well as the results in \cite{bmr} establishing global solutions when $\lambda=2$ and $2/5 \le \alpha \le 1/2$.

\begin{proof}  We let $n_0$ be a sufficiently large natural number (depending on $\lambda,\alpha,\delta$) to be chosen later.  
We then construct $t_{n_0}, t_{n_0+1},\ldots$ and $X_n(t)$ iteratively as follows:
\begin{itemize}
\item[Step 1.]  Initialise $k=0$ and $t_{n_0}=0$.  We also initialise
$$ X_{n_0}(0)=1; \quad X_{n_0+k}(0) = 0 \hbox{ for all } k \geq 1.$$
\item[Step 2.]  Now suppose that $t_{n_0+k}$ has been constructed, and the solution $X_n(t)$ constructed for all times $0 \leq t \leq t_{n_0+k}$ and $n \geq n_0$.  We then solve the pump system with dissipation
\begin{align}
\partial_t X_{n_0+k} &= -\lambda^{2(n_0+k)\alpha} X_{n_0+k} - \lambda^{n_0+k} X_{n_0+k} X_{n_0+k+1} \label{eot}\\
\partial_t X_{n_0+k+1} &= -\lambda^{2(n_0+k+1)\alpha} X_{n_0+k+1} + \lambda^{n_0+k} X_{n_0+k}^2\label{eot-2}
\end{align}
within the time interval $t \in [t_{n_0+k}, t_{n_0+k+1}]$, where $t_{n_0+k+1}$ is the first time for which $X_{n_0+k+1}(t_{n_0+k+1}) = \lambda^{-\delta(k+1)}$; we justify the existence of such a time below.
\item[Step 3.]  For each $n \neq n_0+k,n_0+k+1$, we evolve $X_n$ on $[t_{n_0+k}, t_{n_0+k+1}]$ by the linear ODE
$$ \partial_t X_n = -\lambda^{2n\alpha} X_n.$$
\item[Step 4.]  Increment $k$ to $k+1$ and return to Step 2.
\end{itemize}

Let us now establish that the time $t_{n_0+k+1}$ introduced in Step 2 is well defined for any given $k \geq 0$.  If we make the change of variables
$$ x(t) := \lambda^{\delta k} X_{n_0+k}( t_{n_0+k} + \lambda^{-n_0-k + \delta k} t )$$
$$ y(t) := \lambda^{\delta k} X_{n_0+k+1}( t_{n_0+k} + \lambda^{-n_0-k + \delta k} t )$$
then we see from construction that we have the initial conditions
$$ x(0)=1, y(0)=0$$
and the evolution equations
\begin{align}
 \partial_t x = - \eps x - xy \label{able}\\
 \partial_t y = - \lambda^{-2\alpha} \eps y + x^2 \label{bble}
\end{align}
where
$$ \eps := \lambda^{-(1-2\alpha)n_0 - (1-2\alpha-\delta)k},$$
and our task is to show that $y(t) = \lambda^{-\delta}$ for some finite $t>0$.  However, from the explicit solution \eqref{pomp} to the pump gate \eqref{pump}, we see that in the case $\eps=0$, this occurs at time $t = \operatorname{tanh}^{-1}( \lambda^{-\delta} )$; standard perturbation arguments then show that if $n_0$ is sufficiently large (which forces $\eps$ to be sufficiently small), the claim occurs at some time $t \leq 2 \operatorname{tanh}^{-1}( \lambda^{-\delta} )$ (say).  Undoing the scaling, we see that
$$ t_{n_0+k+1} - t_{n_0+k} \leq 2 \operatorname{tanh}^{-1}( \lambda^{-\delta} ) \lambda^{-n_0-k+\delta k}$$
so $t_n$ converges to a finite limit $T_*$ as $n \to \infty$, and the claim follows.
\end{proof}

As mentioned in the introduction, one can use (a slight modification of) this proposition to obtain a weaker ``exogenous'' version of Theorem \ref{main} in which the averaged operator $\tilde B = \tilde B(t)$ is now allowed to depend on the time coordinate $t$ (in a piecewise constant fashion, with an unbounded number of discontinuities as $t$ approaches the blowup time).  We leave the details (which are an adaptation of those in Section \ref{euler-avg}) to the interested reader.

\begin{remark} One cannot take $\delta=0$ in the above argument, because the pump gate never quite transfers all of its energy from the $x$ mode to the $y$ mode.  If however we worked with the modified equation
$$
\partial_t X_n = - \frac{\lambda^n}{g(\lambda^n)^2} X_n + 1_{(t_{n-1},t_n)}(t) \lambda^{n-1} (X_{n-1}^2 + X_{n-1} X_n) - 1_{(t_n,t_{n+1})}(t) \lambda^n (X_n X_{n+1} + X_{n+1}^2)$$
for some function $g: [0,+\infty) \to [0,+\infty)$ increasing to infinity, and defines $t_{n_0+k+1}$ to be the first time for which $X_{n_0+k}(t_{n_0+k+1})=0$ (so that $X_{n_0+k+1}$ is the only non-zero mode at this time), then a modification of the above argument establishes finite time blowup whenever $n_0$ is sufficiently large and
$$ \int_1^\infty \frac{ds}{sg(s)^2} < \infty,$$
basically because one can show inductively that $X_{n_0+k}(t_{n_0+k})$ is comparable to $1$, $t_{n_0+k+1}-t_{n_0+k}$ is comparable to $\lambda^{-n}$, and the energy dissipation on each time interval $[t_{n_0+k},t_{n_0+k+1}]$ is comparable to $\frac{1}{g(\lambda^{n_0+k})^2}$; we omit the details.  This is compatible with the heuristic calculation in \cite[Remark 1.2]{tao-hyper}.  In the converse direction, the arguments in \cite{tapay} or \cite{wu} should ensure global regularity for the above equation (or for the analogous hyperdissipative version of \eqref{xn}) under the condition
$$ \int_1^\infty \frac{ds}{sg(s)^4} = +\infty.$$
This leaves an intermediate regime (e.g. $g(s) = \log(1+s)^\beta$ for $1/4 < \beta \leq 1/2$) in which it is unclear whether one can force blowup\footnote{Since the initial release of this manuscript, it has been shown in \cite{bmr-conj} (see also \cite{bmr-dyadic}) that blowup in fact does not occur in this intermediate regime.  Roughly speaking, the basic point is that as the energy moves from low frequency modes to high frequency modes, it must transition through all intermediate frequency scales, and the cumulative energy dissipation from such transitions is enough to prevent the solution from escaping to frequency infinity in this intermediate regime.} with any of these ODE models.  The analysis in \cite{tapay} or \cite{wu} suggests that this may be possible, but one would have to work with models in which many different modes are activated at once (in contrast to the situation in Proposition \ref{bkp}, in which only two modes have interesting dynamics at any given time).
\end{remark}

\subsection{The amplifier gate}

The \emph{amplifier gate} is a reversed version of the pump gate:
\begin{equation}\label{amp}
\begin{split}
\partial_t x &= - \alpha y^2 \\
\partial_t y &= \alpha xy.
\end{split}
\end{equation}
Here again $\alpha>0$ is a coupling constant, indicating the strength of the amplifier.  We will use this gate in the regime in which $x$ is positive and large, and $y$ is positive but small.  In this case, we can explicitly solve the second equation to obtain
\begin{equation}\label{y-gronwall}
 y(t) = \exp\left( \alpha \int_{t_0}^t x(t')\ dt' \right) y(t_0)
\end{equation}
for any $t \geq t_0$, which suggests that $y$ grows exponentially at rate comparable to $\alpha x(0)$, until such time that the $y$ mode begins to drain a significant fraction of energy from the $x$ mode.  Thus, the $x$ mode can be viewed as causing exponential amplification in the $y$ mode.  Of course, in the presence of forcing terms, we no longer have the exact formula \eqref{y-gronwall}, but we may take advantage of Gronwall's inequality to obtain analogous control on $y$.

As with the pump gate, the amplifier gate preserves the total energy $x^2+y^2$.  An explicit solution to \eqref{amp} is given by
$$ x(t) = A \operatorname{tanh}(\alpha A (T-t)); \quad y(t) = A \operatorname{sech}( \alpha A (T-t) )$$
for any $A>0$ and $T>0$.  For $0 < t < T$, the quantity $y$ increases exponentially at rate about $\alpha A$, while $x$ stays roughly steady at $A$.

By using the amplifier with a large coupling constant $\alpha$, $x$ large and positive, and $y$ small and positive, we can cause $y$ to grow at a rapid exponential rate, and in particular to transition abruptly from being small (e.g. $y \leq \eps$ for some threshold $\eps$) to being large (e.g. $y>2\eps$), if the threshold $\eps$ is set low enough that $y$ does not yet begin to drain significant amounts of energy from $x$.  This ability to generate abrupt transitions is of course needed in our quest to engineer an abrupt delayed transition of energy from one mode to another.  This behaviour can be disrupted if $x$ becomes negative at some point, but we will avoid this in practice by making $x$ the output of a pump (which, as discussed previously, can serve to ``rectify'' an alternating input into a steadily increasing output).  We will represent the amplifier schematically by a triangle-headed arrow (Figure \ref{fig:amp}).

\begin{figure} [t]
\centering
\includegraphics{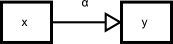}
\caption[Amplifier gate]{The amplifier gate from the $x$ mode to the $y$ mode with coupling constant $\alpha$.}
\label{fig:amp}
\end{figure}

\subsection{The rotor gate}

The \emph{rotor gate} is a ternary gate
\begin{equation}\label{rotor}
\begin{split}
\partial_t x &= -\alpha yz \\
\partial_t y &= \alpha xz \\
\partial_t z &= 0
\end{split}
\end{equation}
where again $\alpha >0$ is a parameter.
This of course preserves the total energy $x^2+y^2+z^2$ and has the explicit solution
\begin{align*}
x(t) &= x(t_0) \cos\left( \alpha z(t_0) (t-t_0) \right) - y(t_0) \sin\left( \alpha z(t_0) (t-t_0) \right) \\
y(t) &= y(t_0) \cos\left( \alpha z(t_0) (t-t_0) \right) + x(t_0) \sin\left( \alpha z(t_0) (t-t_0) \right) \\
z(t) &= z(t_0) 
\end{align*}
in which $(x(t),y(t))$ rotates around the origin at a contant angular rate $\alpha z(t_0)$, while $z$ remains fixed. Thus the $z$ mode can be viewed as driving the oscillating interchange of energy between the $x$ and $y$ modes.

Because we will be coupling the rotor to various forcing terms in $x$, $y$, and $z$, we cannot rely directly on the above explicit solution, although this solution is of course very useful for supplying intuition as to how the rotor behaves.  Instead, we will use energy-based analyses of the rotor, which are much more robust with respect to forcing terms.  Firstly we observe that for the rotor with no forcing, the combined energy of the $x$ and $y$ modes is conserved:
$$ \partial_t (x^2+y^2) = 0.$$
In a related spirit, we have the \emph{equipartition of energy identity}
$$ \alpha z (x^2-y^2) = \partial_t (xy)$$
or in integral form
$$ \alpha \int_{t_0}^T z(t) (x^2(t) - y^2(t))\ dt = x(T)y(T) - x(t_0)y(t_0).$$
Using the conserved energy $x^2+y^2=E$ and the constant nature of $z$, this becomes
$$ \frac{1}{T-t_0} \int_{t_0}^T x^2(t)\ dt = \frac{1}{2} E + O\left( \frac{E}{\alpha |z(t_0)| (T-t_0)}\right)$$
and similarly with $x(t)$ replaced by $y(t)$.  Thus we see that over any time interval significantly longer than the period $\frac{2\pi}{\alpha |z(t_0)|}$, the $x$ mode absorbs about half the energy $E$ of the combined pair $x,y$, and similarly for $y$.

In our application, we will use the rotor with the driving mode $z$ being the output of an amplifier.  As noted previously, amplifier outputs can transition rapidly from being small to being large, so the pair $(x,y)$ will initially be almost stationary, and then suddenly transition to a highly oscillatory state.  This creates a ``jolt'' of ``alternating current'', which we will then quickly transform to ``direct current'' via a pump gate.

We describe the rotor gate schematically by a loop connecting the $x$ and $y$ modes that is driven by the $z$ mode: see Figure \ref{fig:rotor}.

\begin{figure} [t]
\centering
\includegraphics{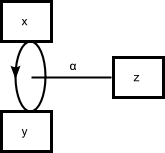}
\caption[Rotor gate]{The rotor gate that uses the $z$ mode to exchange energy between the $x$ and $y$ modes using the coupling constant $\alpha$.}
\label{fig:rotor}
\end{figure}

\subsection{A delayed and abrupt energy transition}

We can now build a ``quadratic circuit'' that achieves the goal of abruptly transitioning almost all of its energy from one mode to another, after a certain delay (and thus exhibiting ``digital'' transition behaviour rather than ``analog'' transition behaviour).  To describe this circuit, we first need a large parameter $K \gtrsim 1$, together with a very small parameter $0 < \eps \lesssim 1$ which will be sufficiently small depending on $K$ (e.g. one could choose $\eps := 1/\exp(\exp(CK))$ for some large absolute constant $C$).  We will then consider a five-mode circuit $X = (a,b,c,d,\tilde a)$ obeying the equations
\begin{align}
\partial_t a &= - \eps^{-2} c d - \eps ab - \eps^2 \exp(-K^{10}) ac \label{a-eq} \\
\partial_t b &= \eps a^2 - \eps^{-1} K^{10} c^2 \label{b-eq} \\
\partial_t c &= \eps^2 \exp(-K^{10}) a^2 + \eps^{-1} K^{10} bc \label{c-eq} \\
\partial_t d &= \eps^{-2} ca - K d\tilde a \label{d-eq}\\
\partial_t \tilde a &= K d^2\label{ta-eq}
\end{align}
and with initial data
\begin{equation}\label{a-init}
 a(0) = 1; \quad b(0)=c(0)=d(0)=\tilde a(0)=0.
 \end{equation}
This system looks complicated and artificial, with a rather arbitrary looking set of coupling constants of wildly differing magnitudes, but it should be viewed as a superposition of five quadratic gates:
\begin{itemize}
\item A pump of coupling constant $\eps$ that transfers a small amount of energy from $a$ to $b$;
\item A pump of coupling constant $\eps^2 \exp(-K^{10})$ that transfers a minute amount of energy from $a$ to $c$;
\item An amplifier of coupling constant $\eps^{-1} K^{10}$ that uses $b$ to rapidly amplify $c$;
\item A rotor of coupling constant $\eps^{-2}$ that uses $c$ to (eventually) rotate energy very rapidly between $a$ and $d$; and
\item A pump of coupling constant $K$ that drains energy from $d$ to $\tilde a$ at a moderately fast pace.
\end{itemize}
See Figure \ref{fig:delay}.

\begin{figure} [t]
\centering
\includegraphics{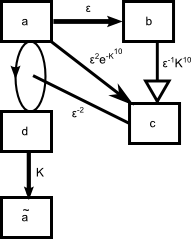}
\caption[Delay circuit]{A circuit that creates a delayed, but abrupt, transition of energy from $a$ to $\tilde a$.}
\label{fig:delay}
\end{figure}

One should view $a$ as the input mode for this circuit, and $\tilde a$ as the output mode; in later sections we will chain an infinite sequence of these circuits (rescaled by an exponentially growing parameter) together by identifying the output mode for each circuit with the input mode for the next.  One can check by hand that this system is of the form \eqref{ode} with bilinear form $G$ obeying the cancellation condition \eqref{g-cancel} (basically because the system is composed of gates, each of which individually satisfy this condition).

As a caricature, the evolution of this system can be described as follows, involving a critical time $t_c \approx \sqrt{2}$:
\begin{enumerate}
\item[(i)] At early times $0 \leq t \leq t_c - 1/\sqrt{K}$ (say), nothing much appears to happen: $a$ remains very close to $1$, $b$ grows linearly like $\eps t$, $c$ grows exponentially like $\eps^2 \exp( (\frac{1}{2}t^2-1) K^{10} )$, and $d$ and $\tilde a$ are close to $0$.
\item[(ii)] At a critical time $t_c \approx \sqrt{2}$, there is an abrupt transition when the exponentially growing $c$ suddenly (within a time of $O(K^{-10})$ or so) transitions from being much smaller than $\eps^2$ to being much larger than $\eps^2$.  This ignites the rotor gate, which then begins to rapidly transfer energy between $a$ and $d$.  By equipartition of energy, $d^2$ will approximately be equal to $1/2$ on the average.
\item[(iii)] After time $t_c+1/K$ or so, the pump between $d$ and $\tilde a$ begins to activate, and steadily drains the energy from $d$ (which, as mentioned before, contains about half the energy of the system) to $\tilde a$.  Meanwhile, $b$ and $c$ remain very small (because of the $\eps$ factors), but with $c$ large enough to continue the rapid mixing of energy between $a$ and $d$ throughout this process.  
\item[(iv)] By time $t_c + 1/\sqrt{K}$ (say), all but an exponentially small remnant of energy has been drained into $\tilde a$.
\end{enumerate}

We depict these dynamics schematically in Figure \ref{fig:transfer}.

\begin{figure} [t]
\centering
\includegraphics{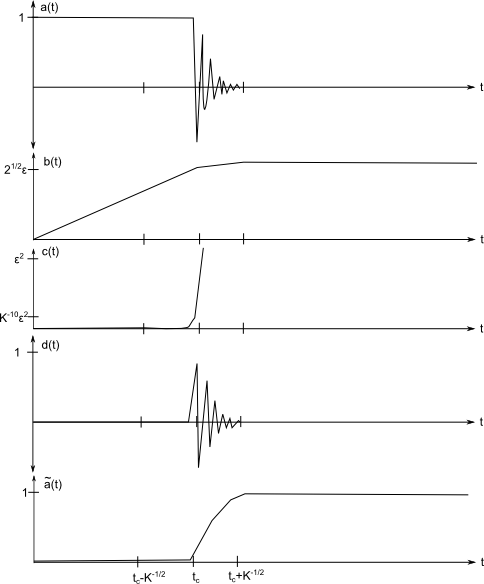}
\caption{A schematic description (not entirely drawn to scale) of the dynamics of $a,b,c,d,\tilde a$ before, near, and after the critical time $t_c$.  Note that $b$ and $c$ are of significantly smaller magnitude (by powers of $\eps$) than the other modes $a,d,\tilde a$, even after taking into account the exponential growth of $c$.  Note also how the pump gates convert alternating inputs to monotone outputs, and how the amplifier gate converts a slowly growing input to an exponentially growing output.  Finally, observe the transient nature of $d$, which only plays a role near the critical time $t_c$; the bulk of the energy is concentrated at the $a$ mode before time $t_c$ and at the $\tilde a$ mode after time $t_c$.}
\label{fig:transfer}
\end{figure}

To summarise the above description of the dynamics, this circuit has achieved the stated goal of creating an abrupt transition (of duration $O(1/\sqrt{K})$) of energy from one mode $a$ to another $\tilde a$, after a long delay (of duration $t_c \approx \sqrt{2}$).  (It is by no means the only circuit that can accomplish this task, but the author was not able to locate a circuit of lower complexity that did so.)

More formally, we claim

\begin{theorem}[Delayed abrupt energy transition]\label{daet} If $K$ is sufficiently large, and $\eps$ sufficiently small depending on $K$, then there exists a time 
\begin{equation}\label{tcable}
t_c = \sqrt{2} + O(1/\sqrt{K})
\end{equation}
 such that
\begin{equation}\label{able2}
 a(t) = 1 + O(K^{-10}); \quad b(t), c(t), d(t), \tilde a(t) = O(K^{-10})
\end{equation}
for $0 \leq t \leq t_c - 1/\sqrt{K}$, and
\begin{equation}\label{beable}
\tilde a(t) = 1 + O(K^{-10}); \quad a(t), b(t), c(t), d(t) = O(K^{-10})
\end{equation}
for $t \geq t_c + 1/\sqrt{K}$.
\end{theorem}

\begin{proof}  We shall use the usual bootstrap procedure of starting with crude estimates and steadily refining them to stronger estimates on this interval, using continuity arguments if necessary in case the crude estimates are initially only available at $t=0$ rather than for all $t \in [0,2]$.

From conservation of energy and \eqref{a-init} we have
\begin{equation}\label{energy-con}
a(t)^2 + b(t)^2 + c(t)^2 + d(t)^2 + \tilde a(t)^2 = 1
\end{equation}
throughout this interval.  In particular, we have
\begin{equation}\label{ob}
a(t),b(t),c(t),d(t),\tilde a(t) = O(1).
\end{equation}
We can improve this bound on $b$ and $c$ as follows.  From \eqref{b-eq}, \eqref{c-eq} we have the local energy identity
$$ \partial_t (b^2+c^2) = 2 \eps a^2 b + 2\eps^2 \exp(-K^{10}) a^2 c$$
and thus by \eqref{ob}
$$ \partial_t (b^2+c^2) = O\left( \eps (b^2+c^2)^{1/2} \right)$$
or equivalently
$$ \partial_t \left((b^2+c^2)^{1/2}\right) = O( \eps )$$
(in a weak derivative\footnote{To justify this step (a very simple example of the \emph{diamagnetic inequality} (see e.g. \cite[\S 7.19-7.22]{ll}) in action), one can first work instead with $(b^2+c^2+\delta)^{1/2}$ for some small $\delta>0$, in order to avoid any singularity, and then take distributional limits as $\delta \to 0$.} sense), and so from \eqref{a-init} and the fundamental theorem of calculus we see in particular that
\begin{equation}\label{ob-2}
b(t), c(t) = O(\eps)
\end{equation}
for $t \in [0,2]$.  Inserting this into \eqref{c-eq}, we see that
\begin{equation}\label{cgrow}
\partial_t c = O( \eps^2 \exp(-K^{10}) ) + O( K^{10} c )
\end{equation}
for $t \in [0,2]$ (note from the initial condition $c(0)=0$ and a comparison argument that $c(t) \geq 0$ for all $t \geq 0$).  By Gronwall's inequality we thus have
\begin{equation}\label{code}
0 \leq c(t) \lesssim \eps^2 \exp\left((Ct-1)K^{10} \right)
\end{equation}
for all $t \in [0,2]$ and some absolute constant $C>0$.  Finally, from \eqref{d-eq}, \eqref{ta-eq} we have the local energy identity
$$
\partial_t (d^2+\tilde a^2) = 2 \eps^{-2} c ad
$$
and hence by \eqref{ob}
$$
\partial_t \left((d^2+\tilde a^2)^{1/2}\right) = O( \eps^{-2} c ) $$
and thus by \eqref{a-init} and the fundamental theorem of calculus
\begin{equation}\label{dora}
d(t), \tilde a(t) = O\left( \eps^{-2} \int_0^t c(t')\ dt' \right).
\end{equation}
In particular, from \eqref{code} we have
\begin{equation}\label{explorer-2}
|d(t)|, |\tilde a(t)| \lesssim K^{-10} \exp\left( (Ct-1) K^{10} \right),
\end{equation}
which is a good bound for short times $t \leq 1/C$.

Having obtained crude bounds on all five quantities $a(t), b(t), c(t), d(t),\tilde a(t)$, we now return to obtain sharper bounds on these quantities.  Let $t_c$ be the supremum of all the times $t \in [0,2]$ for which $c(t') \leq K^{-10} \eps^2$ for all $0 \leq t' \leq t$, thus $0 < t_c \leq 2$ and
\begin{equation}\label{boots}
c(t) \leq K^{-10} \eps^2
\end{equation}
for $t \in [0,t_c]$.  Comparing this with \eqref{code} we conclude that $t_c \gtrsim 1$.  From \eqref{dora} we have
$$ |d(t)|, |\tilde a(t)| \lesssim K^{-10}$$
for $t \in [0,t_c]$.  Inserting these bounds and \eqref{ob-2} back into \eqref{a-eq}, we have
$$ 
\partial_t a = O( K^{-20} ) + O( \eps^2 )
$$
on $[0,t_c]$, so from \eqref{a-init} (and assuming $\eps$ sufficiently small depending on $K$) we have
$$ 
a(t) = 1 + O( K^{-20} )$$
for $t \in [0,t_c]$.  This already gives all the bounds \eqref{able2}.  Inserting the $a$ bound into \eqref{b-eq} and using \eqref{boots}, we have
$$ \partial_t b = \eps + O( K^{-20} \eps ) + O( K^{-10} \eps^3 ) $$
and so from \eqref{a-init} (again assuming $\eps$ sufficiently small depending on $K$) we have
\begin{equation}\label{bogo-2}
 b(t) = \eps t + O( K^{-20} \eps )
\end{equation}
for $t \in [0,t_c]$.  Inserting these bounds into \eqref{c-eq}, we have
$$ \partial_t c = (1 + O(K^{-20})) \eps^2 \exp(-K^{10}) + (K^{10} t + O(K^{-10})) c$$
and hence by \eqref{a-init} and Gronwall's inequality
\begin{align*}
c(t) &= \int_0^t \exp\left(\int_{t'}^t \left(K^{10} t'' + O(K^{-10})\right)\ dt''\right) \left(1 + O(K^{-20})\right) \eps^2 \exp(-K^{10})\ dt' \\
&= \left(1 + O(K^{-10})\right) \eps^2 \exp\left(\left(\frac{1}{2} t^2- 1\right)K^{10}\right) \int_0^t \exp\left( - \frac{1}{2} (t')^2 K^{10}\right)\ dt'
\end{align*}
for $t \in [0,t_c]$.  In particular (since $t_c \gtrsim 1$), standard asymptotics on the error function give
$$ c(t_c) = \left(1 + O(K^{-10})\right) \eps^2 \exp\left(\left(\frac{1}{2} t_c^2-1\right)K^{10}\right) \sqrt{\frac{\pi}{2 K^{10}}}
$$
which, when compared against the definition of $t_c$, shows \eqref{tcable}.  In particular, $t_c < 2$ (for $K$ large enough), and so 
\begin{equation}\label{c-bound}
c(t_c) = K^{-10} \eps^2.
\end{equation}

Having described the evolution up to time $t_c$, we now move to the future of $t_c$.  From \eqref{bogo-2} we have
$$ b(t_c) \gtrsim \eps.$$
Meanwhile, from \eqref{code}, \eqref{b-eq} (discarding the non-negative $\eps a^2$ term) we have
$$ \partial_t b \geq - \eps^3 \exp(O( K^{10} ) )$$
for $t \in [t_c,2]$, so (for $\eps$ small enough) we also have
$$ b(t) \gtrsim \eps$$
for $t \in [t_c,2]$.  Inserting this bound into \eqref{c-eq}, and discarding the non-negative $\eps^2 \exp(-K^{10}) a^2$ term, we arrive at the exponential growth 
$$ \partial_t c(t) \gtrsim K^{10} c(t)$$
for $t \in [t_c,2]$.  From this, \eqref{c-bound}, and Gronwall's inequality, we see in particular that
\begin{equation}\label{c-large}
c(t) \geq K^{100} \eps^2
\end{equation}
for $t$ in the time interval $I := [t_c + K^{-9}, 2]$.  In other words, the rotor gate will be continuously and strongly activated from time $t_c+K^{-9}$ onwards.  
On the other hand, from \eqref{cgrow}, \eqref{c-large} we also have
\begin{equation}\label{solace}
\partial_t c(t) = O( K^{10} c(t) )
\end{equation}
for $t \in I$, so the exponential growth rate of $c$ remains under control in this region.

Now we use equipartition of energy to establish some reasonably rapid energy drain from $a,b,c,d$ to $\tilde a$.  
From \eqref{a-eq}-\eqref{d-eq} one has
$$ \partial_t (a^2+b^2+c^2+d^2) = - 2Kd^2 \tilde a.$$
Similarly, from \eqref{a-eq}, \eqref{d-eq}, and \eqref{ob} one has
$$ \partial_t(ad) = \eps^{-2} c (a^2-d^2) - O(K).$$
Finally, from \eqref{ta-eq}, \eqref{ob} one has
$$ \partial_t \tilde a  = O(K).$$
We conclude using \eqref{c-large}, \eqref{solace}, and the product rule that
\begin{equation}\label{douse}
\begin{split}
\partial_t( ad \frac{\eps^2}{c} \tilde a) &= - (a^2-d^2) \tilde a - ad \frac{\eps^2}{c} \frac{\partial_t c}{c} \tilde a + ad \frac{\eps^2}{c} \partial_t \tilde a + O(K^{-99}) \\
&= - (a^2-d^2) \tilde a + O(K^{-90})
\end{split}
\end{equation}
for $t \in I$,
so if we define the modified energy
$$ E_* := \frac{1}{2}(a^2+b^2+c^2+d^2) + \frac{1}{2} K ad \frac{\eps^2}{c} \tilde a$$
then
$$ \partial_t E_* = - \frac{1}{2} K (a^2+d^2) \tilde a + O( K^{-80} )$$
for $t \in I$.
From \eqref{ob-2} we have
\begin{equation}\label{est}
E_* = \frac{1}{2}(a^2+b^2+c^2+d^2) + O(K^{-99}) = \frac{1}{2} (a^2+d^2) + O(K^{-99}) 
\end{equation}
and thus
$$ \partial_t E_* = - K \tilde a E_* + O( K^{-80} ).$$
Starting with the crude bound $E_*(t_c) = O(1)$ from \eqref{est}, we thus see from Gronwall's inequality that
$$ E_*(t) \lesssim \exp\left( - K \int_{t_c}^t \tilde a(t')\ dt'\right) + O(K^{-80})$$
for all $t \in I$.

Observe from \eqref{ta-eq} that $\tilde a$ is non-decreasing, and thus
\begin{equation}\label{toke}
 E_*(t) \lesssim \exp( - K (t-t')  \tilde a(t')) + O(K^{-80})
\end{equation}
whenever $t_c \leq t' \leq t \leq 2$.  We will use this bound with $t' := t_c + 1/K$.  We claim that
\begin{equation}\label{atc}
\tilde a(t_c+1/K) \geq 0.1.
\end{equation}
Suppose this is not the case; then by \eqref{ta-eq} we have
$$ \int_{t_c}^{t_c+1/K} d(t'')^2\ dt'' \leq \frac{1}{10K}.$$
However, by repeating the derivation of \eqref{douse} we have
$$ \partial_t \left(ad \frac{\eps^2}{c}\right) = - (a^2-d^2) + O(K^{-90})$$
and hence by the fundamental theorem of calculus and \eqref{c-large} we have
$$ \int_{t_c}^{t_c+1/K} (a(t'')^2 - d(t'')^2)\ dt'' = O( K^{-90} );$$
combining this with the previous estimate, we conclude that
$$ \int_{t_c}^{t_c+1/K} \frac{1}{2} (a^2+d^2)(t'')\ dt'' \leq \frac{1}{10K} + O(K^{-90}).$$
On the other hand, for $t'' \in [t_c,t_c+1/K]$ one has $\tilde a(t) \leq 0.1$ by monotonicity of $\tilde a$, and hence by \eqref{est}, \eqref{energy-con} we have
$$ a^2+d^2 \geq 0.99 + O(K^{-99})$$
in this interval, giving the required contradiction.  

Inserting the bound \eqref{atc} into \eqref{toke}, we conclude in particular that
$$ E_*(t) \lesssim K^{-80}$$
for $t_c+\frac{1}{\sqrt{K}} \leq t \leq 2$, and \eqref{beable} follows from \eqref{est} and \eqref{energy-con}.
\end{proof}

\section{Blowup for the cascade ODE}\label{blowup-sec}

We can now prove Theorem \ref{ood} (and hence Theorem \ref{blowup} and Theorem \ref{main}).  The idea is to chain together an infinite sequence of circuits of the form \eqref{a-eq}-\eqref{ta-eq}, so that (a more complicated version of) the analysis from Theorem \ref{daet} may be applied.

\subsection{First step: constructing the ODE}

Let $\epsilon_0 > 0$ be fixed; we allow all implied constants in the $O()$ notation to depend on $\epsilon_0$. As in the previous section, we need a large constant $K \geq 1$, which we assume to be sufficiently large depending on $\epsilon_0$, and then a small constant $0 < \eps < 1$, which we assume to be sufficiently small depending on both $K$ and $\epsilon_0$.  Finally, we take $n_0$ sufficiently large depending on $\epsilon_0, K, \eps$.

The reader may wish to keep in mind the hierarchy of parameters
$$ 1 \ll \frac{1}{\epsilon_0} \ll K \ll \frac{1}{\eps} \ll n_0$$
as a heuristic for comparing the magnitude of various quantities appearing in the sequel.  Thus, for instance, a quantity of the form $O\left( \exp\left( O\left(K^{10}\right) \right) \left(1+\epsilon_0\right)^{-n_0/2} \right)$ will be smaller than $\exp\left(-K^{10}\right) \eps^2$; a quantity of the form $O\left( \exp\left( O\left( K^{10}\right)\right) \eps \right)$ will be smaller than $K^{-100}$; and so forth.

\begin{table}
\centering
 \caption{The non-zero values of $\alpha_{i_1,i_2,i_3,\mu_1,\mu_2,\mu_3}$.}
 \label{alpha-table}
  \begin{tabular}{lllllll}
   \toprule
     $i_1$ & $i_2$ & $i_3$ & $\mu_1$ & $\mu_2$ & $\mu_3$ & $\alpha_{i_1,i_2,i_3,\mu_1,\mu_2,\mu_3}$ \\
   \midrule
	   $3$  & $4$ & $1$ & $0$ & $0$ & $0$ & $-\eps^{-2} / 2$ \\
	   $4$  & $3$ & $1$ & $0$ & $0$ & $0$ & $-\eps^{-2} / 2$ \\
	   $1$  & $3$ & $4$ & $0$ & $0$ & $0$ & $\eps^{-2} / 2$ \\
	   $3$  & $1$ & $4$ & $0$ & $0$ & $0$ & $\eps^{-2} / 2$ \\
	   $1$  & $2$ & $1$ & $0$ & $0$ & $0$ & $-\eps / 2$ \\
	   $2$  & $1$ & $1$ & $0$ & $0$ & $0$ & $-\eps / 2$ \\
	   $1$  & $1$ & $2$ & $0$ & $0$ & $0$ & $\eps$ \\
	   $1$  & $3$ & $1$ & $0$ & $0$ & $0$ & $-\eps^{2} \exp(-K^{10})/ 2$ \\
	   $3$  & $1$ & $1$ & $0$ & $0$ & $0$ & $-\eps^{2} \exp(-K^{-10}) / 2$ \\
	   $1$  & $1$ & $3$ & $0$ & $0$ & $0$ & $\eps^{2} \exp(K^{-10})$ \\
		 $3$  & $3$ & $2$ & $0$ & $0$ & $0$ & $-\eps^{-1} K^{10}$ \\
		 $2$  & $3$ & $3$ & $0$ & $0$ & $0$ & $\eps^{-1} K^{10}/2$ \\
		 $3$  & $2$ & $3$ & $0$ & $0$ & $0$ & $\eps^{-1} K^{10}/2$ \\		
	   $4$  & $4$ & $1$ & $0$ & $0$ & $1$ & $(1+\epsilon_0)^{5/2} K$ \\
	   $1$  & $4$ & $4$ & $1$ & $0$ & $0$ & $-(1+\epsilon_0)^{5/2} K/2$ \\
	   $4$  & $1$ & $4$ & $0$ & $1$ & $0$ & $-(1+\epsilon_0)^{5/2} K/2$\\
	   \bottomrule
  \end{tabular}
\end{table}

The dimension parameter $m$ for the system we will use to prove Theorem \ref{ood} will be taken to be $m=4$.  
We set the coefficients $\alpha_{i_1,i_2,i_3,\mu_1,\mu_2,\mu_3}$ by using Table \ref{alpha-table}, with $\alpha_{i_1,i_2,i_3,\mu_1,\mu_2,\mu_3}$ set equal to zero if it does not appear in the above table.  It is clear that the required symmetry property \eqref{symmetry} and the cancellation property \eqref{cyclic} hold.    Also, the hypotheses of Lemma \ref{eqmot}(v) are satisfied.

Suppose for contradiction that Theorem \ref{ood} fails for this choice of parameters, so that we have global continuously differentiable functions $X_{i,n}: [0,+\infty) \to \R$ and $E_{i,n}: [0,+\infty) \to [0,+\infty)$ obeying the conclusions of Lemma \ref{eqmot}.  More precisely, by Lemma \ref{eqmot}(i), we have the \emph{a priori} regularity
\begin{equation}\label{many}
\sup_{0 \leq t\leq T} \sup_{n \in\Z} \sup_{i=1,\dots,4} \left(1 + (1+\epsilon_0)^{10n}\right) |X_{i,n}(t)| < \infty
\end{equation}
and
$$
\sup_{0 \leq t\leq T} \sup_{n \in\Z} \sup_{i=1,\dots,4} \left(1 + (1+\epsilon_0)^{10n}\right) E_{i,n}(t)^{1/2} < \infty,
$$
for all $0 < T< \infty$.  

It will be convenient to work with the combined energy
$$E_n := E_{1,n} + E_{2,n} + E_{3,n} + E_{4,n},$$
so that
\begin{equation}\label{many-2}
\sup_{0 \leq t\leq T} \sup_{n \in\Z} \left(1 + (1+\epsilon_0)^{10n}\right) E_{n}(t)^{1/2} < \infty
\end{equation}
for all $0 < T < \infty$.

By Lemma \ref{eqmot}(iii), we have the equations of motion
\begin{align}
\partial_t X_{1,n} &= (1+\epsilon_0)^{5n/2} (- \eps^{-2} X_{3,n} X_{4,n} - \eps X_{1,n} X_{2,n} - \eps^2 \exp(-K^{10}) X_{1,n} X_{3,n} + K X_{4,n-1}^2)\nonumber\\
&\quad\quad + O\left( (1+\epsilon_0)^{2n} E_n^{1/2} \right) \label{axin-eq} \\
\partial_t X_{2,n} &= (1+\epsilon_0)^{5n/2} (\eps X_{1,n}^2 - \eps^{-1} K^{10} X_{3,n}^2) + O\left( (1+\epsilon_0)^{2n} E_n^{1/2} \right) \label{bxin-eq} \\
\partial_t X_{3,n} &= (1+\epsilon_0)^{5n/2} (\eps^2 \exp(-K^{10}) X_{1,n}^2 + \eps^{-1} K^{10} X_{2,n} X_{3,n} ) + O\left( (1+\epsilon_0)^{2n} E_n^{1/2} \right) \label{cxin-eq} \\
\partial_t X_{4,n} &= (1+\epsilon_0)^{5n/2} (\eps^{-2} X_{3,n} X_{1,n} - (1+\epsilon_0)^{5/2} K X_{4,n} X_{1,n+1}) + O\left( (1+\epsilon_0)^{2n} E_n^{1/2} \right) \label{dxin-eq}
\end{align}
(compare with \eqref{a-eq}-\eqref{ta-eq}) and the local energy inequality
\begin{equation}\label{lei-0}
 \partial_t E_n \leq (1+\epsilon_0)^{5n/2} K X_{4,n-1}^2 X_{1,n} - (1+\epsilon_0)^{5(n+1)/2} K X_{4,n}^2 X_{1,n+1} 
\end{equation}
for any $n \in \Z$ and $t \geq 0$.

\begin{remark} As mentioned in the previous section, if one ignores the dissipation terms, the system \eqref{axin-eq}-\eqref{dxin-eq} describes an infinite number of (rescaled) copies of the quadratic circuit analysed in Theorem \ref{daet}, with the output of each such circuit chained to the input of a slightly faster-running version of the same circuit; see Figure \ref{fig:circ1}.
\end{remark}

\begin{figure} [t]
\centering
\includegraphics{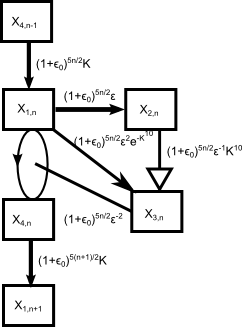}
\caption[Cascade of circuits]{A portion of the system \eqref{axin-eq}-\eqref{dxin-eq}, focusing on the modes at or near scale $n$, and ignoring the dissipation terms.  Compare with Figure \ref{fig:delay}.}
\label{fig:circ1}
\end{figure}

By Lemma \ref{eqmot}(iii), we have the initial conditions
\begin{equation}\label{initial-cond}
E_{n}(0) = \frac{1}{2} 1_{n=n_0}; \quad X_{i,n}(0) = 1_{(i,n)=(1,n_0)}
\end{equation}
for all $n \in \Z$ and $i=1,\dots,4$.

By Lemma \ref{eqmot}(iv), we have
\begin{equation}\label{ein-sum}
\frac{1}{2}  \sum_{i=1}^4 X_{i,n}^2(t)  \leq E_n(t) \leq \frac{1}{2}\sum_{i=1}^4  X_{i,n}^2(t) + O\left( (1+\epsilon_0)^{2n} \int_0^t E_n(t')\ dt' \right)
\end{equation}
for any $n \in \Z$ and $t \geq 0$.  

Finally, from Lemma \ref{eqmot}(v) we have
\begin{equation}\label{nomode}
E_n(t) = X_{1,n}(t) = X_{2,n}(t)=X_{3,n}(t) = X_{4,n}(t)=0
\end{equation}
for all $n < n_0$ and $t \geq 0$.

To prove Theorem \ref{ood}, it thus suffices to show

\begin{theorem}[No global solution for ODE system]\label{ood2}  Let $0 < \epsilon_0 < 1$, let $K>0$ be sufficiently large depending on $\epsilon_0$, let $\eps>0$ be sufficiently small depending on $\epsilon_0,K$, and let $n_0$ be sufficiently large depending on $\epsilon_0, K, \eps$, and the implied constants in \eqref{axin-eq}, \eqref{dxin-eq}, \eqref{ein-sum}.  Then there does not exist continuously differentiable functions $X_{i,n}: [0,+\infty) \to \R$ and $E_n: [0,+\infty) \to [0,+\infty)$ obeying the estimates and equations \eqref{many}-\eqref{nomode} for the indicated range of parameters.
\end{theorem}

It remains to prove Theorem \ref{ood2}.

\subsection{Second step: describing the blowup dynamics}

We will establish the following description of the dynamics of $X_{i,n}$ and $E_{i,n}$:

\begin{proposition}[Blowup dynamics]\label{blowdyn}  Let the hypotheses and notation be as in Theorem \ref{ood2}, and suppose for contradiction that we may find continuously differentiable functions $X_{i,n}:[0,+\infty) \to \R$ and $E_n: [0,+\infty) \to [0,+\infty)$ with the stated properties \eqref{many}-\eqref{nomode}.

Let $N \geq n_0$ be an integer.  Then there exist times
$$ 0 \leq t_{n_0} < t_{n_0+1} < \dots < t_N < \infty$$
and amplitudes
$$ e_{n_0}, \dots, e_N > 0$$
obeying the following properties:
\begin{itemize}
\item[(vi)] (Initialisation) We have
\begin{equation}\label{t-init}
t_{n_0} = 0
\end{equation}
and
\begin{equation}\label{e-init}
e_{n_0} = 1.
\end{equation}
\item[(vii)] (Scale evolution) For all $n_0 < n \leq N$, one has the amplitude stability
\begin{equation}\label{en-stable}
(1+\epsilon_0)^{-1/100} e_{n-1} \leq e_n \leq (1+\epsilon_0)^{1/100} e_{n-1}
\end{equation}
and the lifespan bound
\begin{equation}\label{lifespan}
\frac{1}{100} (1+\epsilon_0)^{-5(n-1)/2} e_{n-1}^{-1} \leq t_n - t_{n-1} \leq 100 (1+\epsilon_0)^{-5(n-1)/2} e_{n-1}^{-1}.
\end{equation}
\item[(viii)] (Transition state)  For all $n_0 \leq n \leq N$, we have the bounds
\begin{align}
X_{1,n}(t_n) &= e_n \label{an-init} \\
|X_{2,n}(t_n)| &\leq 10^{-5} \eps e_n \label{bn-init}\\
|X_{3,n}(t_n)| &\leq 10^{-5} \exp( -K^{10} ) \eps^2 e_n \label{cn-init}\\
X_{3,n}(t_n) &\geq - (1+\epsilon_0)^{-n_0/4} e_n\label{cn-init-2}\\
|X_{4,n}(t_n)| &\leq K^{-10} e_{n} \label{dn-en} \\
E_{n-1}(t_n) &\leq K^{-20} e_n^2. \label{en1-init}
\end{align}
If $n_0 < n \leq N$, we have the additional bounds
\begin{align}
X_{2,n-1}(t_n) &\geq 10^{-5} \eps e_n \label{spin-1}\\
X_{2,n-1}(t_n) &\leq 10^{5} \eps e_n \label{spin-1a}\\
X_{3,n-1}(t_n) &\geq \exp( K^{9} ) \eps^2 e_n.\label{spin-2}\\
X_{3,n-1}(t_n) &\leq \exp( K^{10} ) \eps^2 e_n.\label{spin-2a}
\end{align}
\item[(ix)] (Energy estimates) For all $n_0 < n \leq N$ and $t_{n-1} \leq t \leq t_n$, we have the bounds
\begin{align}
E_{n-m}(t) &\leq K^{-10} (1+\epsilon_0)^{m/10} e_{n-1}^2 \hbox{ for all } m \geq 2 \label{before-en} \\
E_{n-1}(t)+E_n(t) &\leq e_{n-1}^2 \label{during-en} \\
E_{n+m}(t) &\leq K^{-30} (1+\epsilon_0)^{-10m} e_{n-1}^2 \hbox{ for all } m \geq 1 \label{after-en}
\end{align}
\end{itemize}
\end{proposition}

These bounds may appear somewhat complicated, but roughly speaking they assert that at each time $t_n$, the solution concentrates an important part of its energy at scale $n$ (and significantly less energy at adjacent scales); see Table \ref{energy-table} and Figure \ref{energy-figure}.  The precise bounds here do have to be chosen carefully, because of a rather intricate induction argument in which the estimates for a given value of $N$ are used to prove the estimates for $N+1$.  For this reason, no use of the asymptotic notation $O()$ appears in the above proposition.  Of the four modes $X_{1,n}, X_{2,n}, X_{3,n}, X_{4,n}$, it is the first mode $X_{1,n}$ that carries most of the energy at the checkpoint time $t_n$; the secondary modes $X_{2,n}, X_{3,n}$ play an important role in driving the dynamics (and so many of the more technical bounds in (viii) are devoted to controlling these modes) but carry\footnote{As a crude first approximation (ignoring factors depending on $K$), one should think of $X_{2,n}$ as being about $\eps$ the size of $X_{1,n}$ or $X_{4,n}$, and $X_{3,n}$ being about $\eps^2$ the size of $X_{1,n}$ or $X_{4,n}$.} very little energy, while the $X_{4,n}$ mode is only used as a conduit to transfer energy from the $X_{1,n}$ mode to the $X_{1,n+1}$ mode.  The bounds \eqref{spin-1}-\eqref{spin-2a} are technical; they are needed to ensure that the rotor at scale $n-1$ is rotating so quickly that the modes at scale $n-1$ do not cause any ``constructive interference'' with the modes at scale $n$ at time $t_n$ (or at slightly later times).

\begin{table}
\centering
 \caption{Upper bounds for energies at scales close to $N$, at times close to $t_N$.  Note that at time $t_N$, the energy is locally concentrated at scale $N$, but transitions as $t_N \leq t \leq t_{N+1}$ to a state at time $t=t_{N+1}$ at which the energy is locally concentrated at scale $N+1$.}
 \label{energy-table}
  \begin{tabular}{lllll}
   \toprule
     Energy     & $t_{N-1} < t < t_N$                         & $t=t_N$                                   & $t_N < t < t_{N+1}$                     & $t=t_{N+1}$ \\
   \midrule  
   $E_{N-2}$    &  $K^{-10} (1+\epsilon_0)^{2/10} e_{N-1}^2$  & $K^{-10} (1+\epsilon_0)^{2/10} e_{N-1}^2$ & $K^{-10} (1+\epsilon_0)^{3/10} e_{N}^2$ & $K^{-10} (1+\epsilon_0)^{3/10} e_{N}^2$ \\ 
   $E_{N-1}$    &  $e_{N-1}^2$                                & $K^{-20} e_N$                             & $K^{-10} (1+\epsilon_0)^{2/10} e_{N}^2$ & $K^{-10} (1+\epsilon_0)^{2/10} e_{N}^2$ \\
   $E_N$        &  $e_{N-1}^2$                                & $(\frac{1}{2} + O(K^{-20}))e_N^2$         & $e_N^2$                                 & $K^{-20} e_{N+1}^2$ \\ 
   $E_{N+1}$    &  $K^{-30} (1+\epsilon_0)^{-10} e_{N-1}^2$   & $K^{-30} (1+\epsilon_0)^{-10} e_{N-1}^2$  & $e_N^2$                                 & $(\frac{1}{2} + O(K^{-20})) e_{N+1}^2$\\
   $E_{N+2}$    &  $K^{-30} (1+\epsilon_0)^{-20} e_{N-1}^2$   & $K^{-30} (1+\epsilon_0)^{-20} e_{N-1}^2$  & $K^{-30} (1+\epsilon_0)^{-10} e_{N}^2$  & $K^{-30} (1+\epsilon_0)^{-10} e_{N}^2$\\
   \bottomrule
  \end{tabular}
\end{table}

\begin{figure} [t]
\centering
\includegraphics{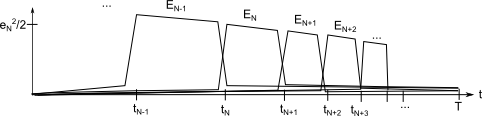}
\caption{A schematic description (not entirely drawn to scale) of the dynamics of the energies $E_n$ for $n$ close to $N$ and times close to $t_N$.  Observe that the energy is concentrated at a single scale $N$ for a lengthy time interval (most of $[t_N,t_{N+1}]$), but then abruptly transfers almost all of its energy (minus some losses due to dissipation and interaction with other scales) to the next finer scale $N+1$ by the time $t_{N+1}$.  The length of the time intervals $[t_N,t_{N+1}]$ decreases geometrically, leading to blowup at some finite time $T$.}
\label{energy-figure}
\end{figure}

Let us now see how the above proposition implies Theorem \ref{ood2} (and hence Theorem \ref{ood}, Theorem \ref{blowup} and Theorem \ref{main}).  Let $N \geq n_0$ be arbitrary.  From \eqref{e-init}, \eqref{en-stable}, \eqref{lifespan} we have
$$ t_n - t_{n-1} \lesssim (1+\epsilon_0)^{-(\frac{5}{2}-\frac{1}{100}) n} $$
and hence by \eqref{t-init} and summing the geometric series we have
$$ t_N \leq T$$
for some finite $T = T_{\epsilon_0}$ independent of $N$.  On the other hand, from \eqref{an-init}, \eqref{en-stable}, \eqref{e-init} we have
$$ |X_{1,N}(t_N)| \geq (1+\epsilon_0)^{-N/100}$$
and hence
$$ \sup_{0 \leq t\leq T} \sup_{n \in\Z} \sup_{i=1,\dots,4} (1 + (1+\epsilon_0)^{10n}) X_{i,n}(t) \geq (1+\epsilon_0)^{9N}$$
for any $N$.  Sending $N$ to infinity, we contradict \eqref{many}.

\subsection{Third step: setting up the induction}

It remains to prove Proposition \ref{blowdyn}.  We do so by an induction on $N$.  The base case $N=n_0$ is easy: one sets $t_{n_0}=0$ and $e_{n_0}=1$, and all the required claims are either vacuously true or follow immediately from the initial conditions \eqref{initial-cond}.  It remains to establish the inductive case of this proposition.  For the convenience of the reader, we state this inductive case as an explicit proposition.

\begin{proposition}[Blowup dynamics, inductive case]\label{blowdyn-induct}  Let the hypotheses and notation be as in Theorem \ref{ood2}, and suppose for contradiction that we may find continuously differentiable functions $X_{i,n}:[0,+\infty) \to \R$ and $E_n: [0,+\infty) \to [0,+\infty)$ with the stated properties \eqref{many}-\eqref{nomode}.

Assume that Proposition \ref{blowdyn} has already been established for some $N \geq n_0$, giving times
$$ 0 \leq t_{n_0} < t_{n_0+1} < \dots < t_N < \infty$$
and energies
$$ e_{n_0}, \dots, e_N > 0$$
with the properties \eqref{t-init}-\eqref{after-en} stated in that proposition. Then there exists a time
$$ t_N < t_{N+1} < \infty$$
and an amplitude $e_{N+1} > 0$ obeying the following properties:
\begin{itemize}
\item[(vii')] (Scale evolution) One has the amplitude stability
\begin{equation}\label{en-stable-induct}
(1+\epsilon_0)^{-1/100} e_N \leq e_{N+1} \leq (1+\epsilon_0)^{1/100} e_N
\end{equation}
and the lifespan bound
\begin{equation}\label{lifespan-induct}
\frac{1}{100} (1+\epsilon_0)^{-5N/2} e_N^{-1} \leq t_{N+1} - t_N \leq 100 (1+\epsilon_0)^{-5N/2} e_N^{-1}.
\end{equation}
\item[(viii')] (Transition state)  We have the bounds
\begin{align}
X_{1,N+1}(t_{N+1}) &= e_{N+1} \label{an-init-induct} \\
|X_{2,N+1}(t_{N+1})| &\leq 10^{-5} \eps e_{N+1} \label{bn-init-induct}\\
|X_{3,N+1}(t_{N+1})| &\leq 10^{-5} \exp( -K^{10} ) \eps^2 e_{N+1} \label{cn-init-induct}\\
X_{3,N+1}(t_{N+1}) &\geq - (1+\epsilon_0)^{-n_0/4} e_{N+1} \label{cn-init-induct-2}\\
|X_{4,N+1}(t_{N+1})| &\leq K^{-10} e_{N+1} \label{dn-en-induct} \\
E_N(t_{N+1}) &\leq K^{-20} e_{N+1}^2 \label{en1-init-induct}\\
X_{2,N}(t_{N+1}) &\geq 10^{-5} \eps e_{N+1}\label{spin-1-induct}\\
X_{2,N}(t_{N+1}) &\leq 10^{5} \eps e_{N+1}\label{spin-1a-induct}\\
X_{3,N}(t_{N+1}) &\geq \exp( K^{9} ) \eps^2 e_{N+1}^2.\label{spin-2-induct}\\
X_{3,N}(t_{N+1}) &\leq \exp( K^{10} ) \eps^2 e_{N+1}^2.\label{spin-2a-induct}
\end{align}
\item[(ix')] (Energy estimates) For all $t_N \leq t \leq t_{N+1}$, we have the bounds
\begin{align}
E_{N+1-m}(t) &\leq K^{-10} (1+\epsilon_0)^{m/10} e_N^2 \hbox{ for all } m \geq 2 \label{before-en-induct} \\
E_{N}(t)+E_{N+1}(t) &\leq e_{N}^2 \label{during-en-induct} \\
E_{N+1+m}(t) &\leq K^{-30} (1+\epsilon_0)^{-10m} e_N^2 \hbox{ for all } m \geq 1 \label{after-en-induct}
\end{align}
\end{itemize}
\end{proposition}

Clearly, Proposition \ref{blowdyn-induct} implies Proposition \ref{blowdyn} (and hence Theorems \ref{ood2}, \ref{ood}, \ref{blowup} and \ref{main}).

Roughly speaking, the situation is as follows.  At time $t_N$, the solution $(X_{i,n})_{i=1,\dots,4; n \in \Z}$ has a large amount of energy at a single mode $X_{1,N}$ at scale $N$ (thanks to \eqref{an-init}) and small amounts of energy at nearby modes and scales (thanks to \eqref{bn-init}, \eqref{cn-init}, \eqref{dn-en}, \eqref{en1-init}, \eqref{before-en}, \eqref{after-en}).  We wish to run the evolution forward to a later time $t_{N+1}$ (which can be approximately determined using \eqref{lifespan-induct}) for which the energy near scale $N$ has now largely transitioned to the $X_{1,N+1}$ mode (see \eqref{an-init-induct}, \eqref{en-stable-induct}), but with little energy at nearby modes and scales (see \eqref{bn-init-induct}, \eqref{cn-init-induct}, \eqref{en1-init-induct}, \eqref{before-en-induct}, \eqref{dn-en-induct}, \eqref{after-en-induct}).  In particular, the transition of energy to the $X_{1,N+1}$ mode needs to be so abrupt that no significant amount of energy leaks into the $N+2$ modes yet (see \eqref{after-en-induct}).  To establish this, we shall first show (under a bootstrap hypothesis) that all scales other\footnote{This is an oversimplification; one has to take some care to also exclude the possibility of ``constructive interference'' between the $N-1$ and $N$ modes, but this is a technical issue of secondary importance.} than the $N$ and $N+1$ scales are under control, thus largely reducing matters to the study of the dynamics between the $N$ and $N+1$ modes, at which point one can run the analysis used to establish Theorem \ref{daet}.

\subsection{Fourth step: renormalising the dynamics}

It is convenient to perform a rescaling to essentially eliminate the role of the time $t_N$, the energy $e_N$, and the scale $(1+\epsilon_0)^{-N}$, in order to make the dynamics closely resemble those in Theorem \ref{daet}. 
More precisely, Proposition \ref{blowdyn-induct} rescales as follows.

\begin{proposition}[Rescaled inductive step]\label{induct-rescale}  Let $0 < \epsilon_0 < 1$, let $K>0$ be sufficiently large depending on $\epsilon_0$, let $\eps>0$ be sufficiently small depending on $\epsilon_0,K$, and let $n_0$ be sufficiently large depending on $\epsilon_0, K, \eps$, and the implied constants in \eqref{axin-eq-resc}-\eqref{dxin-eq-resc}, \eqref{cain}, \eqref{tau-stable} below.  Let $N \geq n_0$, and suppose we have rescaled times
$$\tau_{n_0-N} < \dots < \tau_0 = 0$$
and continuously differentiable functions $a_k,b_k,c_k,d_k: [\tau_{n_0-N},+\infty) \to \R$ and $\tilde E_k: [\tau_{n_0-N},+\infty) \to [0,+\infty)$ obeying the following properties:
\begin{itemize}
\item[(i)] (A priori regularity) We have
\begin{equation}\label{many-norm}
\sup_{\tau_{n_0-N} \leq t\leq T} \sup_{k \in\Z} \left(1 + (1+\epsilon_0)^{10k}\right) (|a_k(t)|+|b_k(t)|+|c_k(t)|+|d_k(t)|) < \infty
\end{equation}
and
\begin{equation}\label{many-norm-2}
\sup_{\tau_{n_0-N} \leq t\leq T} \sup_{k \in\Z} \left(1 + (1+\epsilon_0)^{10k}\right) \tilde E_k(t)^{1/2} < \infty,
\end{equation}
for any $T > \tau_{n_0-N}$.
\item[(ii)] (Equations of motion)  One has
\begin{align}
\partial_t a_k &= (1+\epsilon_0)^{5k/2} \left(- \eps^{-2} c_k d_k - \eps a_k b_k - \eps^2 \exp(-K^{10}) a_k c_k + K d_{k-1}^2 \right) \nonumber \\
&\quad\quad+ O\left( (1+\epsilon_0)^{2k - n_0/2} \tilde E_k^{1/2} \right) \label{axin-eq-resc} \\
\partial_t b_k &= (1+\epsilon_0)^{5k/2} \left(\eps a_k^2 - \eps^{-1} K^{10} c_k^2 \right) + O\left( (1+\epsilon_0)^{2k - n_0/2} \tilde E_k^{1/2} \right) \label{bxin-eq-resc} \\
\partial_t c_k &= (1+\epsilon_0)^{5k/2} \left(\eps^2 \exp(-K^{10}) a_k^2 + \eps^{-1} K^{10} b_k c_k \right) + O\left( (1+\epsilon_0)^{2k - n_0/2}  \tilde E_k^{1/2} \right) \label{cxin-eq-resc} \\
\partial_t d_k &= (1+\epsilon_0)^{5k/2} \left(\eps^{-2} c_k a_k - (1+\epsilon_0)^{5/2} K d_k a_{k+1} \right) + O\left( (1+\epsilon_0)^{2k - n_0/2} \tilde E_k^{1/2} \right) \label{dxin-eq-resc}
\end{align}
and
\begin{equation}\label{lei}
\partial_t \tilde E_k \leq K (1+\epsilon_0)^{5k/2} ( d_{k-1}^2 a_k - (1+\epsilon_0)^{5/2} d_k^2 a_{k+1} )
\end{equation}
for all $k \in \Z$ and $t \geq \tau_{n_0-N}$.
\item[(iii)] (Initial conditions)  One has
\begin{equation}\label{nomad}
a_k(\tau_{n_0-N})=b_k(\tau_{n_0-N})=c_k(\tau_{n_0-N})=d_k(\tau_{n_0-N})=\tilde E_k(\tau_{n_0-N})=0
\end{equation}
whenever $k > n_0-N$.  
\item[(iv)] (Energy defect) One has
\begin{equation}\label{cain}
\begin{split}
\frac{1}{2} \left(a_k^2(t)+b_k^2(t)+c_k^2(t)+d_k^2(t)\right)  &\leq \tilde E_k(t)\\
& \leq \frac{1}{2} \left(a_k^2(t)+b_k^2(t)+c_k^2(t)+d_k^2(t)\right)\\
&\quad\quad  + O\left( (1+\epsilon_0)^{2k-n_0/2} \int_{\tau_{n_0-N}}^t \tilde E_k(t')\ dt' \right)
\end{split}
\end{equation}
whenever $k \in\Z$ and $t \geq \tau_{n_0-N}$.
\item[(v)] (No very low frequencies) One has
\begin{equation}\label{nomodo}
a_k(t)=b_k(t)=c_k(t)=d_k(t)=\tilde E_k(t)=0
\end{equation}
whenever $k < n_0-N$ and $t \geq \tau_{n_0-N}$.
\item[(vii)] (Scale evolution) One has
\begin{equation}\label{tau-stable}
-O\left((1+\epsilon_0)^{(\frac{5}{2}+\frac{1}{100})|k|}\right) \leq \tau_k \leq 0
\end{equation}
for any $n_0-N \leq k \leq 0$.
\item[(viii)] (Transition state)  We have
\begin{align}
a_0(0) &= 1 \label{an-init-rescaled}\\
|b_0(0)| &\leq 10^{-5} \eps \label{bn-init-rescaled}\\
|c_0(0)| &\leq 10^{-5} \exp(-K^{10}) \eps^2 \label{cn-init-rescaled}\\
c_0(0) &\geq - (1+\epsilon_0)^{-n_0/4} \label{cn-init-rescaleda}\\
|d_0(0)| &\leq K^{-10} \label{dn-init-rescaled}\\
\tilde E_{-1}(0) &\leq K^{-20} \label{en1-init-rescaled}
\end{align}
and if $N > n_0$ we have the additional bounds
\begin{align}
b_{-1}(0) &\geq 10^{-5} \eps \label{spin-1-scaled}\\
b_{-1}(0) &\leq 10^{5} \eps \label{spin-1a-scaled}\\
c_{-1}(0) &\geq \exp( K^{9} ) \eps^2.\label{spin-2-scaled}\\
c_{-1}(0) &\leq \exp( K^{10} ) \eps^2.\label{spin-2a-scaled}
\end{align}
\item[(ix)] (Energy estimates) We have
\begin{align}
\tilde E_{k-m}(t) &\leq K^{-10} (1+\epsilon_0)^{m/10 + |k-1|/50}\hbox{ for all } m \geq 2 \label{before-en-rescaled} \\
\tilde E_{k-1}(t)+\tilde E_k(t) &\leq (1+\epsilon_0)^{|k-1|/50} \label{during-en-rescaled} \\
\tilde E_{k+m}(t) &\leq K^{-30} (1+\epsilon_0)^{-10m + |k-1|/50} \hbox{ for all } m \geq 1 \label{after-en-rescaled}
\end{align}
whenever $n_0-N < k\leq 0$ and $\tau_{k-1} \leq t \leq \tau_k$.
\end{itemize}
Then there exists a time 
\begin{equation}\label{lifespan-rescaled}
\frac{1}{100} \leq \tau_1 \leq 100 
\end{equation}
and an amplitude
\begin{equation}\label{en-stable-rescaled}
(1+\epsilon_0)^{-1/100} \leq \mu_1 \leq (1+\epsilon_0)^{1/100} 
\end{equation}
such that we have the bounds
\begin{align}
a_1(\tau_1) &= \mu_1 \label{an-init-next} \\
|b_1(\tau_1)| &\leq 10^{-5} \eps \mu_1 \label{bn-init-next}\\
|c_1(\tau_1)| &\leq 10^{-5} \exp( -K^{10} ) \eps^2 \mu_1 \label{cn-init-next}\\
c_1(\tau_1) &\geq - (1+\epsilon_0)^{-n_0/4} \mu_1 \label{cn-init-rescaleda-next}\\
|d_1(\tau_1)| &\leq K^{-10} \mu_1 \label{dn-en-next} \\
\tilde E_0(\tau_1) &\leq K^{-20} \mu_1^2. \label{en1-init-next}\\
b_0(\tau_1) &\geq 10^{-5} \eps \mu_1 \label{spin-1-next}\\
b_0(\tau_1) &\leq 10^{5} \eps \mu_1 \label{spin-1a-next}\\
c_0(\tau_1) &\geq \exp( K^{9} ) \eps^2 \mu_1\label{spin-2-next}\\
c_0(\tau_1) &\leq \exp( K^{10} ) \eps^2 \mu_1\label{spin-2a-next}
\end{align}
and
\begin{align}
\tilde E_{1-m}(t) &\leq K^{-10} (1+\epsilon_0)^{m/10} \hbox{ for all } m \geq 2 \label{before-en-next} \\
\tilde E_{0}(t)+\tilde E_{1}(t) &\leq 1 \label{during-en-next} \\
\tilde E_{1+m}(t) &\leq K^{-30} (1+\epsilon_0)^{-10m} \hbox{ for all } m \geq 1 \label{after-en-next}
\end{align}
for all $0 \leq t \leq \tau_1$.
\end{proposition}

\begin{remark} The dynamics \eqref{axin-eq-resc}-\eqref{dxin-eq-resc} are depicted in Figure \ref{fig:circ2} (with the dissipative terms ignored).  Note how the rescaling has placed the tiny factor of $(1+\epsilon_0)^{-n_0/2}$ in front of all the viscosity terms in \eqref{axin-eq-resc}-\eqref{dxin-eq-resc}, thus highlighting the lower order nature of these terms for our analysis.  This small factor is ultimately reflecting the supercritical nature of the dissipation; in practice, this factor will allow us to treat all dissipative terms as negligible. 
\end{remark}

\begin{figure} [t]
\centering
\includegraphics{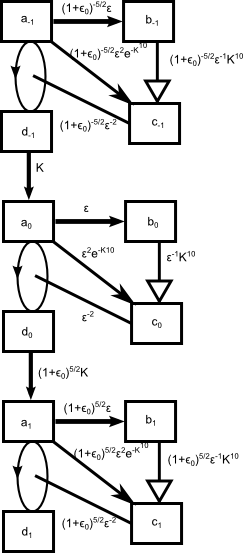}
\caption[Rescaled circuits]{A portion of the rescaled dynamics \eqref{axin-eq-resc}-\eqref{dxin-eq-resc}, again ignoring the dissipation terms.  The factors of $(1+\epsilon_0)$ here play no significant role and may be ignored at a first reading.}
\label{fig:circ2}
\end{figure}

Let us now explain why Proposition \ref{induct-rescale} implies Proposition \ref{blowdyn-induct} (and hence Proposition \ref{blowdyn} and Theorems \ref{ood2}, \ref{ood}, \ref{blowup} and \ref{main}).  Let the notation and hypotheses be as in Proposition \ref{blowdyn-induct}.  We then define the rescaled times
\begin{equation}\label{taudef}
\tau_k := (1+\epsilon_0)^{5N/2} e_N (t_{N+k} - t_N)
\end{equation}
and rescaled amplitudes
\begin{equation}\label{mudef}
\mu_k := e_N^{-1} e_{N+k}  
\end{equation}
for $n_0-N \leq k \leq 0$, as well as the rescaled solutions
\begin{align*}
a_k(t) &:= e_N^{-1} X_{1,N+k}\left( t_N + (1+\epsilon_0)^{-5N/2} e_N^{-1} t \right)  \\
b_k(t) &:= e_N^{-1} X_{2,N+k}\left( t_N + (1+\epsilon_0)^{-5N/2} e_N^{-1} t \right)  \\
c_k(t) &:= e_N^{-1} X_{3,N+k}\left( t_N + (1+\epsilon_0)^{-5N/2} e_N^{-1} t \right) \\
d_k(t) &:= e_N^{-1} X_{4,N+k}\left( t_N + (1+\epsilon_0)^{-5N/2} e_N^{-1} t \right) 
\end{align*}
and rescaled energies
\begin{align*}
\tilde E_k(t) &:= e_N^{-2} E_{N+k}\left( t_N + (1+\epsilon_0)^{-5N/2} e_N^{-1} t \right) 
\end{align*}
for $k \in \Z$ and $t \geq \tau_{n_0-N}$.  Under this rescaling, the \emph{a priori} regularity \eqref{many-norm}, \eqref{many-norm-2} follows from \eqref{many}, \eqref{many-2}.  The bound 
\begin{equation}\label{mu-stable}
(1+\epsilon_0)^{-|k|/100} \leq \mu_{k} \leq (1+\epsilon_0)^{|k|/100}
\end{equation}
for $n_0-N \leq k \leq 0$ follows from \eqref{en-stable} and \eqref{mudef}; from this, \eqref{lifespan}, \eqref{taudef} and summing the geometric series we then obtain \eqref{tau-stable} (recall that we allow implied constants in the $\lesssim$ or $O()$ notation to depend on $\epsilon_0$).

If we directly rescale \eqref{axin-eq}-\eqref{dxin-eq}, we obtain \eqref{axin-eq-resc}-\eqref{dxin-eq-resc}, except with the factors $(1+\epsilon_0)^{2k-n_0/2}$ replaced by 
$(1+\epsilon_0)^{2k - N/2} e_N^{-1}$.  However, from \eqref{en-stable}, \eqref{e-init} we have
$$ e_N^{-1} \leq (1+\epsilon_0)^{(N-n_0)/100}$$
and hence
$$ (1+\epsilon_0)^{2k - N/2} e_N^{-1} \leq (1+\epsilon_0)^{2k-n_0/2}.$$
This gives the equations of motion \eqref{axin-eq-resc}-\eqref{dxin-eq-resc}.  The energy inequality \eqref{lei} is similarly obtained from rescaling \eqref{lei-0}.

The initial conditions \eqref{nomad} follow from rescaling \eqref{initial-cond} (and also using \eqref{t-init}).  Similarly, \eqref{cain} follows from rescaling \eqref{ein-sum}, and \eqref{nomodo} follows from rescaling \eqref{nomode}.  Similarly, the conditions \eqref{an-init-rescaled}-\eqref{spin-2a-scaled} follow from rescaling \eqref{an-init}-\eqref{spin-2a}.  Finally, \eqref{before-en-rescaled}-\eqref{after-en-rescaled} follow from rescaling \eqref{before-en}-\eqref{after-en} and using \eqref{mu-stable}.  We then apply Proposition \ref{induct-rescale} to obtain $\tau_1, \mu_1$ with the stated properties \eqref{lifespan-rescaled}-\eqref{after-en-next}.  It is then routine to verify that the conclusions of Proposition \ref{blowdyn-induct} are satisfied with
$$ t_{N+1} := t_N + (1+\epsilon_0)^{-5N/2} e_N^{-1} \tau_1$$
and
$$ e_{N+1} := \mu_1 e_N.$$

For future reference, we record one consequence of the energy estimates \eqref{before-en-rescaled}-\eqref{after-en-rescaled}:

\begin{lemma}[Cumulative energy bound]\label{cumeng}  For any integer $m = O(1)$, one has
$$ \int_{\tau_{n_0-N}}^0 \tilde E_m(t)\ dt \lesssim 1.$$\
(The implied constant here may depend on $m$.)
\end{lemma}

\begin{proof}
From \eqref{during-en-rescaled}, \eqref{after-en-rescaled} we have
$$ \tilde E_m(t) \lesssim (1+\epsilon_0)^{10k + |k|/50} $$
whenever $\tau_{k-1} \leq t \leq \tau_k$ and $n_0-N < k \leq 0$, and so
$$ \int_{\tau_{n_0-N}}^0 \tilde E_m(t)\ dt \lesssim \sum_{n_0-N < k \leq 0} (1+\epsilon_0)^{10k + |k|/50} |\tau_{k-1}|.$$
Applying \eqref{tau-stable} and summing the geometric series, we obtain the claim.
\end{proof}

\subsection{Fifth step: crude energy estimates for distant modes}

We now begin the proof of Proposition \ref{induct-rescale}.  For the rest of this section, we assume the notations and hypotheses are as in that proposition.

The first stage is to establish the energy bounds \eqref{before-en-next}, \eqref{during-en-next}, \eqref{after-en-next} (and also the bound \eqref{dn-en-next}) on a certain time interval $0 \leq t \leq T_1$; the quantity $\tau_1$ will later be chosen between $0$ and $T_1$, thus establishing the required bounds \eqref{before-en-next}, \eqref{during-en-next}, \eqref{after-en-next}, \eqref{dn-en-next} for $0 \leq t \leq \tau_1$.

We first establish bounds at time $t=0$ that are slightly better than the required bounds \eqref{before-en-next}, \eqref{during-en-next}, \eqref{after-en-next}, \eqref{dn-en-next}.

\begin{lemma}[Initial bounds]\label{prelim}  Let the notation and assumptions be as in Proposition \ref{induct-rescale}.  Then we have
\begin{align}
\tilde E_{1-m}(0) &\leq (1+\epsilon_0)^{-0.08} K^{-10} (1+\epsilon_0)^{m/10} \hbox{ for all } m \geq 2 \label{before-en-0} \\
\tilde E_{0}(0)+\tilde E_{1}(0) &\leq 0.6 \label{during-en-0} \\
\tilde E_{1+m}(0) &\leq (1+\epsilon_0)^{-9.98} K^{-30} (1+\epsilon_0)^{-10m} \hbox{ for all } m \geq 1 \label{after-en-0}\\
|d_1(0)| &\leq \sqrt{2} K^{-15}. \label{dn-en-0} 
\end{align}
\end{lemma}

\begin{proof}
From \eqref{before-en-rescaled}, \eqref{after-en-rescaled} for $k=0$ and $t=0$, we have
\begin{equation}\label{before-1}
\tilde E_{-m}(0) \leq K^{-10} (1+\epsilon_0)^{m/10 + 1/50} \hbox{ for all } m \geq 2
\end{equation}
and
\begin{equation}\label{after-1}
\tilde E_{m}(0) \leq K^{-30} (1+\epsilon_0)^{-10m + 1/50} \hbox{ for all } m \geq 1 
\end{equation}
which implies the claims \eqref{before-en-0} for all $m \geq 3$ and \eqref{after-en-0} for $m \geq 1$, after shifting $m$ by one. The claim \eqref{before-en-0} for $m=2$ follows from \eqref{en1-init-rescaled} (since $K$ is large depending on $\epsilon_0$). Also, from \eqref{after-1} we have
\begin{equation}\label{half}
 \tilde E_1(0) \leq K^{-30};
 \end{equation}
the claim \eqref{dn-en-0} then follows from \eqref{cain}.

It remains to establish \eqref{during-en-0}.  From \eqref{cain} one has
$$ \tilde E_0(0) \leq   \frac{1}{2} \left(a_0^2(0)+b_0^2(0)+c_0^2(0)+d_0^2(0)\right) + O\left( (1+\epsilon_0)^{-n_0/2} \int_{\tau_{n_0-N}}^0 \tilde E_0(t)\ dt \right).$$
From \eqref{an-init-rescaled}-\eqref{dn-init-rescaled} we have
$$ \frac{1}{2} (a_0^2(0)+b_0^2(0)+c_0^2(0)+d_0^2(0))  = 0.5 + O( K^{-10} ).$$
Applying Lemma \ref{cumeng}, and recalling that $K$ and $n_0$ are assumed sufficiently large, the claim \eqref{during-en-0} follows.
\end{proof}

We now define $T_1$ to be the largest time in $[0,100]$ for which one has the bounds
\begin{align}
\tilde E_{1-m}(t) &\leq K^{-10} (1+\epsilon_0)^{m/10} \hbox{ for all } m \geq 2 \label{before-en-next'} \\
\tilde E_{0}(t)+\tilde E_{1}(t) &\leq 1 \label{during-en-next'} \\
\tilde E_{1+m}(t) &\leq K^{-30} (1+\epsilon_0)^{-10m} \hbox{ for all } m \geq 1 \label{after-en-next'}\\
|d_1(t)| &\leq \frac{1}{2} K^{-10}. \label{dn-en-next'} 
\end{align}
for all $0 \leq t \leq T_1$.  Lemma \ref{prelim} ensures that $T_1$ is well-defined (note that all the conditions here are closed conditions in $t$).

We record a variant of the arguments in Lemma \ref{prelim} that will be needed later:

\begin{lemma}[Almost all energy in primary modes]\label{primary}  For any $k=-1,0,1$ and $0 \leq t \leq T_1$, one has
$$ \tilde E_k(t) = \frac{1}{2} \left(a_k(t)^2 + b_k(t)^2 + c_k(t)^2 + d_k(t)^2 \right) + O\left( (1+\epsilon_0)^{-n_0/2} \right).$$
\end{lemma}

\begin{proof} By \eqref{cain} it suffices to show that
$$ \int_{\tau_{n_0-N}}^{T_1} \tilde E_k(t)\ dt \lesssim 1.$$
The portion of the integral with $0 \leq t \leq T_1$ is controlled by \eqref{before-en-next'}, \eqref{during-en-next'}, and the trivial bound $T_1 \leq 100$.  The claim now follows from Lemma \ref{cumeng}.
\end{proof}

The bounds \eqref{before-en-next'}-\eqref{dn-en-next'} look like an infinite number of conditions, but note from the qualitative decay property \eqref{many-norm-2} that
$$
\sup_{0\leq t \leq 100} \sup_{k \in\Z} (1 + (1+\epsilon_0)^{20k}) \tilde E_k(0) < \infty,
$$
which implies that the bounds \eqref{before-en-next'}, \eqref{after-en-next'} are automatically satisfied for all $m \geq M$ and some finite $M$ (independent of $T_1$).  So there are really only a finite number of conditions in the definition of $T_1$.  As $\tilde E_m$ and $d_1$ vary continuously in time, we conclude that either $T_1=100$, or that at least one of the inequalities \eqref{before-en-next'}-\eqref{after-en-next'} is obeyed with equality (for some $m$) at time $t=T_1$.  In particular, from Lemma \ref{prelim} we see that $T_1 \neq 0$, thus $0 < T_1 \leq 100$.

Now we use local energy estimates to rule out several of the ways in which one can ``exit'' the bounds \eqref{before-en-next'}-\eqref{dn-en-next'}.

\begin{lemma}[Closing some exits]  We have
\begin{equation}\label{far-exit}
\tilde E_{1-m}(T_1) < K^{-10} (1+\epsilon_0)^{m/10} \hbox{ for all } m \geq 3
\end{equation}
and
\begin{equation}\label{far2-exit}
\tilde E_{1+m}(T_1) < K^{-30} (1+\epsilon_0)^{-10m} \hbox{ for all } m \geq 1 
\end{equation}
and
\begin{equation}\label{near-exit}
\tilde E_{0}(T_1)+\tilde E_{1}(T_1) < 1.
\end{equation}
\end{lemma}

\begin{proof}  Integrating \eqref{lei} on $[0,T_1]$, we conclude that
\begin{equation}\label{energy-ineq}
\tilde E_k(T_1) \leq \tilde E_k(0) + K (1+\epsilon_0)^{5k/2} \int_0^{T_1} d_{k-1}^2 a_k(t) - (1+\epsilon_0)^{5/2} d_k^2 a_{k+1}(t)\ dt
\end{equation}
for any $k \in \Z$; summing this for $k=0$ and $k=1$ and using \eqref{during-en-0}, we conclude the variant
\begin{equation}\label{energy-ineq2}
 \tilde E_0(T_1) +\tilde E_1(T_1) \leq 0.6 + K \int_0^{T_1} d_{-1}^2 a_0(t) - (1+\epsilon_0)^{5} d_1^2 a_{2}(t)\ dt.
\end{equation}
From \eqref{before-en-next'}, \eqref{during-en-next'}, \eqref{after-en-next'} we have
$$
d_{-1}^2 a_0(t) - (1+\epsilon_0)^{5} d_1^2 a_{2}(t) = O( K^{-5} )$$
and hence from \eqref{energy-ineq2}
$$  \tilde E_0(T_1) +\tilde E_1(T_1) \leq 0.6 + O(K^{-4})$$
giving \eqref{near-exit} for $K$ large enough.

Now suppose that $m \geq 3$.  From \eqref{energy-ineq} with $k=1-m$ and \eqref{before-en-0}, we have
\begin{align*}
\tilde E_{1-m}(T_1) &\leq (1+\epsilon_0)^{-0.08} K^{-10} (1+\epsilon_0)^{m/10} \\
&\quad + O\left( K (1+\epsilon_0)^{-5m/2} \int_0^{T_1} |d_{-m}|^2 |a_{1-m}|(t) + |d_{1-m}|^2 |a_{2-m}(t)|\ dt \right).
\end{align*}
From \eqref{before-en-next'} (and now using the hypothesis $m \geq 3$) we have
$$ |d_{-m}|^2 |a_{1-m}|(t) + |d_{1-m}|^2 |a_{2-m}(t)| \lesssim K^{-15} (1+\epsilon_0)^{3m/20}$$
for $0 \leq t \leq T_1$, and so (since $T_1 \leq 100$)
$$ 
\tilde E_{1-m}(T_1) \leq K^{-10} (1+\epsilon_0)^{m/10} \left((1+\epsilon_0)^{-0.09} + O\left( K^{-4} (1+\epsilon_0)^{-(\frac{5}{2}-\frac{1}{20})m}\right)\right)$$
and \eqref{far-exit} follows (assuming $K$ large enough).  

Similarly, if $m \geq 1$, we may apply \eqref{energy-ineq} with $k=1+m$ and use \eqref{after-en-0} to obtain
$$
\tilde E_{1+m}(T_1) \leq (1+\epsilon_0)^{-9.98} K^{-30} (1+\epsilon_0)^{-10m} + O\left( K (1+\epsilon_0)^{5m/2} \int_0^{T_1} |d_{m}|^2 |a_{m+1}|(t) + |d_{m+1}|^2 |a_{m+2}(t)|\ dt\right).$$
From \eqref{after-en-next'} (and \eqref{dn-en-next'} when $m=1$) we have
$$ |d_{k-1}|^2 |a_k|(t) + |d_k|^2 |a_{k+1}|(t) \lesssim K^{-35} (1+\epsilon_0)^{-15m} $$
for $0 \leq t \leq T_1$, and thus
$$
\tilde E_{1+m}(T_1) \leq K^{-30} (1+\epsilon_0)^{-10m} \left((1+\epsilon_0)^{-9.98} + O\left( K^{-4} (1+\epsilon_0)^{-5m/2} \right) \right)$$
and \eqref{far2-exit} follows (assuming $K$ large enough).  
\end{proof}

From this lemma and the previous discussion, we have some partial control on how we exit the regime:

\begin{corollary}[Exit trichotomy]\label{exit} At least one of the following assertions hold:
\begin{itemize}
\item (Backwards flow of energy) We have
\begin{equation}\label{exit-1}
\tilde E_{-1}(T_1) = K^{-10} (1+\epsilon_0)^{2/10}.
\end{equation}
\item (Forwards flow of energy) We have
\begin{equation}\label{exit-2}
|d_1(T_1)| = \frac{1}{2} K^{-10}.
\end{equation}
\item (Running out the clock) We have
\begin{equation}\label{exit-3}
T_1 = 100.
\end{equation}
\end{itemize}
\end{corollary}

Although we will not need this fact here, it turns out (using a refinement of the analysis below) that it is option \eqref{exit-2} which actually occurs in this trichotomy.

Thanks to \eqref{before-en-next'}-\eqref{dn-en-next'}, the task of proving Proposition \ref{induct-rescale} has now reduced to the following claim:

\begin{proposition}[Reduced induction claim]\label{reduced-claim}  Let the notation and hypotheses be as in Proposition \ref{induct-rescale}, and let $T_1$ be defined as above.  There exists a time
\begin{equation}\label{lifespan-rescaled-2}
\frac{1}{100} \leq \tau_1 \leq T_1
\end{equation}
(in particular, $T_1 \geq 1/100$) and an amplitude
\begin{equation}\label{en-stable-rescaled-2}
(1+\epsilon_0)^{-1/100} \leq \mu_1 \leq (1+\epsilon_0)^{1/100} 
\end{equation}
such that we have the bounds
\begin{align}
a_1(\tau_1) &= \mu_1 \label{an-init-next-2} \\
|b_1(\tau_1)| &\leq 10^{-5} \eps \mu_1 \label{bn-init-next-2}\\
|c_1(\tau_1)| &\leq 10^{-5} \exp( -K^{10} ) \eps^2 \mu_1 \label{cn-init-next-2}\\
c_1(\tau_1) &\geq - (1+\epsilon_0)^{-n_0/4} \mu_1 \label{cn-init-next-2a}\\
\tilde E_0(\tau_1) &\leq K^{-20} \mu_1^2 \label{en1-init-next-2}\\
b_0(\tau_1) &\geq 10^{-5} \eps \mu_1 \label{spin-1-next-2}\\
b_0(\tau_1) &\leq 10^{5} \eps \mu_1 \label{spin-1a-next-2}\\
c_0(\tau_1) &\geq \exp( K^{9} ) \eps^2 \mu_1\label{spin-2-next-2}\\
c_0(\tau_1) &\leq \exp( K^{10} ) \eps^2 \mu_1.\label{spin-2a-next-2}
\end{align}
\end{proposition}

Indeed, the remaining claims \eqref{dn-en-next}, \eqref{before-en-next}, \eqref{during-en-next}, \eqref{after-en-next} of Proposition \ref{induct-rescale} follow for $\tau_1$ obeying \eqref{lifespan-rescaled-2} from \eqref{before-en-next'}-\eqref{dn-en-next'} (using \eqref{en-stable-rescaled-2} to handle the $\mu_1$ factor in \eqref{dn-en-next}).

\subsection{Sixth step: eliminating the role of $b_1,c_1,d_1$}

We now begin the proof of Proposition \ref{reduced-claim}.  Henceforth the notation and assumptions are as in that proposition.

It turns out that we can reduce to the setting in which the dynamics of $b_1,c_1,d_1$ are essentially trivial.  The key proposition is

\begin{proposition}[Small $a_1$ implies small $b_1,c_1,d_1$]\label{todos}  Suppose that $0 \leq \tau \leq T_1$ is a time such that
\begin{equation}\label{sold}
 \int_0^\tau a_1(t)^2\ dt \leq K^{-1/4}
\end{equation}
Then we have the bounds
\begin{align}
|b_1(t)| &\lesssim K^{-1/4} \eps \label{b-small} \\
|c_1(t)| &\lesssim K^{-1/4} \exp( - K^{10}/2 ) \eps^2 \label{c-small}\\
c_1(t)   &\geq - O( (1+\epsilon_0)^{-n_0/3} ) \label{c-small-2}\\
|d_1(t)| &\lesssim K^{-20} \label{d-small}
\end{align}
for all $0 \leq t \leq \tau$.
\end{proposition}

\begin{proof}  We first observe from \eqref{after-en-rescaled} that
$$ \tilde E_1(t) \lesssim K^{-30} (1+\epsilon_0)^{10k}$$
whenever $n_0-N < k\leq 0$ and $\tau_{k-1} \leq t \leq \tau_k$.  From this and \eqref{tau-stable} we conclude the crude bound
\begin{equation}\label{ai}
\tilde E_1(t) \lesssim \frac{K^{-30}}{1+|t|^3}
\end{equation}
whenever $\tau_{n_0-N} \leq t \leq 0$.

Let $\tau'$ be the largest time in $[\tau_{n_0-N},\tau]$ for which
\begin{equation}\label{champ}
\int_{\tau_{n_0-N}}^{\tau'} |b_1(t)|\ dt \leq \frac{1}{10} \eps.
\end{equation}
From continuity we see that either $\tau'=\tau$, or else
\begin{equation}\label{opes}
 \int_{\tau_{n_0-N}}^{\tau'} |b_1(t)|\ dt = \frac{1}{10} \eps.
 \end{equation}
We rule out the latter possibility as follows.  From \eqref{cxin-eq-resc} one has
$$ |\partial_t c_1| \leq O\left(\eps^2 \exp(-K^{10}) \tilde E_1\right) + (1+\epsilon_0)^{5/2} \eps^{-1} K^{10} |b_1| |c_1| + O\left((1+\epsilon_0)^{-n_0/2} \tilde E_1^{1/2}\right)$$
and
$$ \partial_t c_1 \geq - O( (1+\epsilon_0)^{-n_0/2} \tilde E_1 ) - (1+\epsilon_0)^{5/2} \eps^{-1} K^{10} |b_1| |c_1|$$
for all $t \geq \tau_{n_0-N}$, while from \eqref{nomad} one has $c_1(\tau_{n_0-N})=0$.  From Gronwall's inequality and \eqref{champ}, we conclude that
\begin{align*}
 |c_1(t)| &\lesssim \exp( (1+\epsilon_0)^{5/2} \frac{1}{10} K^{10}) \times \\
&\quad \left( \eps^2 \exp(-K^{10}) \int_{\tau_{n_0-N}}^t \tilde E_1(t')\ dt' + O\left( (1+\epsilon_0)^{-n_0/2} \int_{\tau_{n_0-N}}^t \tilde E_1(t')^{1/2}\ dt' \right) \right)
\end{align*}
and
$$c_1(t) \geq - O( \exp( (1+\epsilon_0)^{5/2} \frac{1}{10} K^{10}) (1+\epsilon_0)^{-n_0/2} \int_{\tau_{n_0-N}}^t \tilde E_1(t')^{1/2}\ dt' )$$
for any $\tau_{n_0-N} \leq t \leq \tau'$.  In particular, from \eqref{ai} and \eqref{during-en-next'} we have
\begin{equation}\label{dream}
 |c_1(t)| \lesssim \frac{\eps^2\exp(-3K^{-10}/4)}{1+|t|^2}
 \end{equation}
and
\begin{equation}\label{dream-2}
 c_1(t) \geq - O\left( \exp\left( O(K^{10}) (1+\epsilon_0)^{-n_0/2} \right)\right)
\end{equation}
for all $\tau_{n_0-N} \leq t \leq \tau'$ (here we use the trivial bound $\tau' \leq \tau \leq T_1 \leq 100$). 

In a similar spirit, from \eqref{bxin-eq-resc} one has
$$|\partial_t b_1| \leq \eps a_1^2 + \eps^{-1} K^{10} c_1^2 + O\left( (1+\epsilon_0)^{-n_0} \tilde E_1^{1/2} \right) $$
for all $t \geq \tau_{n_0-N}$, and hence by \eqref{nomad}
$$ |b_1(t)| \leq \int_{\tau_{n_0-N}}^t \eps a_1^2(t') + \eps^{-1} K^{10} c_1^2(t') + O\left((1+\epsilon_0)^{-n_0} \tilde E_1^{1/2}\right) \ dt'$$
In particular, from \eqref{dream}, \eqref{ai}, \eqref{sold} we have
$$ |b_1(t)| \lesssim \frac{K^{-30}}{1+|t|^2} \eps$$
for $\tau_{n_0-N} \leq t \leq 0$, and
\begin{equation}\label{wake}
|b_1(t)| \lesssim K^{-1/4} \eps 
\end{equation}
for $0 \leq t \leq \tau'$.  However, this is inconsistent with \eqref{opes} if $K$ is small enough (recalling that $\tau' \leq 100$).  Thus $\tau'=\tau$.  The bounds \eqref{b-small}, \eqref{c-small}, \eqref{c-small-2} now follow from \eqref{dream}, \eqref{dream-2}, \eqref{wake}.

Finally, from \eqref{dxin-eq-resc} and \eqref{c-small}, \eqref{during-en-next'}, \eqref{after-en-next'} we have
$$
\partial_t d_1 = O\left( \exp\left( -K^{10}/2 \right) \right) + O( K^{-14} |d_1| ) 
$$
for $0 \leq t \leq \tau'$ (taking $n_0$ large enough), and from this, \eqref{dn-en-0}, and Gronwall's inequality one obtains \eqref{d-small} (for $K$ large enough).
\end{proof}

Let $T_2$ be the largest time in $[0,T_1]$ such that
$$
 \int_0^{T_2} a_1(t)^2\ dt \leq K^{-1/4}.
$$

Combining Proposition \ref{todos} with Corollary \ref{exit} and using continuity, we conclude

\begin{corollary}[Exit trichotomy, again]\label{exit-again} At least one of the following assertions hold:
\begin{itemize}
\item (Backwards flow of energy) One has
\begin{equation}\label{exit-1a}
\tilde E_{-1}(T_2) = K^{-10} (1+\epsilon_0)^{2/10}.
\end{equation}
\item (Forwards flow of energy) One has
\begin{equation}\label{exit-2a}
 \int_0^{T_2} a_1(t)^2\ dt = K^{-1/4}.
\end{equation}
\item (Running out the clock) One has
\begin{equation}\label{exit-3a}
T_2 = 100.
\end{equation}
\end{itemize}
\end{corollary}

Again, it turns out that it is option \eqref{exit-2a} that actually occurs, although we will not quite prove (or use) this assertion here.

The most important modes for the remainder of the analysis are $a_0,b_0,c_0,d_0$, and $a_1$.
From \eqref{axin-eq-resc}-\eqref{dxin-eq-resc}, the energy bounds \eqref{before-en-next'}-\eqref{after-en-next'}, and Proposition \ref{todos}, we observe the equations of motion
\begin{align}
\partial_t a_0 &= - \eps^{-2} c_0 d_0 + O\left( K^{-9} \right)  \label{axin-eq-0a} \\
\partial_t b_0 &= \eps a_0^2 - \eps^{-1} K^{10} c_0^2 + O\left( (1+\epsilon_0)^{- n_0/2}\right) \label{bxin-eq-0a} \\
\partial_t c_0 &= \eps^2 \exp(-K^{10}) a_0^2 + \eps^{-1} K^{10} b_0 c_0 + O\left( (1+\epsilon_0)^{- n_0/2} \right) \label{cxin-eq-0a} \\
\partial_t d_0 &= \eps^{-2} c_0 a_0 - (1+\epsilon_0)^{5/2} K d_0 a_1 + O\left( (1+\epsilon_0)^{- n_0/2} \right) \label{dxin-eq-0a}\\
\partial_t a_1 &= (1+\epsilon_0)^{5/2} K d_0^2 + O( K^{-1} |a_1| ) + O\left( K^{-20} \right) \label{axin-eq-1a} 
\end{align}
for these modes in the time interval $0 \leq t \leq T_2$.  When $N > n_0$, we also need to keep some track of the modes $a_{-1}, b_{-1}, c_{-1}, d_{-1}$; again from \eqref{axin-eq-resc}-\eqref{dxin-eq-resc} and \eqref{before-en-next'}-\eqref{after-en-next'}, these equations may be given as
\begin{align}
\partial_t a_{-1} &= - (1+\epsilon_0)^{-5/2} \eps^{-2} c_{-1} d_{-1} + O\left( K^{-9} \right)  \label{axin-eq-minus} \\
\partial_t b_{-1} &= (1+\epsilon_0)^{-5/2} \eps a_{-1}^2 - (1+\epsilon_0)^{-5/2} \eps^{-1} K^{10} c_{-1}^2 + O\left( (1+\epsilon_0)^{- n_0/2}\right) \label{bxin-eq-minus} \\
\partial_t c_{-1} &= (1+\epsilon_0)^{-5/2} \eps^2 \exp(-K^{10}) a_{-1}^2 + (1+\epsilon_0)^{-5/2} \eps^{-1} K^{10} b_{-1} c_{-1} + O\left( (1+\epsilon_0)^{- n_0/2} \right) \label{cxin-eq-minus} \\
\partial_t d_{-1} &= (1+\epsilon_0)^{-5/2} \eps^{-2} c_{-1} a_{-1} - K d_{-1} a_{0} + O\left( (1+\epsilon_0)^{- n_0/2} \right). \label{dxin-eq-minus}
\end{align}
The dynamics of these variables $a_{-1},b_{-1},c_{-1},d_{-1}$, do not directly impact the dynamics in \eqref{axin-eq-0a}-\eqref{axin-eq-1a}; however we will still need to track these variables in order to prevent a premature exit of the form \eqref{exit-1a} that could potentially be caused by energy flowing back from $a_0$ to $d_{-1}$.

The task of proving Proposition \ref{reduced-claim} has now reduced further, to that of establishing the following claim.

\begin{proposition}[Reduced induction claim, II]\label{reduced-claim-2}  Let the notation and hypotheses be as in Proposition \ref{induct-rescale}, and let $T_1$ and $T_2$ be defined as above.  There exists a time
\begin{equation}\label{lifespan-rescaled-3}
\frac{1}{100} \leq \tau_1 \leq T_2
\end{equation}
(in particular, $T_2 \geq 1/100$)
such that we have the bounds
\begin{align}
(1+\epsilon_0)^{-1/100} \leq a_1(\tau_1) &\leq (1+\epsilon_0)^{1/100} \label{an-init-next-3}\\
b_0(\tau_1) &\geq 2 \times 10^{-5} \eps \label{spin-1-next-3}\\
b_0(\tau_1) &\leq \frac{1}{2} 10^{5} \eps \label{spin-1a-next-3}\\
c_0(\tau_1) &\geq 2 \times \exp( K^{9} ) \eps^2 \label{spin-2-next-3}\\
c_0(\tau_1) &\leq \frac{1}{2} \exp( K^{10} ) \eps^2 \label{spin-2a-next-3}\\
\tilde E_0(\tau_1) &\leq \frac{1}{2} K^{-20}, \label{en1-init-next-3}
\end{align}
\end{proposition}

Indeed, Proposition \ref{reduced-claim} follows from Proposition \ref{reduced-claim-2} and Proposition \ref{todos} once we set $\mu_1 := a_1(\tau_1)$ (and take $K$ sufficiently large, $\eps$ sufficiently small, and $n_0$ sufficiently large).

\subsection{Seventh step: dynamics at the zero scale}

We now prove Proposition \ref{reduced-claim-2} (and hence Propositions \ref{reduced-claim},  \ref{blowdyn-induct}, \ref{blowdyn} and Theorems \ref{ood2}, \ref{ood}, \ref{blowup} and \ref{main}).

The task at hand is now very close to the situation in Theorem \ref{daet}, and we will now repeat the proof of that theorem with minor modifications, except for a technical distraction having to do with eliminating a premature exercise of the option \eqref{exit-1a}, which requires some analysis of the $-1$-scale dynamics.

From \eqref{during-en-next'} we have
\begin{equation}\label{ab}
a_0(t), b_0(t), c_0(t), d_0(t), a_1(t) = O(1)
\end{equation}
for all $0 \leq t \leq T_2$.  Actually, we can do a bit better than this.  From \eqref{axin-eq-0a}-\eqref{axin-eq-1a} and \eqref{ab} we have
$$ \partial_t (a_0^2+b_0^2+c_0^2+d_0^2+a_1^2) = O( K^{-1} ) $$
for $0 \leq t \leq T_2$
(if $n_0$ is large enough), whereas from \eqref{an-init-rescaled}-\eqref{dn-init-rescaled} and \eqref{after-en-rescaled} we have
$$ a_0(0)^2+b_0(0)^2+c_0(0)^2+d_0(0)^2+a_1(0)^2 = 1 + O(K^{-20}).$$
By the fundamental theorem of calculus, we conclude that
\begin{equation}\label{ab-refine}
a_0(t)^2+b_0(t)^2+c_0(t)^2+d_0(t)^2+a_1(t)^2 = 1 + O(K^{-1})
\end{equation}
for all $0 \leq t \leq T_2$.

Now (as in the proof of Theorem \ref{daet}) we obtain improved bounds on $b,c$.  From \eqref{bxin-eq-0a}, \eqref{cxin-eq-0a}, \eqref{ab} one has
$$ 
\partial_t (b_0^2+c_0^2) = 
2 \eps a_0^2 b_0 + 2 \eps^2 \exp(-K^{10}) a_0^2 c_0 + O\left( (1+\epsilon_0)^{-n_0/2} (b_0^2+c_0^2)^{1/2} \right)$$
for all $0 \leq t \leq T_2$, and thus by \eqref{ab-refine}
$$ 
\partial_t (b_0^2+c_0^2)^{1/2} \leq (1 + O(K^{-1})) \eps$$
for all $0 \leq t \leq T_2$ (interpreting the derivative in a weak sense).
On the other hand, from \eqref{bn-init-rescaled}, \eqref{cn-init-rescaled} we have
$$ (b_0(0)^2+c_0(0)^2)^{1/2} \leq (10^{-5} + O(K^{-1})) \eps.$$
From the fundamental theorem of calculus, we conclude that
\begin{equation}\label{ab-2}
|b_0(t)|, |c_0(t)| \leq (10^{-5} + t + O(K^{-1})) \eps
\end{equation}
for all $0 \leq t \leq T_2$.  Inserting this (and \eqref{ab-refine}) into \eqref{cxin-eq-0a}, we obtain
$$ |\partial_t c_0| \leq O( \eps^2 \exp(-K^{10}) ) + \left(10^{-5}+t+O(K^{-1})\right) K^{10} |c_0|$$
for all $0 \leq t \leq T_2$.
In particular, by \eqref{cn-init-rescaled} and Gronwall's inequality, we have the bound
\begin{equation}\label{code2}
 |c_0(t)| \lesssim \eps^2 \exp\left( K^{10} \left( - \frac{1}{4} + 10^{-5} t + \frac{1}{2} t^2 + O(K^{-1}) \right) \right)
\end{equation}
for all $0 \leq t \leq T_2$.
Finally, from \eqref{dxin-eq-0a}, \eqref{axin-eq-1a} we have
$$ \partial_t \left(d_0^2 + a_1^2\right) = 2 \eps^{-2} c_0 a_0 d_0 + O\left( K^{-1} \left(d_0^2+a_1^2\right) \right) + O\left( K^{-20} \left(d_0^2+a_1^2\right)^{1/2} \right) $$
for all $0 \leq t \leq T_2$,
and hence by \eqref{ab}
\begin{equation}\label{explorer}
 \partial_t \left(d_0^2 + a_1^2\right)^{1/2} = O( \eps^{-2} |c_0|)  + O\left( K^{-1} \left(d_0^2+a_1^2\right)^{1/2} \right) + O( K^{-20} )
\end{equation}
for all $0 \leq t \leq T_2$ (interpreted in a weak sense).
From \eqref{dn-init-rescaled}, \eqref{after-en-rescaled} we have
\begin{equation}\label{slam}
 \left(d_0(0)^2+a_1(0)^2\right)^{1/2} = O( K^{-10} )
\end{equation}
and hence by \eqref{code2} and Gronwall's inequality
\begin{equation}\label{domo}
 |d_0(t)|, |a_1(t)| \lesssim  \exp\left( K^{10} \left( - 1/4 + 10^{-5} t + \frac{1}{2} t^2 + O(K^{-1}) \right) \right) + K^{-10}
 \end{equation}
 for all $0 \leq t \leq T_2$.
Inserting this bound into \eqref{axin-eq-0a}, we see that
$$
|\partial_t a_0| \lesssim  \exp\left( 2K^{10} \left( - 1/4 + 10^{-5} t + \frac{1}{2} t^2 + O\left(K^{-1}\right) \right) \right) + K^{-9}$$
for all $0 \leq t \leq T_2$,
which among other things implies (from \eqref{an-init-rescaled}) that $a_0(t) \geq 0$ whenever $0 \leq t \leq \min(T_2,1/2)$.  From the $k=-1$ case of \eqref{lei}, we thus have
$$ \partial_t \tilde E_{-1} \leq K (1+\epsilon_0)^{-5/2} d_{-2}^2 a_{-1}$$
for $0 \leq t \leq \min(T_2,1/2)$; by \eqref{before-en-next'} we conclude that
$$ \partial_t \tilde E_{-1} \leq O( K^{-14} )$$
on this interval, and hence by \eqref{before-en-0}
$$ \tilde E_{-1}\left(\min(T_2,1/2)\right) \leq  K^{-10} \left(1+\epsilon_0\right)^{2/10} \left(\left(1+\epsilon_0\right)^{-0.08} + O\left( K^{-4} \right)\right),$$
which rules out the first option of Corollary \ref{exit-again} if $T_2 \leq 1/2$.  The second option of this corollary is also ruled out when $T_2 \leq 1/2$, thanks to \eqref{domo}.  We conclude that
\begin{equation}\label{tc-below}
T_2 \geq 1/2.
\end{equation}

Now we sharpen the bounds on $a_0(t), b_0(t), c_0(t), d_0(t), a_1(t)$.  Let $t_c$ be the supremum of all the times $t \in [0,T_2]$ for which $|c(t')| \leq K^{-10} \eps^2$ for all $0 \leq t \leq t'$, thus 
$$0 \leq t_c \leq T_2 \leq T_1 \leq 100$$
 and
\begin{equation}\label{boots2}
|c_0(t)| \leq K^{-10} \eps^2
\end{equation}
for all $0 \leq t \leq t_c$.  Comparing this with \eqref{code2}, \eqref{tc-below}, we conclude that
\begin{equation}\label{cando}
t_c \geq 1/2.
\end{equation}
From \eqref{explorer}, \eqref{slam}, \eqref{boots2}, and Gronwall's inequality one has
\begin{equation}\label{ak10}
 |d_0(t)|, |a_1(t)| \lesssim K^{-10}
\end{equation}
for all $0 \leq t \leq t_c$.  Inserting these bounds and \eqref{ab-2} back into \eqref{axin-eq-0a}, we see that
$$ \partial_t a_0 = O( K^{-9} ) $$
for $0 \leq t \leq t_c$, and thus by \eqref{an-init-rescaled} we have
\begin{equation}\label{aorta}
 a_0(t) = 1 + O(K^{-9})
\end{equation}
for $0 \leq t \leq t_c$.  Inserting this into \eqref{bxin-eq-0a} and using \eqref{boots2}, we conclude that
$$ \partial_t b_0(t) = \eps \left(1 + O\left(K^{-9}\right)\right) $$
for $0 \leq t \leq t_c$, and hence by \eqref{bn-init-rescaled}
\begin{equation}\label{stan-lee}
 \eps\left(t - 10^{-5} - O\left(K^{-9}\right)\right) \leq b_0(t) \leq \eps\left(t + 10^{-5} + O\left(K^{-9}\right)\right)
\end{equation}
for all $0 \leq t \leq t_c$.  Meanwhile, inserting \eqref{aorta} into \eqref{cxin-eq-0a}, we obtain
$$
\partial_t c_0(t) \geq \left(1 + O(K^{-9})\right) \eps^2 \exp(-K^{10}) + \eps^{-1} K^{10} b_0(t) c_0(t) $$
for all $0 \leq t \leq t_c$, and hence by 
\eqref{cn-init-rescaleda}, \eqref{stan-lee} and Gronwall's inequality we see that
$$
c_0(t) \gtrsim \exp\left( \left(\frac{1}{2} t^2 - 10^{-5} t - 1 + O(K^{-9})\right) K^{10} \right) \eps^2$$
whenever $1/2 \leq t \leq t_c$.  Comparing this with \eqref{boots2} we see that
\begin{equation}\label{tc-bound}
 t_c \leq 2
\end{equation}
(say), which by definition of \eqref{boots2} implies that
\begin{equation}\label{c-bound-happy}
 c_0(t_c) = K^{-10} \eps^2.
\end{equation}

Having described the evolution up to time $t_c$, we now move to the future of $t_c$, and specifically in the interval $[t_c,\tau_1]$ where
\begin{equation}\label{t3-def}
\tau_1 := \min( t_c + K^{-1/2}, T_2 ).
\end{equation}
From \eqref{ak10} we have
$$ \int_0^{t_c} a_1(t)^2\ dt \lesssim K^{-10}$$
and hence by \eqref{during-en-next'} and \eqref{t3-def} we have
$$ \int_0^{t} a_1(t)^2\ dt < K^{-1/4}$$
whenever $t \leq \tau_1$.  From this and Corollary \ref{exit-again} (and \eqref{tc-bound}) we conclude

\begin{proposition}[Exit dichotomy]\label{exit-options} At least one of the following assertions hold:
\begin{itemize}
\item (Backwards flow of energy) One has
\begin{equation}\label{exit-1b}
\tilde E_{-1}(\tau_1) = K^{-10} (1+\epsilon_0)^{2/10}.
\end{equation}
\item (Running out the clock) One has
\begin{equation}\label{exit-3b}
\tau_1 = t_c + K^{-1/2}.
\end{equation}
\end{itemize}
\end{proposition}

We will shortly eliminate the option \eqref{exit-1b}, but first we need more control on the dynamics.

From \eqref{stan-lee} and \eqref{cando} we have
$$ b_0(t_c) \geq 10^{-1} \eps.$$
Meanwhile, from \eqref{code2}, \eqref{bxin-eq-0a} (discarding the non-negative $\eps a_0^2$ term) we have
$$
\partial_t b_0 \geq - O\left( \eps^3 \exp\left( O\left( K^{10} \right) \right) \right)$$
for all $t_c \leq t \leq \tau_1$ (with $n_0$ large enough); we conclude (for $\eps$ small enough) that
\begin{equation}\label{botany}
 b_0(t) \geq 10^{-4} \eps
\end{equation}
(say) for all $t_c \leq t \leq \tau_1$.  Inserting this bound into \eqref{cxin-eq-0a}, and discarding the non-negative $\eps^2 \exp(-K^{10}) a_0^2$ term, we see from \eqref{c-bound-happy} and a continuity argument that
\begin{equation}\label{gargle}
 c_0(t) \geq K^{-10} \eps^2
\end{equation}
for $t_c \leq t \leq \tau_1$ (in particular, $c_0$ is positive on this interval), and furthermore that we have the exponential growth
\begin{equation}\label{chap}
 \partial_t c_0(t) \gtrsim K^{10} c_0(t)
\end{equation}
for $t_c \leq t \leq \tau_1$.  We conclude that
\begin{equation}\label{c-large-again}
c_0(t) \geq K^{100} \eps^2
\end{equation}
for $t$ in the interval $I := [t_c+K^{-9}, \tau_1]$.  (We have not yet ruled out the possibility that this time interval is empty, although we will shortly show that this is not the case.)  In the opposite direction, we see from \eqref{ab}, \eqref{ab-2}, \eqref{gargle}, \eqref{cxin-eq-0a} that
\begin{equation}\label{solace-again}
\partial_t c_0 \lesssim K^{10} c_0 
\end{equation}
for $t_c \leq t \leq \tau_1$.  From \eqref{c-bound-happy}, \eqref{t3-def}, and Gronwall's inequality, we thus have the upper bound
\begin{equation}\label{co-upper}
c_0(t) \lesssim \exp\left( O( K^{10 - 1/2} ) \right) \eps^2
\end{equation}
for $t_c \leq t \leq \tau_1$.  Crucially, this upper bound will be significantly smaller than a lower bound for $c_{-1}$ in the same interval, leading to an important mismatch in speeds between the $0$-scale and $-1$-scale dynamics that prevents a premature exit via \eqref{exit-1b}.  More precisely, we have

\begin{proposition}[No exit to coarse scales]\label{noexit}  We have
\begin{equation}\label{toast}
\tilde E_{-1}(t) \lesssim K^{-14}
\end{equation}
for all $t_c \leq t \leq \tau_1$.
In particular, by Proposition \ref{exit-options} we have
\begin{equation}\label{tau-tc}
\tau_1 = t_c + K^{-1/2}
\end{equation}
and hence the interval $I = [t_c+K^{-9}, \tau_1]$ is non-empty.
\end{proposition}

\begin{proof}  If $N=n_0$ then this is immediate from \eqref{nomodo}, so we may assume that $N > n_0$.  In particular, the bounds \eqref{spin-1-scaled}-\eqref{spin-2a-scaled} are available.

We will need some additional bounds on $b_{-1}, c_{-1}$.
From \eqref{bxin-eq-minus}, \eqref{before-en-next'} (discarding the second term in \eqref{bxin-eq-minus} as being non-positive) we have
$$ \partial_t b_{-1}(t) \leq O( \eps )$$
for all $0 \leq t \leq \tau_1$.  From this and \eqref{spin-1a-scaled}, we have
\begin{equation}\label{bmn}
 b_{-1}(t) \leq O( \eps )
\end{equation}
for $0 \leq t \leq \tau_1$.  Meanwhile, from \eqref{cxin-eq-minus} (using \eqref{before-en-next'} to bound $a_{-1}$) we have
$$ \partial_t |c_{-1}|^2 = O( \eps^2 |c_{-1}| ) + 2 (1+\epsilon_0)^{-5/2} \eps^{-1} K^{10} b_{-1} |c_{-1}|^2 $$
and thus by \eqref{bmn}
$$ \partial_t |c_{-1}|^2(t) \leq O( \eps^2 |c_{-1}| ) + O\left( K^{10} |c_{-1}|^2 \right)$$
and thus
$$ \partial_t |c_{-1}|(t) \leq O( \eps^2 ) + O\left( K^{10} |c_{-1}| \right).$$
By Gronwall's inequality and \eqref{spin-2a-scaled}, we thus have
\begin{equation}\label{cibosh}
 |c_{-1}(t)| \lesssim \exp\left( O\left( K^{10} \right) \right) \eps^2
\end{equation}
for $0 \leq t \leq \tau_1$.
Inserting this back into \eqref{bxin-eq-minus}, we see that
$$ \partial_t b_{-1} \geq - O\left( \exp\left( O\left( K^{10} \right) \right) \right) \eps^3$$
and hence by \eqref{spin-1-scaled}
$$ b_{-1}(t) \geq 10^{-6} \eps$$
for $0 \leq t \leq \tau_1$ (if $\eps$ is small enough).  Inserting this into \eqref{cxin-eq-minus}, we see that
$$ \partial_t c_{-1} \geq 10^{-6} K^{10} c_{-1} - O\left( (1+\epsilon_0)^{- n_0/2} \right)$$
for $0 \leq t \leq \tau_1$, and hence by \eqref{spin-2-scaled}
\begin{equation}\label{cmn}
 c_{-1}(t) \geq \exp( 10^{-8} K^{10} ) \eps^2
 \end{equation}
for $1/10 \leq t \leq \tau_1$ (if $n_0$ is large enough).  Returning to \eqref{bxin-eq-minus}, we now have (thanks to \eqref{cibosh}) that
$$ \partial_t b_{-1}(t) = O( \eps )$$
for $0 \leq t \leq \tau_1$, and thus by \eqref{spin-1a-scaled}
\begin{equation}\label{bibosh}
 b_{-1}(t) = O( \eps )
\end{equation}
for $0 \leq t \leq \tau_1$; from \eqref{cxin-eq-minus}, \eqref{cmn}, \eqref{before-en-next'} we now have
\begin{equation}\label{c-bang}
\partial_t c_{-1}(t) = O( K^{10} c_{-1}(t) )
\end{equation}
for $0 \leq t \leq \tau_1$.

From \eqref{before-en-0} we have
\begin{equation}\label{eop}
\tilde E_{-1}(0) \lesssim K^{-10}.
\end{equation}
This bound is too big for \eqref{toast}, so we will first need to establish some decay in $\tilde E_{-1}$ as one moves from time $t=0$ to time $t=t_c$.
From \eqref{lei} we have
$$ \partial_t \tilde E_{-1} \leq K (1+\epsilon_0)^{-5/2} d_{-2}^2 a_{-1} - K d_{-1}^2 a_0 $$
From \eqref{before-en-next'} we have $d_{-2}^2 a_{-1}(t) = O(K^{-15})$ for $0 \leq t \leq \tau_1$, so 
\begin{equation}\label{chimp}
 \partial_t \tilde E_{-1} \leq - K d_{-1}^2 a_0 + O( K^{-14} )
\end{equation}
for $0 \leq t \leq \tau_1$.
Equipartition of energy suggests that $d_{-1}^2$ oscillates around $\tilde E_{-1}$ on the average.  To formalise this, observe from \eqref{axin-eq-minus}, \eqref{dxin-eq-minus}, \eqref{before-en-next'} that
$$
\partial_t (a_{-1} d_{-1}) = (1+\epsilon_0)^{-5/2} \eps^{-2} c_{-1} \left(a_{-1}^2 - d_{-1}^2\right)  - K a_{-1} d_{-1} a_0 + O(K^{-14}) $$
and hence by the product rule
\begin{align*}
\partial_t \left(a_{-1} d_{-1} \frac{\eps^2}{c_{-1}} a_0\right)& = (1+\epsilon_0)^{-5/2} \left(a_{-1}^2 - d_{-1}^2\right) a_0 - \frac{\eps^2}{c_{-1}} \frac{\partial_t c_{-1}}{c_{-1}} a_{-1} d_{-1} - \frac{\eps^2}{c_{-1}} K a_{-1} d_{-1} a_0^2 \\
&\quad + a_{-1} d_{-1} \frac{\eps_2}{c_{-1}} \partial_t a_0 + O\left(K^{-14} \frac{\eps^2}{c_{-1}}\right)
\end{align*}
for $0 \leq t \leq \tau_1$.
from \eqref{axin-eq-0a}, \eqref{during-en-next'}, \eqref{co-upper} we have
$$ \partial_t a_0 = O\left( \exp\left( O\left( K^{10 - 1/2} \right) \right) \right).$$
Using \eqref{cmn}, \eqref{c-bang}, \eqref{before-en-next}, we conclude that
$$ \partial_t \left(a_{-1} d_{-1} \frac{\eps^2}{c_{-1}} a_0\right) = (1+\epsilon_0)^{-5/2} \left(a_{-1}^2 - d_{-1}^2\right) a_0 + O\left(K^{-100}\right)$$
for $0 \leq t \leq T_*$.
If we define the modified energy
$$ E^* := \tilde E_{-1} - \frac{1}{2} (1+\epsilon_0)^{5/2} K a_{-1} d_{-1} \frac{\eps^2}{c_{-1}} a_0$$
then from \eqref{cmn}, \eqref{before-en-next'} we have
\begin{equation}\label{vamonos}
E^* = \tilde E_{-1} + O(K^{-100})
\end{equation}
while from \eqref{chimp} we have
$$ \partial_t E^* \leq -\frac{1}{2} K (1+\epsilon_0)^{-5/2} \left(a_{-1}^2 + d_{-1}^2\right) a_0 + O\left(K^{-14} \right)$$
for $0 \leq t \leq \tau_1$.
By Lemma \ref{primary}, \eqref{bibosh}, \eqref{cibosh} we have
$$ \tilde E_{-1} = \frac{1}{2} \left(a_{-1}^2 + d_{-1}^2\right) + O(\eps^2)$$
and thus by \eqref{vamonos}
$$ \partial_t E^* \leq -\frac{1}{2} K (1+\epsilon_0)^{-5/2} E^* a_0 + O\left(K^{-14} \right)$$
From \eqref{eop} we have $E^*(0) \lesssim K^{-10}$, and by \eqref{vamonos} it will suffice to show that $E^*(t) \lesssim K^{-14}$ for $t_c \leq t \leq \tau_1$.  By Gronwall's inequality, it thus suffices to show that
$$ \int_0^t a_0(t')\ dt' \gtrsim 1$$
for all $t_c \leq t \leq \tau_1$.  But from \eqref{during-en-next'} and the bound $\tau_1 \leq t_c + K^{-1/2}$ we have
$$ \int_0^t a_0(t')\ dt'  = \int_0^{t_c} a_0(t')\ dt' + O(K^{-1/2})$$
and the claim now follows from \eqref{aorta} and \eqref{cando}.
\end{proof}

We now resume the analysis of the $0$-scale modes. 
From Proposition \ref{noexit} and \eqref{axin-eq-resc} (and \eqref{during-en-next'}), we see that we can improve the error term in \eqref{axin-eq-0a} to
\begin{equation}\label{achoo}
\partial_t a_0 = - \eps^{-2} c_0 d_0 + O\left( K^{-14} \right) 
\end{equation}
in the interval $t_c \leq t \leq \tau_1$.  This improvement will be needed in order to close the bootstrap argument.

 From \eqref{bxin-eq-0a}, \eqref{during-en-next'}, \eqref{co-upper} we have
$$ |\partial_t b_0| \leq 10 \eps$$
for $t_c \leq t \leq \tau_1$, and hence by \eqref{stan-lee} and \eqref{botany} we have
\begin{equation}\label{bilboa}
| b_0| \leq 10^4 \eps
\end{equation}
for $t_c \leq t \leq \tau_1$.  Inserting this into \eqref{cxin-eq-0a} and using \eqref{c-large-again}, \eqref{during-en-next'} we have
$$
|\partial_t c_0(t)| \lesssim K^{10} c_0(t) 
$$
for $t_c \leq t \leq \tau_1$; combining this with \eqref{solace-again} we have
\begin{equation}\label{solace-final}
K^{10} c_0(t) \lesssim \partial_t c_0(t) \lesssim K^{10} c_0(t). 
\end{equation}

Now we use equipartition of energy to establish some energy drain from $a_0,d_0$ to $a_1$.  From \eqref{achoo}, \eqref{dxin-eq-0a}, \eqref{during-en-next'}, \eqref{after-en-next'} one has
\begin{equation}\label{sanitary}
 \partial_t \frac{1}{2}(a_0^2+d_0^2) = - (1+\epsilon_0)^{5/2} K d_0^2 a_1 + O\left( K^{-14} \left(a_0^2+d_0^2\right)^{1/2} \right) + O( K^{-100} )
\end{equation}
and
$$ \partial_t \left(a_0 d_0\right) = \eps^{-2} c_0 \left(a_0^2 - d_0^2\right) + O( K )$$
for $t \in I$.
Meanwhile, from \eqref{axin-eq-1a}, \eqref{during-en-next'} we have
$$ \partial_t a_1 = O(K).$$
From this and \eqref{c-large-again}, \eqref{solace-final} that
\begin{equation}\label{douse-again}
\begin{split}
\partial_t\left( a_0d_0 \frac{\eps^2}{c_0} a_1\right) &= - (a_0^2-d_0^2) a_1 - a_0d_0 \frac{\eps^2}{c_0} \frac{\partial_t c_0}{c_0} a_1 + a_0d_0 \frac{\eps^2}{c_0} \partial_t a_0 + O(K^{-100}) \\
&= - (a_0^2-d_0^2) a_1 + O(K^{-100})
\end{split}
\end{equation}
for $t \in I$,
so if we define the modified energy
$$
E_* := \frac{1}{2}\left(a_0^2+d^2_0\right) + \frac{1}{2} K a_0d_0 \frac{\eps^2}{c_0} a_1
$$
then from \eqref{c-large}, \eqref{during-en-next'}, \eqref{after-en-next'} we have
\begin{equation}\label{est-again}
E_* = \tilde E_0 + O(K^{-100}) = \frac{1}{2} \left(a_0^2+d_0^2\right) + O(K^{-100}) 
\end{equation}
and
$$ \partial_t E_* = - \frac{1}{2} K \left(a_0^2+d_0^2\right) a_1 + O\left( K^{-14} \left(a_0^2+d_0^2\right)^{1/2} \right) + O( K^{-99} )$$
for $t \in I$, and hence
$$ \partial_t E_* = - K a_1 E_* + O\left( K^{-14} E_*^{1/2} \right) + O( K^{-99} )$$
and hence
$$ \partial_t \left(E_* + K^{-28}\right)^{1/2} = \frac{1}{2} K a_1 \left(E_* + K^{-28}\right)^{1/2} + O(K^{-14})$$
for $t \in I$.
Starting with the crude bound $E_*(t_c+K^{-9}) \lesssim 1$ from \eqref{est-again}, \eqref{during-en-next'}, we conclude from Gronwall's inequality that
$$
\left(E_*(t)+K^{-28}\right)^{1/2} \lesssim \exp\left( - \frac{1}{2} K \int_{t_c+K^{-9}}^t a(t')\ dt'\right) + O( K^{-14} )$$
for any $t \in I$.  In particular, from \eqref{est-again} we have
\begin{equation}\label{eeyore}
\tilde E_0(t) \lesssim \exp\left( - K \int_{t_c+K^{-9}}^t a(t')\ dt'\right) + O( K^{-28} )
\end{equation}
for $t \in I$.

From \eqref{axin-eq-1a} we have
\begin{equation}\label{fallow}
\partial_t a_1 \geq O( K^{-1} |a_1| ) + O( K^{-20} )
\end{equation}
for $t \in I$.  In particular, if we can show
\begin{equation}\label{atc-again}
a_1(t_c+1/K) \geq 0.1
\end{equation}
then by Gronwall's inequality we will have
\begin{equation}\label{gb}
 a_1(t) \geq 0.05
\end{equation}
for all $t_c+1/K \leq t \leq \tau_1 = t_c+1/K^{1/2}$, and in particular from \eqref{eeyore} we have
\begin{equation}\label{samuel}
\tilde E_0(\tau_1) \lesssim K^{-28}
\end{equation}
giving \eqref{en1-init-next-3}.

We now show \eqref{atc-again}.   Suppose this is not the case.  From \eqref{axin-eq-1a}, \eqref{ak10}, \eqref{ab-refine} we have
$$ a_1(t_c + K^{-9}) \lesssim K^{-1} $$
so from \eqref{axin-eq-1a}, \eqref{after-en-next'}, and the failure of \eqref{atc-again} we have
$$ \int_{t_c+K^{-9}}^{t_c+1/K} K d_0(t)^2\ dt \leq 0.1 + O( K^{-1} )$$
and thus
$$ \int_{t_c+K^{-9}}^{t_c+1/K} d_0(t)^2\ dt \leq \frac{1}{10K} + O( K^{-2} ).$$
However, by repeating the derivation of \eqref{douse-again} we have
$$ \partial_t (a_0d_0 \frac{\eps^2}{c_0}) = - (a_0^2-d_0^2) + O(K^{-100})$$
on $I$,
and hence by the fundamental theorem of calculus and \eqref{c-large} we have
$$ \int_{t_c+K^{-9}}^{t_c+1/K} a_0(t)^2 - d_0(t)^2\ dt = O( K^{-100} )$$
and thus
\begin{equation}\label{jam}
\int_{t_c+K^{-9}}^{t_c+1/K} \frac{1}{2} (a_0^2+d_0^2)(t)\ dt \leq \frac{1}{10K} + O(K^{-2}).
\end{equation}
On the other hand, for $t \in [t_c+K^{-9},t_c+1/K]$ one has $a_1(t) \leq 0.1 + O(K^{-20})$ by \eqref{fallow}, the failure of \eqref{atc-again}, and Gronwall's inequality.  From \eqref{ab-refine}, \eqref{est-again} we conclude that
$$ a_0^2+d_0^2(t) \geq 0.99 + O(K^{-1}),$$
which contradicts \eqref{jam}.  This concludes the proof of \eqref{atc-again} and hence \eqref{en1-init-next-3}.

To finish up, we need to establish the bounds \eqref{an-init-next-3}-\eqref{spin-2a-next-3} (the bounds \eqref{lifespan-rescaled-3} coming from \eqref{cando} and construction of $\tau_1$).  From \eqref{ab-refine}, \eqref{samuel} we have
$$ a_1(t)^2 = 1 + O(K^{-1})$$
and \eqref{an-init-next-3} follows from this and \eqref{gb}.  The bounds \eqref{spin-1-next-3}, \eqref{spin-1a-next-3} follow from \eqref{bilboa} and \eqref{botany}, while the bounds \eqref{spin-2-next-3}, \eqref{spin-2a-next-3} follows from \eqref{c-bound-happy}, \eqref{solace-final}, \eqref{tau-tc}, and Gronwall's inequality.  This (finally!) completes the proof of Proposition \ref{reduced-claim-2}, and hence of Theorem \ref{main}.

\end{document}